\documentclass[english,11pt]{smfart}

\usepackage{etex}

\usepackage{amsbsy}
\usepackage{amsmath,amsfonts,amssymb,amsthm,mathrsfs,mathtools}

\usepackage{bm}

\usepackage[a4paper,vmargin={3cm,3cm},hmargin={3cm,3cm}]{geometry}
\usepackage[font=sf, labelfont={sf,bf}, margin=1cm]{caption}
\usepackage{graphicx}
\usepackage{epsfig}
\usepackage{latexsym}
\usepackage{xcolor}
\usepackage{ae,aecompl}
\usepackage{soul,framed}
\usepackage{comment}

\usepackage{xcolor}
\usepackage[pdfpagemode=UseNone,bookmarksopen=false,colorlinks=true,urlcolor=blue,citecolor=blue,citebordercolor=blue,linkcolor=blue]{hyperref}
\usepackage{smfhyperref}
\usepackage[capitalize]{cleveref}

\usepackage{pstricks}
\usepackage{enumerate}
\usepackage{tikz,animate,media9}						
\usepackage{todonotes}
\usepackage{pifont}
\usepackage{bm,marvosym}
\usepackage{algorithm}
\usepackage{algorithmic}

\usepackage{bbm}

\usepackage{tcolorbox}

\usepackage{natbib}  

\definecolor{aleacolor}{rgb}{0.16,0.59,0.78}

\hypersetup{
	breaklinks,
	colorlinks=true,
	linkcolor=aleacolor,
	urlcolor=aleacolor,
	citecolor=aleacolor}

\newcount\colveccount
\newcommand*\colvec[1]{
	\global\colveccount#1
	\begin{pmatrix}
		\colvecnext
	}
	\def\colvecnext#1{
		#1
		\global\advance\colveccount-1
		\ifnum\colveccount>0
		\\
		\expandafter\colvecnext
		\else
	\end{pmatrix}
	\fi
}

\newcommand{\ndN}{\mathbb{N}}
\newcommand{\ndZ}{\mathbb{Z}}

\newcommand{\ndR}{\mathbb{R}}

\newcommand{\ndB}{\mathbb{B}}

\renewcommand{\Pr}[1]{\mathbb{P}(#1)}

\newcommand{\Prb}[1]{\mathbb{P}\left(#1\right)}

\newcommand{\Ex}[1]{\mathbb{E}[#1]}

\newcommand{\Exb}[1]{\mathbb{E}\left[#1\right]}

\newcommand{\Va}[1]{\mathbb{V}[#1]}

\newcommand{\one}{{\mathbbm{1}}}

\newcommand{\convd}{\,{\buildrel \mathrm{d} \over \longrightarrow}\,}

\newcommand{\convp}{\,{\buildrel \mathrm{p} \over \longrightarrow}\,}

\newcommand{\convas}{\,{\buildrel \mathrm{a.s} \over \longrightarrow}\,}

\newcommand{\eqdist}{\,{\buildrel \mathrm{d} \over =}\,}

\newcommand{\atv}{\,{\buildrel \mathrm{d} \over \approx}\,}
\newcommand{\dtv}{d_{\mathrm{TV}}}

\newcommand{\He}{\mathrm{H}}

\newcommand{\Di}{\mathrm{D}}

\newcommand{\cA}{\mathcal{A}}
\newcommand{\cB}{\mathcal{B}}
\newcommand{\cC}{\mathcal{C}}

\newcommand{\cE}{\mathcal{E}}
\newcommand{\cF}{\mathcal{F}}

\newcommand{\cH}{\mathcal{H}}
\newcommand{\cI}{\mathcal{I}}
\newcommand{\cJ}{\mathcal{J}}
\newcommand{\cK}{\mathcal{K}}

\newcommand{\cM}{\mathcal{M}}

\newcommand{\cR}{\mathcal{R}}
\newcommand{\cS}{\mathcal{S}}

\newcommand{\cV}{\mathcal{V}}
\newcommand{\cW}{\mathcal{W}}

\newcommand{\cZ}{\mathcal{Z}}

\newcommand{\mD}{\mathsf{D}}

\newcommand{\mH}{\mathsf{H}}

\newcommand{\mM}{\mathsf{M}}

\newcommand{\mQ}{\mathsf{Q}}
\newcommand{\mR}{\mathsf{R}}
\newcommand{\mS}{\mathsf{S}}
\newcommand{\mT}{\mathsf{T}}

\newcommand{\mV}{\mathsf{V}}

\newcommand{\intervalle}[4]{\mathopen{#1}#2
	\mathclose{}\mathpunct{},#3
	\mathclose{#4}}
\newcommand{\intervalleff}[2]{\intervalle{[}{#1}{#2}{]}}

\newcommand{\enstq}[2]{\left\lbrace#1\mathrel{}\middle|\mathrel{}#2\right\rbrace}

\newcommand{\co}{\mathrm{c}}
\newcommand{\ve}{\mathrm{v}}
\newcommand{\ed}{\mathrm{e}}
\newcommand{\he}{\mathrm{h}}

\newtheorem{theorem}{Theorem}[section]

\newtheorem{corollary}[theorem]{Corollary}
\newtheorem{proposition}[theorem]{Proposition}
\newtheorem{lemma}[theorem]{Lemma}

\newtheorem{definition}[theorem]{Definition}

\numberwithin{equation}{section}

\title[Non-bijective scaling limits and phase transitions of planar maps]{\textbf{Non-bijective scaling limits and phase transitions of planar maps}}
\date{}

\author{Benedikt Stufler}

\address[Benedikt Stufler]{Vienna University of Technology}
\email{benedikt.stufler [at] tuwien.ac.at}

\begin{document}

\begin{abstract}
	We prove that the uniform random non-separable planar map with $n$ edges admits the Brownian sphere as Gromov--Hausdorff--Prokhorov scaling limit as $n$ tends to infinity. Our proof introduces a non-bijective ``common-core transfer method'' that constitutes a novel and universal proof strategy for scaling limits of random discrete structures. As an application, we complete the phase diagram for limiting shapes of block-weighted planar maps by Stufler~(2019). We describe phases with limits given by the Brownian sphere, stable trees, and Brownian sphere decorated stable trees recently introduced by S{\'e}nizergues, Stef{\'a}nsson and Stufler~(2023).
\end{abstract}


\maketitle

\section{Introduction and main results}

\subsection{Universality of the Brownian sphere}

Scaling limits of sequences of  random connected graphs are random measured metric spaces that describe their asymptotic geometric shape and mass distribution. Roughly speaking, they tell us what we see asymptotically when we look at the graphs from ``far away''.

Prominent examples of such limit objects are the Brownian continuum random tree, see \cite{MR1085326, MR1166406,MR1207226}, and the Brownian sphere, also called the Brownian map~\cite{MR2294979, MR3112934, MR3070569}. They are known to describe the asymptotic shape of different sequences of random graphs arising naturally in probability and combinatorics. This phenomenon is referred to as \emph{universality}. 

The universality of the Brownian sphere has been a subject of intense study, see for example~\cite{MR3256874, abraham2016, MR3706731, MR3729639, zbMATH06964619, zbMATH07787923}. Our first main result extends its universality class. We show that the uniform random non-separable planar map $\mV_n$ with $n$ edges, equipped with the  graph distance $d_{\mV_n}$ and the stationary distribution $\mu_{\mV_n}$, admits the Brownian sphere $(\mathbf{M}, d_{\mathbf{M}}, \mu_{\mathbf{M}})$ illustrated in Figure~\ref{fi:bm} as scaling limit with respect to the Gromov--Hausdorff--Prokhorov metric.

\begin{theorem}
	\label{te:main1}
	As $n$ tends to infinity,
	\begin{align*}
		\left(\mV_n, \left(\frac{3}{8n}\right)^{1/4} d_{\mV_n}, \mu_{\mV_n}\right) \convd (\mathbf{M}, d_{\mathbf{M}}, \mu_{\mathbf{M}}).
	\end{align*}
\end{theorem}

\begin{figure}[h]
	\centering
	\begin{minipage}{\textwidth}
		\centering
		\includegraphics[width=1\linewidth]{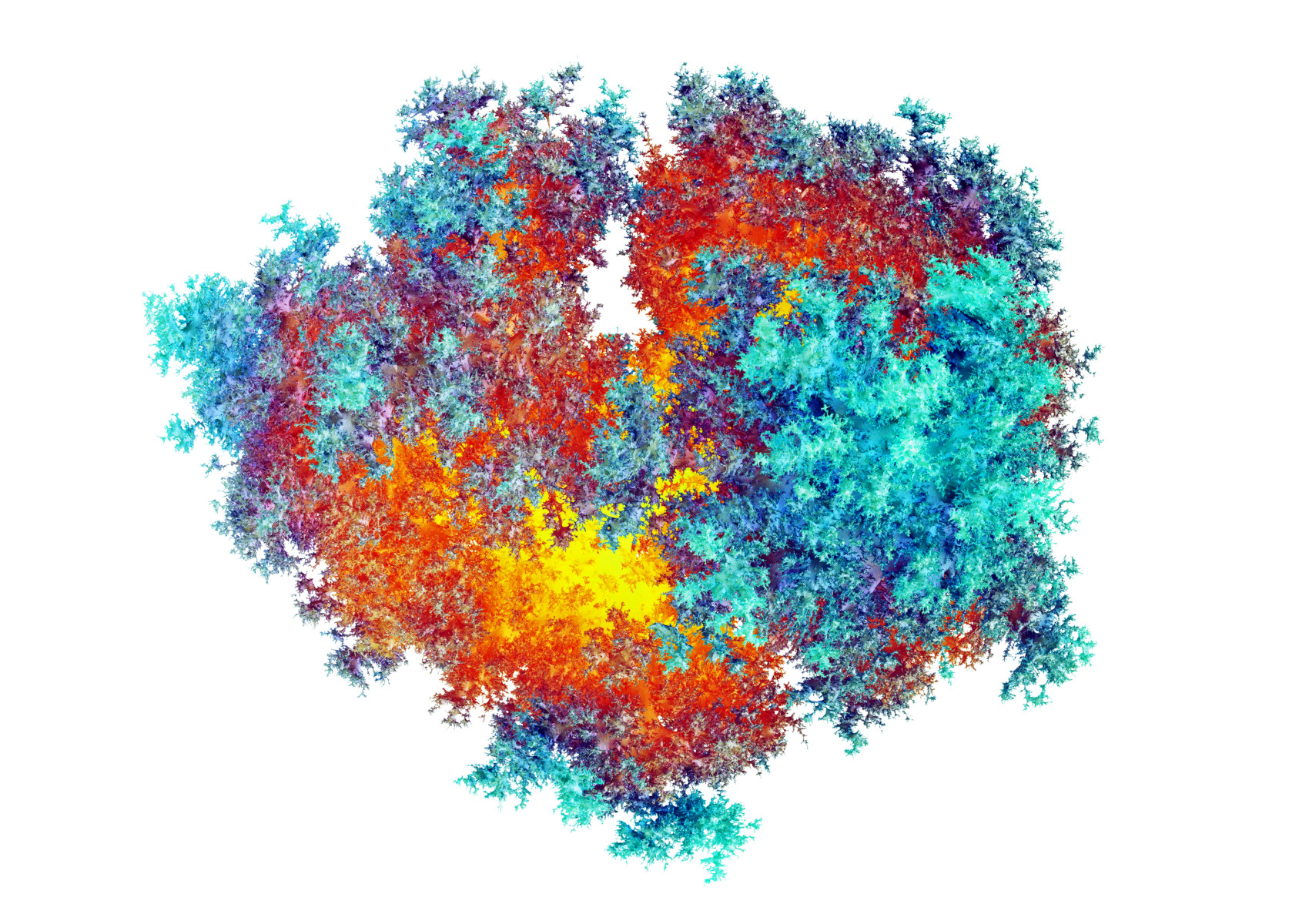}
		\caption[]{The Brownian sphere, approximated by a random simple quadrangulation  of the two-dimensional sphere with 64 million quadrangles.\footnotemark}
		\label{fi:bm}
	\end{minipage}
\end{figure}

The ``common-core transfer method'' we introduce to prove Theorem~\ref{te:main1} is a new and universal technique for proving scaling limits of random graphs, with potential applications to a wide range of maps and graphs with constraints. We provide a list of concrete settings below.

\footnotetext{The coordinates in the drawing are determined by a spring-electrical embedding method. Quadrangles are coloured according to their closeness centrality in the dual map.  The tools used for the simulation and visualisation are \texttt{simquad} (\url{https://github.com/BenediktStufler/simquad}), \texttt{scent} (\url{https://github.com/BenediktStufler/scent}), \texttt{Mathematica},  \texttt{Blender}, and \texttt{GIMP}.}

The classical approach for proving convergence for a model of random planar maps towards the Brownian sphere is to use a suitable combinatorial bijection to some class of trees endowed with a label process and show that these tree structures converge towards the Brownian snake. The downside to this approach is that combinatorial constraints on the models do not always translate to manageable constraints on the encoding trees. Further approaches include verifying that specific combinatorial transformations such as taking dual maps are  suitable for transferring convergence directly between models of random planar maps~\cite{zbMATH07144469}, and that cores of maps may serve as an approximation: Addario-Berry and Wen~\cite{MR3729639} established joint convergence of random quadrangulations and their non-separable cores and simple cores, conditional on the edge count of each object. For a quadrangulation with $n$ edges, these cores have sizes linear in $n$, with fluctuations of order $n^{2/3}$ governed by a $3/2$-stable limiting law. Conditional on its size, the simple core is a uniform simple quadrangulation. The scaling limit of this model was established in~\cite{MR3706731} via convergence of its blossoming-tree encoding to the Brownian snake. Addario-Berry and Wen then transferred this convergence from the simple model to the less restrictive  models.

In a certain sense, the proof of Theorem~\ref{te:main1} proceeds in the opposite direction, by passing convergence of the unrestricted planar map $\mM_n$ with $n$ edges established by Bettinelli, Jacob, and Miermont~\cite{MR3256874} to the non-separable case, instead of starting with unconstrained $3$-connected maps (for which scaling limits have not been obtained so far) and passing to the less restrictive non-separable case.  The arguments of Addario-Berry and Wen~\cite{MR3729639} may be adapted to show that the non-separable core $\cV(\mM_n)$ of $\mM_n$ is close to $\mM_n$ on a global scale. The challenge, however, is to deduce properties of the random non-separable planar map $\mV_n$ with precisely $n$ edges from the  core $\cV(\mM_n)$ whose number of edges fluctuates around $n/3$ at the characteristic order $n^{2/3}$ governed by a $3/2$-stable limiting law.

A first step towards a solution is given in the proof of~\cite[Thm. 1.3]{zbMATH07665039}, which showed local convergence of non-separable maps \[
\mV_n \convd \widehat{\mV}, \quad n \to \infty\]
in the annealed and quenched sense towards a limit $\widehat{\mV}$ named the uniform infinite $2$-connected planar map. The proof strategy solves precisely the described issue in the context of local convergence. First, the proof  extends the annealed local convergence of $\mM_n$ to quenched convergence and then transfers this quenched convergence from $\mM_n$ to the non-separable core $\cV(\mM_n)$. Next, it transfers local convergence from $\cV(\mM_n)$ to an even more constrained (essentially $3$-connected up to contracting paths to edges) $\cK$-core $\cK(\mM_n) := \cK(\cV(\mM_n))$. The same argument shows that in order to prove local convergence of $\mV_n$ it suffices to prove local convergence of the $\cK$-core $\cK(\mV_n)$. This reduces the initial task, that is, transferring convergence from $\cV(\mM_n)$ to $\mV_n$,  to transferring local convergence from  $\cK(\mM_n)$ to $\cK(\mV_n)$. Both cores $\cK(\mM_n)$ and $\cK(\mV_n)$  have random numbers of edges that concentrate at a linear fraction of $n$ with the characteristic $n^{2/3}$-fluctuations governed by different $3/2$-stable laws. With $\cV(\mM_n)$ having about $n/3$ edges, it follows that $\cK(\mM_{3n})$ has about the same linear number of edges as $\cK(\mV_n)$. The quotient of the densities for the limiting laws of the fluctuation is bounded away from zero and infinity on any compact interval. This results in a contiguity relation between $\cK(\mM_{3n})$ and $\cK(\mV_n)$. Roughly speaking, this means that  anything that is very likely for one model is also very likely for the other. Since quenched local convergence essentially reduces to laws of large numbers for the proportion of points with a given neighbourhood, it may be transferred from $\cK(\mM_{3n})$ to $\cK(\mV_n)$, hence completing the proof chain.

Thus, for notions of convergence that are preserved by contiguity such as quenched local limits, convergence may be passed from $\mM_n$ to $\mV_n$ by solving the general issue of passing from a randomly sized planar map $\cV(\mM_n)$ to an $n$-edge model $\mV_n$. In general, however, contiguity is not sufficient for transferring Gromov--Hausdorff--Prokhorov convergence. In this case, it only yields tightness. Hence we need to strengthen the approximation.  First, we simplify the approach by using simple non-separable cores $\cS(\mV_n)$ and $\cS(\mM_n) := \cS(\cV(\mM_n))$ instead of the more involved $\cK$-cores. Both cores $\cS(\mM_{3n})$ and $\cS(\mV_n)$ are mixtures of uniform simple non-separable maps. Their mixing distributions differ asymptotically only through the limiting $3/2$-stable laws governing the $n^{2/3}$ fluctuation around their common asymptotic mean which is linear in $n$. By the core approximation arguments introduced by Addario-Berry and Wen~\cite{MR3729639}, the Gromov--Hausdorff--Prokhorov scaling limit for $\mM_n$ yields a scaling limit for the core $\cS(\mM_{3n})$, and if we can transfer this to convergence to $\cS(\mV_n)$ we may deduce a scaling limit for $\mV_n$. For each constant $\theta \in \ndR$ we consider the perturbation $\cS(\mM_{3n + \lfloor n^{2/3} \theta \rfloor})$ of the mixing measure.  This changes the limiting $3/2$-stable law for the fluctuations of the number of edges, but $\cS(\mM_{3n + \lfloor n^{2/3} \theta \rfloor})$ still satisfies the same Gromov--Hausdorff--Prokhorov scaling limit as $\cS(\mM_{3n})$. We obtain a family of $3/2$-stable limiting laws, indexed by $\theta \in \ndR$, for the fluctuations of the number of edges. The linear span of the densities of these laws lies dense in $L^1(\ndR)$. This allows us to  approximate the mixing law for the number of edges of~$\cS(\mV_n)$ and pass Gromov--Hausdorff--Prokhorov convergence between the cores, hence completing the proof.

We call our proof strategy the \emph{common-core transfer method} because we pass from the source scaling limit of $\mM_n$  to the target scaling limit of $\mV_n$ through a transfer via common cores, that is, via $\cS(\mM_n)$ and $\cS(\mV_n)$. There is an abundance of different cores of all kinds of planar maps~\cite{MR1871555}, all exhibiting the characteristic linear number of edges with $n^{2/3}$-size Airy-type fluctuations. Thus, due to the well-established universality of Airy-type fluctuations for core sizes there are plenty of potential applications. Moreover, the choice of core, the $n^{2/3}$ fluctuation scale, the $3/2$-stable limiting distributions, and the fact that the scaling limit is the Brownian sphere are not essential to our approach. The key requirements are that the limiting distribution for the fluctuation has a density whose Fourier transform has no zeros, and that the cores function as Gromov--Hausdorff--Prokhorov approximations of the ambient structures. We may also work with the Gromov--Hausdorff distance instead if we are only interested in convergence with respect to this metric. 

 Here is a list of examples for potential applications to  scaling limits of various families of planar maps, for which suitable core decompositions are readily available.
\begin{itemize}
	\item \textbf{Simple triangulations and quadrangulations.} Recover the scaling limits by~\cite{MR3706731}. Use unrestricted triangulations and quadrangulations as source~\cite{MR3112934,MR3070569} and transfer through their irreducible cores.
	\item \textbf{Simple planar maps and loopless planar maps}.  Use planar maps as source~\cite{MR3256874} and transfer through simple non-separable cores. The required convergence of the core of the source is already treated in the present paper.
	\item \textbf{Bipartite non-separable maps and simple bipartite maps}. Use bipartite planar maps as source~\cite{abraham2016}. Transfer through the bipartite simple non-separable core. 
	\item \textbf{Simple non-separable planar maps.} Use Theorem~\ref{te:main1} as source and transfer through the $3$-connected core equipped with a first-passage-percolation distance instead of the graph distance. We do not require a direct comparison between the first-passage-percolation distance and the stretched graph distance. The arguments of~\cite{zbMATH07947540} that relate exchangeable link components with a first-passage-percolation distance suffice.
\end{itemize}

There is a plenitude of further potential applications, where suitable core decompositions might not be explicit in the literature, but may be obtained with manageable effort given the well-developed set of tools~\cite{zbMATH07004718,MR1871555,MR4132643,zbMATH07923679} for the study of core sizes and the proven universality of Airy-type phenomena in planar maps. Examples include simple $p$-angulations for $p \ge 5$, but also planar maps with a boundary as we describe below.

Let us highlight two related results.  Bettinelli, Curien, Fredes, and Sep\'ulveda~\cite{zbMATH08013060} developed a different non-bijective method for deducing a Brownian disc scaling limit of quadrangulations with a simple boundary from a scaling limit of quadrangulations with an unrestricted boundary. Roughly speaking, they approximate a restriction of the exact-size conditioned target model with a restriction of the randomly sized core. 

It is not clear whether the method of~\cite{zbMATH08013060} carries over to the present setting. At minimum, it would require genuinely new input on how to obtain suitable restrictions of non-separable planar maps. Conversely, our approach cannot readily be applied to the setting of~\cite{zbMATH08013060}. However, we have a clear strategy for making it work: We may view a planar map with a boundary as the result of taking an outerplanar map and inserting ``chordless'' planar maps with a simple boundary in its faces. That is, the components are planar maps with a simple boundary and no chords, that is, no non-boundary edges with both endpoints on the boundary. In particular, a planar map with a simple boundary is obtained from a dissection of a polygon by inserting chordless planar maps with a simple boundary in its faces. It stands to reason that in a general regime, exactly one of these components is macroscopic and yields a core candidate for transferring a scaling limit from the source model with an unrestricted boundary to the target model with a simple boundary. Specifically, such a phase exists for general face-weighted outerplanar maps as described in the phase diagram by Stef{\'a}nsson and Stufler~\cite{zbMATH07138334}, and face weights or face degree restrictions on a general planar map with a boundary precisely impose face-weights on the underlying outerplanar core. We aim to pursue this application in future work.

Another recent work by Fleurat~\cite{fleurat2025tauberianapproachmetricscaling} deduces a scaling limit of irreducible quadrangulations from the known scaling limit of unconstrained quadrangulations. The proof is based on the fact that irreducible quadrangulations of the hexagon admit a growth procedure established by Addario--Berry~\cite{zbMATH06349212}. This growth procedure essentially constructs a random irreducible quadrangulation with $n+1$ vertices from a random irreducible quadrangulation with $n$ vertices via a local modification. This allows Fleurat to control the change of geometry caused by the $n^{2/3}$-scale fluctuations of the number of vertices in the irreducible core of a source quadrangulation model.

The approach in~\cite{fleurat2025tauberianapproachmetricscaling} does not presently apply to the setting of Theorem~\ref{te:main1} since no suitable growth procedure for the target model of non-separable maps is known. Nevertheless, future work might establish such a procedure. Conversely, currently there is no known core of irreducible quadrangulations that would allow the application of the common-core transfer method, but an investigation in that direction might change that.  At present, growth procedures of planar map classes are known for irreducible quadrangulations of the hexagon~\cite{zbMATH06349212}, $2p$-angulations for $p \ge 2$, and simple triangulations~\cite{zbMATH07762022}. (See also~\cite{zbMATH05730489} for counterexamples of simply generated trees models that cannot admit growth procedures.) These constructions rely on model-specific bijective encodings via tree-like objects that admit compatible growth rules. In principle, these encodings might also facilitate the classical approach of proving convergence  of the underlying encoding tree with a suitable label process to the Brownian snake, but this would require a more technical analysis.

\subsection{Phase transition of block-weighted planar maps}

A block of a graph is a $2$-connected component, that is, a maximal $2$-connected subgraph. For planar maps, it is common to define a block as a maximal non-separable submap. The works~\cite{MR4132643,zbMATH06918043, zbMATH07431204} introduced block-weighted models of planar maps and graphs, which assume a given outcome with probability proportional to the product of weights assigned to its blocks. We consider a special case of such a model, where the weight of a block only depends on its size. Suppose that
\[
	f: \ndN_0 \to \ndR_{\ge 0}
\]
is a given function satisfying $f(0)=1$ and $f(k)>0$ for at least one $k \ge 1$. To each planar map $M$ we assign its weight
\begin{align}
	\label{eq:bw}
	\omega(M) = \prod_{B} f(\ed(B)),
\end{align}
with $B$ ranging over the non-separable components  of $M$. Set
\[
	d_\omega := \gcd\{i\ge 1\mid f(i)>0\}.
\]
This way, $\omega(M) = 0$ whenever the number $\ed(M)$ of edges satisfies $\ed(M) \not\equiv 0 \mod d_\omega$, and for each sufficiently large $n  \ge 1$ with $n \equiv 0 \mod d_\omega$ there exists at least one planar map with $n$ edges having positive weight. See~\cite{MR4132643} for details. For any such integer $n$ we let $\mM_n^\omega$ denote the random planar map with $n$ edges drawn with probability proportional to its weight. We set
\[
	\Phi(z) = 1 + \sum_{i \ge 1}  \frac{2(3i-3)!}{i!(2i-1)!} f(i) z^{2i}.
\]
Here, the factor $\frac{2(3i-3)!}{i!(2i-1)!}$ counts the number of non-separable planar maps with $i$ edges, all of which have the same weight by construction. Let $\rho_\Phi \in [0, \infty]$ denote the radius of convergence of $\Phi(z)$. We set $\nu:=0$ if $\rho_\Phi = 0$. If $\rho_\Phi >0$ then by~\cite[Sec. 3]{MR2908619} the limit
\[
	\nu := \lim_{t \uparrow \rho_\Phi} \frac{\Phi'(t) t}{\Phi(t)} \in ]0, \infty]
\]
exists. 

The constant $\nu$ is known to govern the asymptotic behaviour of $\mM_n^\omega$. Specifically, by~\cite[Thm. 6.57, Thm. 6.59]{MR4132643} and \cite[Thm. 1.3]{zbMATH07665039} there exists a random infinite planar map $\widehat{\mM}$ (which depends on the weights under consideration) such that in the local topology
\[
	\mM_n^\omega \convd \widehat{\mM}, \qquad n \to \infty, \qquad n \equiv 0 \mod d_\omega.
\]
See also prior work by Stephenson~\cite{MR3769811} who proved local convergence of regular critical Boltzmann planar maps, a class that partially overlaps with the present model. The phase diagram of block-weighted planar maps given in~\cite{MR4132643} distinguishes the following three phases defined in terms of $\nu$:

\textbf{Phase I ($\nu \ge 1$)}: the root of $\mM_n^\omega$ only observes small non-separable components. Consequently, $\widehat{\mM}$ consists of an infinite number of random finite non-separable maps glued together in a tree-like fashion. 

\textbf{Phase II ($0 <\nu < 1$)}: the root of  $\mM_n^\omega$ observes in its vicinity a large block as well as numerous smaller blocks. The limiting infinite planar map $\widehat{\mM}$ is assembled from the local limit $\widehat{\mV}$ of uniform non-separable maps and an infinite number of small non-separable planar maps in a way that is made explicit in~\cite{MR4132643}. 

\textbf{Phase~III ($\nu=0$)}: the root of $\mM_n^\omega$ observes in its vicinity only a single non-separable component. The limiting infinite planar map satisfies $\widehat{\mM} \eqdist \widehat{\mV}$ in this phase.

In order to  describe the behaviour of $\mM_n^\omega$ on a global scale in these phases we need to distinguish further regularity assumptions. We also need to define normalising sequences depending on these assumptions.

When $\nu>0$, let
\[
\tau := 
\begin{cases}
	\rho_\Phi,	& 0<\nu\le1,\\
	\text{the unique }t\in ]0,\rho_\Phi[ \text{ satisfying } t\Phi'(t)=\Phi(t), & \nu>1,
\end{cases}
\]
and define a random non-negative integer $\xi$ with probability generating function
\[
	\Ex{z^\xi} = \Phi(\tau z) / \Phi(\tau).
\]
This way, 
\[
	\Ex{\xi} = \frac{\tau\Phi'(\tau)}{\Phi(\tau)} = \min(1,\nu).
\]
If $\nu \ge 1$ we may define a random positive integer $\widehat\xi$ with size-biased distribution
\[
	\Prb{\widehat\xi=k} = k\Prb{\xi=k}, \qquad k\ge0.
\]
Conditionally on $\widehat\xi=2m$, let $\mV_{\widehat\xi}$ be a uniform non-separable planar map with $m$ edges. Let $d(\mV_{\widehat\xi})$ denote the distance between the origin of the root edge of $\mV_{\widehat\xi}$ and the vertex incident to a uniformly selected corner of $\mV_{\widehat\xi}$. We set
\[
	\kappa := \Exb{d(\mV_{\widehat\xi})} \in ]0, \infty].
\]
Conditionally on $(\mathbf M,d_{\mathbf M},\mu_{\mathbf M})$, let $U_1,U_2$ denote independent points with distribution $\mu_{\mathbf M}$ and define the deterministic constant
\[
	\mathfrak d_{\mathbf M} := \Exb{d_{\mathbf M}(U_1,U_2)}.
\]

We say that the finite-variance condition holds if
\begin{align}
	\label{eq:bwfvv}
	\nu > 1 \qquad\text{or}\qquad \left( \nu=1 \quad\text{and}\quad \Phi''(\tau)<\infty \right).
	\tag{$\mathrm{FV}$}
\end{align}
In this case, note that $\Va{\xi}=\tau^2\Phi''(\tau) / \Phi(\tau) < \infty$ and set $\alpha := 2$ and 
\[
	b_n := \sqrt{\Va{\xi} n /2}.
\]
For any $\alpha \ge 1$, we say that the corresponding regular-variation condition holds if there exists a slowly varying function $L$ such that
\begin{align}
	\label{eq:bwregvar}
	\Pr{\xi = 2kd_{\omega} } \sim L(k)k^{-1-\alpha}, \qquad k\to\infty.
	\tag{$\mathrm{RV}_\alpha$}
\end{align}
\begin{enumerate}[a)]
\item If $\nu=1$ and $1 < \alpha < 2$ we set
\begin{align*}
	b_n := \min \left\{ m \in \ndN \mid \Pr{\xi \ge m} n|\Gamma(1-\alpha)| \le  1 \right\}.
\end{align*}
This way, $b_n$ varies regularly with index $1/\alpha$.
\item If $\nu=1$ and $\alpha=2$ and $\Va{\xi}=\infty$, set
\[
	\upsilon(x) := \Exb{(\xi-\Ex{\xi})^2\one_{\{|\xi-\Ex{\xi}|\le x\}}}, \qquad x\ge0,
\]
and define
\begin{align*}
	b_n := \min \left\{m\in\ndN \mid n\upsilon(m)\le2m^2\right\}.
\end{align*}
Then $\upsilon$ is slowly varying, $b_n$ varies regularly with index $1/2$, and $ n\upsilon(b_n) / b_n^2 \to 2$. 
\item If $\nu=1$ and $\alpha=1$ set $X:=\xi-1$ and define
\begin{align*}
	a_n:=\min\left\{m\in\ndN\mid n\Prb{X\ge m}\le1\right\}, \qquad
	b_n:=n\Exb{X\one_{\{X>a_n\}}}.
\end{align*}
The sequences $(a_n)_{n\ge1}$ and $(b_n)_{n\ge1}$ vary regularly with index $1$. It holds that $ a_n=o(b_n)$ and $b_n = o(n)$.
\end{enumerate}

\noindent When $\nu=0$, we say that the complete-condensation condition holds if the truncations $\Phi_{\le k}(z) := \sum_{i=0}^{k}  z^i [x^i] \Phi(x)$ satisfy for each $0<\epsilon<1$
\begin{align}
	\label{eq:bwcc14}
	\frac{
		[z^{2n}]
		\left( \Phi_{\le  2n-\lfloor\epsilon n^{1/4}\rfloor}(z) \right)^{2n+1}
	}{
		[z^{2n}]\Phi(z)^{2n+1}
	}
	\to 0, \qquad n\to\infty, \qquad n \equiv 0 \mod d_\omega.
	\tag{$\mathrm{CC}_{1/4}$}
\end{align}
Sufficient conditions for this setting in terms of the coefficients of $\Phi(z)$ can be found in~\cite[Ex. 19.36]{MR2908619}.

We use the following notation for the limiting measured metric spaces. For $1<\alpha\le2$, let $(\mathbf{T}_\alpha, d_{\mathbf{T}_\alpha}, \mu_{\mathbf{T}_\alpha}) $ denote the $\alpha$-stable tree~\cite{MR1954248,MR1964956}. The case $\alpha=2$ corresponds to the celebrated Brownian continuum random tree constructed from $\sqrt{2}$ times a Brownian excursion of duration one, see~\cite{MR1085326,MR1166406,MR1207226}. For $1<\alpha<5/4$, let
$ \left(\mathbf{T}^{\mathrm{dec}}_\alpha(\mathbf{M}), d_{\mathbf{T}^{\mathrm{dec}}_\alpha(\mathbf{M})}, \mu_{\mathbf{T}^{\mathrm{dec}}_\alpha(\mathbf{M})}\right)$ denote the Brownian sphere decorated $\alpha$-stable tree, which is a special case of the stable decorated trees introduced by S{\'e}nizergues, Stef{\'a}nsson and Stufler~\cite{zbMATH07790315}. We refer to Section~\ref{sec:background} for the precise normalisations of all limiting objects.

\begin{theorem}
	\label{te:main2}
	All limits below are taken in the Gromov--Hausdorff--Prokhorov sense as $n\to\infty$, restricted to $n\equiv0 \mod d_\omega $.
	
	\begin{enumerate}
		\item Suppose that Condition~\eqref{eq:bwfvv} holds, or that $\nu=1$ and Condition~\eqref{eq:bwregvar} holds for some $\frac{5}{4} < \alpha \le 2$. Then $\kappa < \infty$ and
		\[
			\left( \mM_n^\omega, \frac{2^{1/\alpha-1}}{\kappa}\frac{b_{n}}{n} d_{\mM_n^\omega}, \mu_{\mM_n^\omega} \right) \convd (\mathbf{T}_\alpha, d_{\mathbf{T}_\alpha}, \mu_{\mathbf{T}_\alpha}).
		\]
		\item Suppose that $\nu=1$, that Condition~\eqref{eq:bwregvar} holds with $\alpha=5/4$, and that the limit  \[
			c_\omega := \lim_{k \to \infty} L(k) \in  ]0, \infty[
		\]
		exists. Then
		\[
		   \left( \mM_n^\omega, \frac{1}{\mathfrak d_{\mathbf{M}}} \left(\frac{6}{d_\omega}\right)^{1/4} \left(\frac{5\Gamma(3/4)^4}{c_\omega}\right)^{1/5} \frac{1}{n^{1/5}\log n} d_{\mM_n^\omega}, \mu_{\mM_n^\omega} \right)
		\convd \left( \mathbf T_{5/4}, d_{\mathbf T_{5/4}}, \mu_{\mathbf T_{5/4}} \right).
		\]
		\item Suppose that $\nu=1$ and that Condition~\eqref{eq:bwregvar} holds for some $1 < \alpha < \frac{5}{4}$.	 Then
		\[
			\left( \mM_n^\omega, \left( \frac{3}{2^{2 + 1/\alpha} b_{n}} \right)^{1/4} d_{\mM_n^\omega}, \mu_{\mM_n^\omega} \right) \convd \left(\mathbf{T}^{\mathrm{dec}}_\alpha(\mathbf{M}), d_{\mathbf{T}^{\mathrm{dec}}_\alpha(\mathbf{M})}, \mu_{\mathbf{T}^{\mathrm{dec}}_\alpha(\mathbf{M})}\right).
		\]
		\item Suppose that $\nu=1$ and that Condition~\eqref{eq:bwregvar} holds with $\alpha=1$. Then
		\[ 
			\left(\mM_n^\omega, \left(\frac{3}{8 b_{n}}\right)^{1/4} d_{\mM_n^\omega}, \mu_{\mM_n^\omega} \right) \convd (\mathbf{M}, d_{\mathbf{M}},\mu_{\mathbf{M}}).
		\]
		\item Suppose that either $0<\nu<1$ and Condition~\eqref{eq:bwregvar} holds for some $\alpha\ge1$, or $\nu=0$ and Condition~\eqref{eq:bwcc14} holds. Then
		\[
			\left( \mM_n^\omega, \left( \frac{3}{8(1-\nu)n} \right)^{1/4} d_{\mM_n^\omega}, \mu_{\mM_n^\omega} \right) \convd (\mathbf{M},d_{\mathbf{M}},\mu_{\mathbf{M}}).
		\]
	\end{enumerate}
\end{theorem}

Extremal sizes of blocks were described in~\cite[Sec. 6.6]{MR4132643}. The scaling limit for $\nu>1$ was already given in~\cite[Thm. 6.62]{MR4132643}. The argument is that the diameter of the non-separable components is of smaller order than the height of an underlying simply generated tree. This behaviour extends to the stable tree case for $\alpha > 5/4$. The $\alpha=5/4$ phase transition is precisely the boundary case left open by~\cite[Rem. 1.1]{zbMATH07790315} and requires a particularly delicate analysis. At this value, the height of the underlying tree has order $n^{1-1/\alpha} = n^{1/5}$ which coincides with the order $(n^{1/\alpha})^{1/4} = n^{1/5}$ of the diameter of a large non-separable component. The cumulative metric contributions of non-separable components along a path in the underlying tree require a logarithmic correction, whose normalisation depends on the precise asymptotic behaviour of the slowly varying factor in~\eqref{eq:bwregvar}. The decorated stable tree phase makes use of an invariance principle given in~\cite{zbMATH07790315}. Verifying its assumptions requires us to establish cumulative diameter bounds for non-separable maps and provide a version of~\cite[Prop.~3.2]{zbMATH07790315} with less strict requirements to ensure that  small-degree decorations in blow-ups of random trees are negligible. The $\alpha=1$ cases benefit from numerous condensation results on conditioned Bienaym\'e--Galton--Watson trees in the Cauchy domain~\cite{MR3914558,zbMATH08114289}. See also recent work on their height~\cite{MR4912670}. The Brownian sphere regime for $0\le \nu <1$ follows the strategy introduced by~\cite{MR3729639} and likewise benefits from condensation results of random trees~\cite{MR3335012,MR2908619,MR4144886}.

The simplest special case of  block-weighted planar maps is the case $f(k) = u$ for all $k \ge 1$ with $u>0$ a fixed constant. The corresponding planar map $\mM_n^u$ is biased by the number of blocks, hence we call it block-biased in order to distinguish it from the known more general block-weighted models~\cite{MR4132643,zbMATH06918043, zbMATH07431204}.

It is immediate from the definitions and the known generating series of non-separable planar maps, see Section~\ref{sec:enum},  that for this case we have 
\begin{align*}
	\Phi_u(z) = 1+ u\sum_{k\ge1} \frac{2(3k-3)!}{k!(2k-1)!}z^{2k},
\end{align*}
with radius of convergence $\sqrt{4/27}$.
Hence
\[
	\nu_u = \frac{8u}{3(u+3)} \begin{cases}>1,  &u>9/5 \\ =1,  &u=9/5 \\< 1,  &u<9/5. \end{cases}
\]
Thus, for $u>9/5$ the scaling limit of the block-biased map $\mM_n^u$ is the  Brownian tree  by Stufler~\cite[Thm. 6.62]{MR4132643}. For $u \le 9/5$ we have~\eqref{eq:bwregvar} holding for $\alpha_u= 3/2$ since
\begin{align*}
	[z^{2k}] \Phi_u(z) \sim u\frac{2}{9\sqrt{3\pi}} k^{-5/2} \left(\frac{27}{4}\right)^k, \qquad k \to \infty.
\end{align*}
See Section~\ref{sec:enum}. Thus, for $u < 9/5$ the block-biased map $\mM_n^u$ has a unique macroscopic block by Stufler~\cite[Sec. 6.6]{MR4132643}. By Theorem~\ref{te:main1} this block admits the Brownian sphere as scaling limit, yielding that $\mM_n^u$ also admits the Brownian sphere as scaling limit as in Theorem~\ref{te:main2}. For $u=9/5$ we obtain the $\alpha_u$-stable tree as scaling limit as in Theorem~\ref{te:main2}. The proof is essentially  the same as in~Stufler~\cite[Thm. 6.62]{MR4132643}, but with slightly different bounds when $\xi$ lacks finite exponential moments. The main point is that the diameter of large non-separable components has a smaller order of growth than the height of an underlying simply generated tree.  Fleurat and Salvy~\cite[Thm. 5.4]{zbMATH07829961} also gave a stable tree scaling limit for $u=9/5$ whose proof is likewise based on~\cite[Thm. 6.62]{MR4132643}. Additionally,  Fleurat and Salvy~\cite{zbMATH07829961}~established a phase transition for block-biased quadrangulations, building on~\cite{MR3706731,MR3729639} for the required input scaling limit of non-separable quadrangulations. Block-biased planar maps have also received recent attention from Duplantier and Guitter~\cite{duplantier2025liouvillequantumdualityrandom,duplantier2026liouvillequantumdualityrandom}, who draw parallels between the associated scaling exponents and predictions stemming from Liouville quantum gravity.

\begin{corollary}
	\label{co:ma}
	All limits below are taken in the Gromov--Hausdorff--Prokhorov sense as $n\to\infty$. For $u\ge 9/5$, set $r_u:=\sqrt{1-\frac1u}$ and
	\[
		\kappa_u := r_u	\sum_{m\ge1} \frac{(3m-3)!}{(m-1)!(2m-1)!} (r_u^2(1-r_u))^{m-1} \Exb{d(\mV_m)},
	\]
	where $d(\mV_m)$ denotes the distance between the origin of the root edge of $\mV_m$ and the vertex incident to a
	uniformly selected corner.
	\begin{enumerate}
		\item If $u>9/5$, then
		\[
			\left( \mM_n^u,	\frac{1}{2\kappa_u} \sqrt{\frac{r_u}{3r_u-2}}	n^{-1/2}d_{\mM_n^u},\mu_{\mM_n^u}	\right)	\convd (\mathbf{T}_2,d_{\mathbf{T}_2},\mu_{\mathbf{T}_2}).
		\]
		\item If $u=9/5$, then
		\[
			\left( \mM_n^{9/5},	\frac{2^{2/3}}{3\kappa_{9/5}} n^{-1/3}d_{\mM_n^{9/5}}, \mu_{\mM_n^{9/5}} \right) \convd \left( \mathbf{T}_{3/2}, d_{\mathbf{T}_{3/2}}, \mu_{\mathbf{T}_{3/2}} \right).
		\]
		\item If $0<u<9/5$, then
		\[
			\left( \mM_n^u, \left(\frac{9(u+3)}{8(9-5u)n}\right)^{1/4} d_{\mM_n^u}, \mu_{\mM_n^u}\right)\convd (\mathbf{M},d_{\mathbf{M}},\mu_{\mathbf{M}}).
		\]
	\end{enumerate}
\end{corollary}

Theorem~\ref{te:main2} and Corollary~\ref{co:ma} complete the phase diagram of limiting shapes by~\cite{MR4132643} in the case where the block weights depend only on the block size. Theorem~\ref{te:main1} supplies the required input at the non-separable level, since, under size-dependent block weights, every block is conditionally uniform given its size.

In its full generality, the framework of block-weighted planar maps and graphs developed in~\cite{MR4132643,zbMATH06918043, zbMATH07431204} permits arbitrary non-negative weights on non-separable maps and thus encompasses any constraint or decoration that factors over the maximal non-separable components.  It includes, in particular, uniform and block-weighted models of simple, bipartite, loopless, or bridgeless planar maps, as well as any intersection of these classes. Further examples include tree-rooted planar maps~\cite[Rem.~6.65]{MR4132643}, models in which the admissible block sizes are restricted to a prescribed set, and planar maps whose underlying graphs belong to block-stable classes defined by excluding $2$-connected minors.

Future local or global limits for weighted non-separable maps could therefore be combined with the transfer methods developed here to extend the existing picture beyond the local weak limits for $\nu\ge 1$ established in~\cite[Thm.~6.57]{MR4132643} and the scaling limits for $\nu>1$ established in~\cite[Thm.~6.62]{MR4132643}, potentially yielding rich phase diagrams for broad classes of block-weighted planar maps.

\subsection*{Notation}
We let $\ndN = \{1, 2, \ldots\}$ denote the set of positive integers and set $\ndN_0 = \ndN \cup \{0\}$.  All random variables occurring in this work are assumed to be defined on a common probability space whose measure we denote by $\mathbb{P}$. All unspecified limits are as $n \to \infty$. For two sequences $(X_n)_{n \ge 1}$ and $(Y_n)_{n \ge 1}$ of random variables with values in a common Polish space  we write $X_n \atv Y_n$ if their total variation distance $d_{\textsc{TV}}(X_n,Y_n)$ tends to zero. We say an event holds with high probability if its probability tends to $1$ as $n$ becomes large. Convergence in probability and distribution are denoted by $\convp$ and $\convd$. Almost sure convergence is denoted by $\convas$. For any sequence $a_n>0$ we let $o_p(a_n)$ denote an unspecified random variable $Z_n$ such that $Z_n / a_n \convp 0$. Likewise $O_p(a_n)$ is a random variable $Z_n$ such that $Z_n / a_n$ is stochastically bounded.

\section{Limiting spaces of planar maps}
\label{sec:background}

\paragraph*{Planar maps}
Planar maps are embeddings of connected multigraphs on the $2$-sphere that are viewed up to orientation-preserving homeomorphism. The number of such equivalence classes depends on the graph under consideration and may be equal to zero.

We consider planar maps that are rooted at a corner, or equivalently an oriented root edge. An exception is made only for the map consisting of a single vertex and no edges. This map has no corners to be rooted at, but we count it as rooted nevertheless.

Given a planar map $M$ we let  $\co(M)$, $\ed(M)$, and $\ve(M)$  denote its number of corners, edges, and vertices. We let $\He(M)$  denote its height with respect to the origin of the root edge and $\Di(M)$ its diameter. The height of a vertex $v$ is denoted by $\he_M(v)$.

\paragraph*{Gromov--Hausdorff--Prokhorov distance}
We let $d_{\mathrm{H}}$ denote the Hausdorff distance between compact subsets of a metric space, $d_{\mathrm{P}}$ the Prokhorov distance between Borel measures, $d_{\mathrm{GH}}$ the Gromov--Hausdorff distance between compact metric spaces, $d_{\mathrm{GP}}$ the Gromov--Prokhorov distance between measured Polish metric spaces, and $d_{\mathrm{GHP}}$  the Gromov--Hausdorff--Prokhorov distance. Pointed variants of these distances are denoted by adding a superscript $\bullet$, such as $d_{\mathrm{H}}^\bullet$.  We refer the reader to~\cite[Ch. 27]{zbMATH05306371}, \cite[Ch. 7]{MR1835418},  \cite[Sec. 6]{MR2571957}, and the recent survey~\cite{janson2020gromovprohorov} for a  detailed exposition of their definitions.

\paragraph*{Brownian sphere}

We recall the construction of the Brownian sphere following~\cite{MR3112934,MR3070569}. The coding pair $(\mathbf{e}, \mathbf{Z})$ of the  Brownian snake  is a random element of $\cC([0,1], \ndR^2)$. Here $\mathbf{e}$ denotes a Brownian excursion normalised to have duration one. Conditionally on $\mathbf{e}$, the random function $\mathbf{Z}$ is a continuous, centred Gaussian process with covariance \[
\mathrm{Cov}(\mathbf{Z}(s), \mathbf{Z}(t)) = \min \{ \mathbf{e}(x) \mid \min(s,t) \le x \le \max(s,t)\}, \qquad s,t \in [0,1].
\]
Let $d_{\mathbf{e}}$ denote the pre-metric on $[0,1]$ defined by
\[
	d_{\mathbf{e}}(s,t) = \mathbf{e}(s) + \mathbf{e}(t) - 2 \min \{ \mathbf{e}(x) \mid \min(s,t) \le x \le \max(s,t)\}, \qquad s,t \in [0,1].
\]
Similarly, let $d_{\mathbf{Z}}$ be defined by
\[
	d_{\mathbf{Z}}(s,t) = \mathbf{Z}(s) + \mathbf{Z}(t) - 2 \max\left(\min_{x \in [\min(s,t), \max(s,t)]} \mathbf{Z}(x), \min_{x \in [0,1] \, \setminus \, ]\min(s,t), \max(s,t)[ } \mathbf{Z}(x)     \right)
\]
for $s,t \in [0,1]$. Let $\mathbf{D} \in \cC([0,1]^2, \ndR)$ denote the largest pseudo-distance on $[0,1]$ that satisfies $\mathbf{D} \le d_{\mathbf{Z}}$ and $\{ d_{\mathbf{e}} = 0\} \subset \{ \mathbf{D} =0 \}$. The Brownian sphere $(\mathbf{M}, d_{\mathbf{M}})$ may be defined as the quotient metric space $[0,1]/\{\mathbf{D}=0\}$ corresponding to $\mathbf{D}$. The Borel probability measure $\mu_{\mathbf{M}}$ on  $(\mathbf{M}, d_{\mathbf{M}})$ is defined as the push-forward of the Lebesgue measure along the canonical surjection $[0,1] \to \mathbf{M}$.

\paragraph*{Scaling limit of unconstrained planar maps} By the main result of~\cite{MR3256874},
\begin{align}
	\label{eq:mapconverge}
	\left(\mM_n, (8n/9)^{-1/4} d_{\mM_n}, \mu_{\mM_n}\right) \convd (\mathbf{M}, d_{\mathbf{M}}, \mu_{\mathbf{M}})
\end{align}
holds in the Gromov--Hausdorff--Prokhorov sense. To be precise,~\cite{MR3256874} states only Gromov--Hausdorff convergence, but the proof given there also entails Gromov--Hausdorff--Prokhorov convergence: It is shown that there is an ordering  $c(0), c(1), \ldots, c(2n-1)$ of the corners of $\mM_n$ such that, with $c(2n) := c(0)$, a linear interpolation $\tilde{D}_n$ on $[0,1]^2$ of the function
\[
	\left\{\frac{0}{2n}, \frac{1}{2n}, \ldots, \frac{2n}{2n}\right\}^2 \to \ndR, \quad \left(\frac{i}{2n},\frac{j}{2n}\right) \mapsto (8n/9)^{-1/4}d_{\mM_n}(c(i), c(j))
\]
satisfies
\begin{align*}
	\tilde{D}_n \convd \mathbf{D}
\end{align*}
as random elements of $\cC([0,1]^2, \ndR)$. Here we let $d_{\mM_n}(c(i), c(j))$ denote the distance between the vertices incident to the corners $c(i)$ and $c(j)$. This yields the full Gromov--Hausdorff--Prokhorov convergence.

\paragraph*{Stable trees} Let $f: [0,1] \to [0, \infty[$ denote a continuous function satisfying $f(0) = f(1) = 0$ and $f(x) >0$ for $0<x<1$. We may define a pre-metric on $[0,1]$ by
\[
	d_{f}(x,y) = f(x) + f(y) - 2 \min_{\min(x,y) \le t \le \max(x,y)} f(t).
\]
The corresponding quotient metric space that identifies points having $d_f$-distance zero is compact and may be equipped with the push-forward measure of the Lebesgue measure. The result is a so-called measured real tree. 

Given $1<\alpha\le2$, let $X^{(\alpha)} = ( X_t^{(\alpha)})_{t \ge 0}$ denote the spectrally positive $\alpha$-stable L\'evy process with
\[
	\Exb{\exp\left(-\lambda X_t^{(\alpha)}\right)} = \exp\left(t\lambda^\alpha\right),
	\qquad
	t,\lambda \ge 0.
\]
Let $(X^{\mathrm{exc}, (\alpha)}_t, t \in [0,1])$ denote an excursion of $X^{(\alpha)}$ with duration one. For $0\le s \le t \le 1$ set
\begin{align*}
	I_s^t = \inf_{r \in [s,t]} X^{\text{exc},(\alpha)}_r.
\end{align*}
The associated height process is defined by
\[
	H_t^{\mathrm{exc}, (\alpha)} = \lim_{\epsilon \downarrow 0} \frac{1}{\epsilon} \int_{0}^t \one_{X^{\mathrm{exc}, (\alpha)}_s \,<\, I_s^t + \epsilon}\,\,\mathrm{d}s.
\]
The $\alpha$-stable tree $(\mathbf{T}_\alpha, d_{\mathbf{T}_\alpha}, \mu_{\mathbf{T}_\alpha})$ is defined to be the measured real tree constructed above for $f$ equal to a continuous modification of~$H_t^{\mathrm{exc}, (\alpha)}$, see~\cite{MR1954248,MR1964956}. For $\alpha=2$, we have $f \eqdist \sqrt{2} \mathbf{e}$, with $\mathbf{e}$ denoting a Brownian excursion of duration $1$. The tree $(\mathbf{T}_2, d_{\mathbf{T}_2}, \mu_{\mathbf{T}_2})$ corresponds to the Brownian tree~\cite{MR1085326,MR1166406,MR1207226}, which however often is normalised differently so that it is constructed either from $\mathbf{e}$ or $2\mathbf{e}$ instead.

\paragraph*{Brownian sphere decorated stable trees}

Decorated stable trees were introduced by~\cite{zbMATH07790315} and may be constructed from stable trees by blowing up branch points into rescaled independent copies of a decoration space. In the setting of Theorem~\ref{te:main2}, this decoration space is given by the Brownian sphere. We briefly recall the construction of~\cite{zbMATH07790315} and refer the reader to this source for more details. See also prior work on stable looptrees~\cite{MR3286462}.

Suppose throughout the following that $1<\alpha<5/4$. Let $\mathbf{p}$ denote the canonical projection from $[0,1]$ to $\mathbf{T}_\alpha$. Define a partial order $\preceq$ on $[0,1]$ as follows: for every $s,t\in[0,1]$ 
\begin{align*}
	s\preceq t \qquad \text{if} \qquad (\,\, s\le t \quad \text{and} \quad X^{\text{exc},(\alpha)}_{s^{-}} \le I_s^t \,\,).
\end{align*}
We write $s\prec t$ if $s\preceq t$ and $s\neq t$. For $s,t\in[0,1]$ with $s\preceq t$, define $x_s^t := I_s^t - X_{s^-}^{\text{exc}, (\alpha)}$, and $\Delta_s := X^{\text{exc},(\alpha)}_s-X^{\text{exc},(\alpha)}_{s^-}$, and  
\[
	u_s^t := \begin{cases} \frac{x_s^t}{\Delta_s},  &\Delta_s >0 \\ 0,  &\Delta_s = 0 \end{cases}.
\]
Here we use the convention $X_{0^-}^{\mathrm{exc},(\alpha)} := X_0^{\mathrm{exc},(\alpha)} =0$. The subset $\ndB=\enstq{v\in \intervalleff{0}{1}}{\Delta_v>0}$ of the compact unit interval corresponds to the set of branch points of $\mathbf{T}_\alpha$. Almost surely, these points in the tree all have an infinite degree, meaning that for any $v\in \ndB$, the space $\mathbf{T}_\alpha \setminus\{\mathbf{p}(v)\}$ has a countably infinite number of connected components. For any $v\in\ndB$, we let 
\begin{align*}
	\cA_v:= \enstq{u_v^t}{t\in\intervalleff{0}{1} \quad \text{and} \quad \exists s\in \intervalleff{0}{1}, v\prec s \prec t }.
\end{align*} 
This set $\cA_v$ is in one-to-one correspondence with the connected components of $\mathbf{T}_\alpha \setminus\{\mathbf{p}(v)\}$ above $\mathbf{p}(v)$, meaning not the one containing the root. 

Conditionally on $X^{\mathrm{exc},(\alpha)}$, let $\bigl((\mathbf{M}_v,d_v,\mu_v)\bigr)_{v\in\ndB}$
be independent copies of the Brownian sphere $(\mathbf{M},d_{\mathbf{M}},\mu_{\mathbf{M}})$, independent of $X^{\mathrm{exc},(\alpha)}$. Conditionally on the excursion and on these copies, sample, independently for all $v\in\ndB$, an ``inner root'' $\rho_v$ and a family of ``outer roots'' $ \bigl(Y_{v,a}\bigr)_{a\in\cA_v}$ such that the points $\rho_v$ and $(Y_{v,a})_{a\in\cA_v}$ are independent and each has distribution $\mu_v$. For $s\in[0,1]\setminus\ndB$, let $\mathbf{M}_s:=\{s\}$. For every $v\in\ndB$, equip $\mathbf{M}_v$ with the rescaled metric $\delta_v:=\Delta_v^{1/4}d_v$. Consider the disjoint union
\[
	\mathbf{T}^{\ast,\mathrm{dec}}_\alpha(\mathbf{M}) := \bigsqcup_{s\in[0,1]}\mathbf{M}_s.
\]
For $x\in\mathbf{T}^{\ast,\mathrm{dec}}_\alpha(\mathbf{M})$, let $\sigma(x)\in[0,1]$ denote the unique index such that $x\in\mathbf{M}_{\sigma(x)}$. For $v\in\ndB$ and $x\in\mathbf{T}^{\ast,\mathrm{dec}}_\alpha(\mathbf{M})$ satisfying $v\preceq\sigma(x)$, define
\[
	G_v(x) :=
\begin{cases}
	x, & \text{if }\sigma(x)=v,\\
	Y_{v,u_v^{\sigma(x)}}, & \text{if }v\prec\sigma(x) \text{ and }u_v^{\sigma(x)}\in\cA_v,\\
	\rho_v, &\text{if }v\prec\sigma(x) \text{ and } u_v^{\sigma(x)}\notin\cA_v.
\end{cases}
\]
For $x,y\in\mathbf{T}^{\ast,\mathrm{dec}}_\alpha(\mathbf{M})$, choose $x\wedge y\in[0,1]$ so that $\mathbf{p}(x\wedge y)$ is the most recent common ancestor of $\mathbf{p}(\sigma(x))$ and $\mathbf{p}(\sigma(y))$ in the genealogical order on $\mathbf{T}_\alpha$. We choose this representative so that
\[
	x\wedge y\preceq\sigma(x) \qquad\text{and} \qquad x\wedge y\preceq\sigma(y).
\]
When the most recent common ancestor is a branch point of $\mathbf{T}_\alpha$, $x\wedge y$ is chosen to be the unique element of $\ndB$ coding that branch point. For $r\in[0,1]$ and $x\in\mathbf{T}^{\ast,\mathrm{dec}}_\alpha(\mathbf{M})$ satisfying $r\preceq\sigma(x)$, define
\begin{align}
	\label{eq:d0sum}
	d_0(r,x) := \sum_{\substack{v\in\ndB\\ r\prec v\preceq\sigma(x)}} \delta_v\bigl(\rho_v,G_v(x)\bigr).
\end{align}
We verify below using \cite[Thm.~6.6]{zbMATH07790315} that the sum in~\eqref{eq:d0sum} is almost surely finite. For arbitrary $x,y\in\mathbf{T}^{\ast,\mathrm{dec}}_\alpha(\mathbf{M})$, set
\begin{align*}
	d_\alpha^{\mathrm{dec}}(x,y) :=  d_0(x\wedge y,x)+d_0(x\wedge y,y) +
	\begin{cases}
		\delta_{x\wedge y}
		\bigl(G_{x\wedge y}(x),G_{x\wedge y}(y)\bigr), & \text{if }x\wedge y\in\ndB,\\
		0, & \text{if }x\wedge y\notin\ndB.
	\end{cases}
\end{align*}
We declare $x \sim y$ whenever $d_\alpha^{\mathrm{dec}}(x,y)=0$ and define
\[
	\mathbf{T}^{\mathrm{dec}}_\alpha(\mathbf{M}) := \mathbf{T}^{\ast,\mathrm{dec}}_\alpha(\mathbf{M})/\sim.
\]
Let
\[
	\mathbf{p}^{\mathrm{dec}}: \mathbf{T}^{\ast,\mathrm{dec}}_\alpha(\mathbf{M}) \to \mathbf{T}^{\mathrm{dec}}_\alpha(\mathbf{M})
\]
denote the canonical projection, and let $d_{\mathbf{T}^{\mathrm{dec}}_\alpha(\mathbf{M})}$ be the induced metric. We may consider the metric space $(\mathbf{T}^{\mathrm{dec}}_\alpha(\mathbf{M}), d_{\mathbf{T}^{\mathrm{dec}}_\alpha(\mathbf{M})})$ as rooted at the point $\mathbf{p}^{\mathrm{dec}}(0)$. For each $s\in[0,1]\setminus\ndB$ let $\iota(s) \in  \mathbf{T}^{\ast,\mathrm{dec}}_\alpha(\mathbf{M})$ denote  the unique point of the singleton $\mathbf{M}_s$. Define the probability measure $\mu_{\mathbf{T}^{\mathrm{dec}}_\alpha(\mathbf{M})}$ on $\mathbf{T}^{\mathrm{dec}}_\alpha(\mathbf{M})$ as the push-forward of the Lebesgue measure $\mathrm{Leb}\big|_{[0,1]\setminus\ndB}$ along $\mathbf{p}^{\mathrm{dec}} \iota$. That is,
\[
	\mu_{\mathbf{T}^{\mathrm{dec}}_\alpha(\mathbf{M})} :=  \bigl(\mathrm{Leb}\big|_{[0,1]\setminus\ndB}\bigr) \circ  \bigl(\mathbf{p}^{\mathrm{dec}} \iota \bigr)^{-1}.
\]
Since $\ndB$ is almost surely countable, this is a probability measure. We write
\[
	\left( \mathbf{T}^{\mathrm{dec}}_\alpha(\mathbf{M}), d_{\mathbf{T}^{\mathrm{dec}}_\alpha(\mathbf{M})}, \mu_{\mathbf{T}^{\mathrm{dec}}_\alpha(\mathbf{M})} \right)
\]
for the resulting Brownian sphere decorated $\alpha$-stable tree. The diameter of the Brownian sphere has moments of every positive order. Indeed, it follows from its construction recalled above that
\begin{align*}
	\Di(\mathbf{M}) \le  2\left( \sup_{0\le t\le1}\mathbf{Z}(t) - \inf_{0\le t\le1}\mathbf{Z}(t) \right),
\end{align*}
and the random variable on the right-hand side has finite moments of every positive order, see~\cite[Cor. 3]{zbMATH02055268}. Using $1/4>\alpha-1$ it follows that~\cite[Thm.~6.6]{zbMATH07790315} applies, ensuring that the sums in \eqref{eq:d0sum} are almost surely finite and that $\left( \mathbf{T}^{\mathrm{dec}}_\alpha(\mathbf{M}), d_{\mathbf{T}^{\mathrm{dec}}_\alpha(\mathbf{M})}, \mu_{\mathbf{T}^{\mathrm{dec}}_\alpha(\mathbf{M})} \right)$ is almost surely a compact measured metric space. We refer to it as the Brownian sphere decorated $\alpha$-stable tree.

\section{The non-separable core of planar maps}

\label{sec:nscorepm}

A planar map is called \emph{non-separable} (or \emph{$2$-connected}), if its edge-set may not be partitioned into two disjoint subsets $E_1$ and $E_2$ such that there is precisely one vertex that is incident with both a member of $E_1$ and a member of $E_2$.  Thus the only non-separable planar map containing a loop is the map consisting of a single vertex with a loop. The map consisting of a single vertex with no edges is also non-separable. The well-known block-decomposition decomposes any multigraph and hence also any planar map $M$ into its non-separable components, the (typically unique) largest of which is called the non-separable core.  We denote the non-separable core of $M$ by $\cV(M)$. If there are several blocks with a maximal number of edges, we make a canonical choice for $\cV(M)$.

The non-separable core $\cV(\mM_n)$ of the uniform random planar map $\mM_n$ with $n$ edges has a linear number of edges with fluctuations of order $n^{2/3}$ quantified by a $3/2$-stable law. The following local limit theorem was shown by~\cite[Thms. 1--3]{MR1871555} using analytic methods. Alternatively, it may be deduced in a probabilistic way as~\cite{zbMATH07004718} showed that the number of corners of $\cV(\mM_n)$ corresponds to the maximal offspring in a specific conditioned subcritical branching process, and there is a local limit theorem for the maximal offspring for size-constrained subcritical branching processes~\cite{MR4144886}.

\begin{proposition}
	\label{pro:lltvcore}
Let $Z_{3/2}$ denote the $3/2$-stable random variable whose Laplace transform satisfies $\Ex{\exp(-\lambda Z_{3/2})} = \exp(\lambda^{3/2})$. Let $h$ denote its density function. With
$
	g_{\cM} = \left(\frac{16}{243}\right)^{1/3}
$ we have
\begin{align*}
\Pr{ \ed(\cV(\mM_n)) =  \ell} = \frac{1}{g_\cM n^{2/3}}\left(h\left(\frac{ n/3 - \ell}{g_\cM n^{2/3}}   \right) + o(1)\right)
\end{align*}
uniformly for all integers $\ell \ge 0$. Consequently, 
\begin{align*}
	\frac{n/3 - \ed(\cV(\mM_n))}{g_\cM n^{2/3}} \convd Z_{3/2}.
\end{align*}
Moreover, for any sufficiently small $\epsilon>0$ and any $0 < \delta < 1/3$ we have
\[
	 \Prb{\ed(\cV(\mM_n)) \le \epsilon n / \log(n)} = o(n^{-2/3})
\]
and
\[
	\Prb{\ed(\cV(\mM_n)) \le \delta n} = O(n^{-2/3} ).
\]
\end{proposition}
\begin{proof}
	The local limit theorem follows directly from~\cite[Thm. 1.1]{MR4144886}. It also implies the central limit theorem. The bound for $\Prb{\ed(\cV(\mM_n)) \le \epsilon n / \log(n)}$ follows from~\cite[Eq. (2.10)]{MR4144886}. The bound for $\Prb{\ed(\cV(\mM_n)) \le \delta n}$ follows from~\cite[Eq. (2.11), (2.12)]{MR4144886}. 
\end{proof}

We are going to formalise a statement on the  Gromov--Hausdorff--Prokhorov distance  between $\mM_n$ and its non-separable core $\cV(\mM_n)$. Throughout the rest of this paper we will use the notation $\mu_{(\cdot)}$ to denote the stationary distribution on some finite connected graph. That is, a vertex gets drawn according to this distribution with probability proportional to its degree. (Unless the graph consists of only a single vertex, in which case the stationary distribution assigns mass $1$ to that vertex.) For planar maps, this corresponds to the vertex incident to a uniformly selected corner. In particular,  $\mu_{\mM_n}$ denotes  the stationary distribution on the vertex set of $\mM_n$, and  $\mu_{\cV(\mM_n)}$ denotes the stationary distribution on the vertex set of $\cV(\mM_n)$. 

The map $\mM_n$ consists of its non-separable core $\cV(\mM_n)$ with maps $\cM_i(\mM_n)$, $1 \le i \le \co(\cV(\mM_n))$, attached to its corners. 
See Figure~\ref{fi:block} for an illustration.  Here we order the components in a canonical way  according to a breadth-first-search of the corners of the core $\cV(\mM_n)$ that starts at its root edge. This way, $\cM_i(\mM_n)$ gets attached to the $i$th corner of $\cV(\mM_n)$ for all $1 \le i \le \co(\cV(\mM_n))$, and $\cM_1(\mM_n)$ is the unique component \emph{corresponding} to the root edge of $\mM_n$. That is, if the root half-edge of $\mM_n$ is not contained in $\cV(\mM_n)$ then $\cM_1(\mM_n)$ is the unique component containing the root edge. If the root half-edge is contained in $\cV(\mM_n)$ then $\cM_1(\mM_n)$  is the unique component attached to the root corner / half-edge. Note that a component may consist of a single vertex with no edges, if nothing is to be inserted into the corresponding corner of $\cV(\mM_n)$. Furthermore, unless it is empty, $\cM_1(\mM_n)$ contains two root corners, a first and a second. The first marks the place where it is attached to the core $\cV(\mM_n)$. The second corresponds to the root of $\mM_n$ and may either be a corner of $\mM_n$ or equal to a placeholder value (corresponding to the case where the root half-edge of $\mM_n$ is contained in $\cV(\mM_n)$). This way, $\mM_n$ is fully described by the core $\cV(\mM_n)$ together with the components $(\cM_i(\mM_n))_{1 \le i \le \co(\cV(\mM_n))}$.

	Let $\nu_{\cV(\mM_n)}$ denote the probability measure on $\cV(\mM_n)$ such that a vertex $v \in \cV(\mM_n)$ incident to components $(\cM_i(\mM_n))_{i \in I(v)}$ with $I(v) \subset \{1, \ldots, \co(\cV(\mM_n))\}$  gets drawn with probability 
	\[
		\nu_{\cV(\mM_n)}(\{v\}) = 	\frac{1}{2n}\sum_{i \in I(v)} (1 + \co(\cM_i(\mM_n))).
	\]
	We emphasise that  $\nu_{\cV(\mM_n)}$ actually depends on the entire map $\mM_n$, whereas the stationary distribution $\mu_{\cV(\mM_n)}$ only depends on the non-separable core $\cV(\mM_n)$.

\begin{lemma}
	\label{le:2corepart1}
 As $n \to \infty$, 
	\begin{align*}
		\left(\cV(\mM_n), (8n/9)^{-1/4} d_{\cV(\mM_n)}, \nu_{\cV(\mM_n)}\right)  \convd (\mathbf{M}, d_{\mathbf{M}}, \mu_{\mathbf{M}})
	\end{align*}
	in the Gromov--Hausdorff--Prokhorov sense.
\end{lemma}

\begin{proof}
	
	\begin{figure}[t]
		\centering
		\begin{minipage}{\textwidth}
			\centering
			\includegraphics[width=1.0\linewidth]{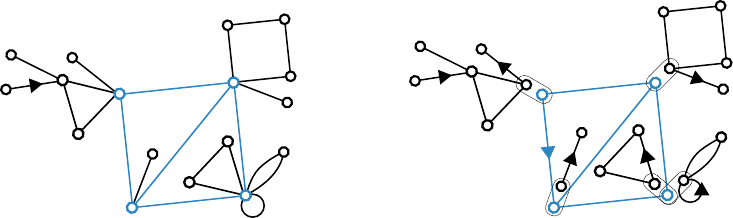}
			\caption{Decomposing a planar map into its non-separable core and the attached submaps. The components consisting of a single vertex with no edges are omitted in the illustration.}
			\label{fi:block}
		\end{minipage}
	\end{figure}

	We let 
	\begin{align}
		\cM(z) := \sum_{M} z^{\co(M)} = 1 + 2z^2 + \ldots
	\end{align}
	denote the generating series of planar maps indexed by their number of corners, with the sum index $M$ ranging over all (finite) planar maps. It is well-known by classical work of Tutte~\cite{MR0146823} that its radius of convergence equals $\rho_\cM = 1 / (2 \sqrt{3})$ and that the coefficients $[z^{2n}]\cM(z \rho_\cM)$ vary regularly with index $-5/2$. We let $\mM$ denote a random planar map that assumes a planar map $M$ with probability 
	\begin{align}
		\Pr{\mM= M} = \frac{\rho_\cM^{\co(M)}}{\cM(\rho_\cM)}.
	\end{align}
	We let $\mM^\circ$ denote the random finite planar map obtained by biasing $\mM$ on $1 + \co(\mM)$. That is, $\mM^\circ$ has a second marked corner that is either a corner of itself or a placeholder value. If $M^\circ$ is a finite planar map with a second marked corner that is either a corner of itself or a placeholder value, then
	\begin{align}
			\Pr{\mM^\circ= M^\circ} = \frac{\rho_\cM^{\co(M)}}{\cM(\rho_\cM) + \rho_\cM \cM'(\rho_\cM)}.
	\end{align}
	Let $(\mM(i))_{i \ge 2}$ denote independent copies of $\mM$ and set $\mM(1) = \mM^\circ$. We set
	\[
	\Delta_{\langle n \rangle}^\cM :=  \sup  \left(\left \{ d \ge 1 \,\,\bigg\rvert\,\, \co(\mM(1)) +  \sum_{i=2}^d (1+ \co(\mM(i))) \le 2n  \right\} \cup \{0\}\right).
	\]
	It follows from~\cite[Eq. (9.6) and (9.7)]{zbMATH07665039}  that for any sequence of integers $(t_n)_n$ with $t_n\to \infty$ and $t_n=o(n)$
	\begin{align}
		\label{eq:s1}
		\lim_{n \to \infty} \dtv\left( (\cM_i(\mM_n))_{1 \le i \le \co(\cV(\mM_n)) - t_n}, (\mM(i))_{1 \le i \le \Delta_{\langle n \rangle}^\cM - t_n} \right) = 0,
	\end{align}
	and with a probability that tends to $1$ as $n \to \infty$
	\begin{align}
		\label{eq:s2}
		\sum_{i=\co(\cV(\mM_n))-t_n+1}^{\co(\cV(\mM_n))} \co(\cM_i(\mM_n)) \le 2 \Ex{\co(\mM)} t_n.
	\end{align}

	In other words, all but a negligible number of the $\co(\cV(\mM_n)) =  2n/3 + O_p(n^{2/3})$ components asymptotically behave like a sequence of copies of $\mM$ that stops once we have accumulated sufficient mass. 
	
	The deviation bound in \cite[Thm. 3.5]{MR3318042} implies (using the tail asymptotics for $\ed(\mM)$) that the graph distance diameter $\Di(\mM)$ satisfies 
	\begin{align}
	\Ex{\Di(\mM)^s} < \infty
	\end{align}
	 for any constant $0<s<6$.  Hence, by Equations~\eqref{eq:s1} and~\eqref{eq:s2}  it follows that for any $\delta > 1/6$ it holds with probability tending to $1$ as $n$ becomes large that
	\begin{align}
	\label{eq:smallattachment}
	\max_{1 \le i \le \co(\cV(\mM_n))} \Di(\cM_i(\mM_n)) \le n^\delta.
	\end{align}
	Together with Equation~\eqref{eq:s1} it follows that  the pointed Hausdorff distance $d_{\mathrm{H}}^\bullet$ between $\mM_n$ (pointed at its root corner) and $\cV(\mM_n)$ (pointed at its root corner) converges in probability to zero when edges are rescaled to have a length of order $n^{-1/4}$. That is, for any $c>0$
	\begin{align}
		\label{eq:hm}
		d_{\mathrm{H}}^\bullet \left( (\mM_n, c n^{-1/4} d_{\mM_n}), (\cV(\mM_n), cn^{-1/4} d_{\cV(\mM_n)}) \right) \convp 0.
	\end{align}
If we uniformly draw a corner $c_n$ of $\mM_n$, then its incident vertex $v_n$ is distributed according to $\mu_{\mM_n}$ (conditionally on $\mM_n$). There is a unique index $i(c_n) \in \{1, \ldots, \co(\cV(\mM_n))\}$ such that either $c_n$ is a corner of $\cM_{i(c_n)}(\mM_n)$ or $c_n$ is the  corner of $\cV(\mM_n)$ where $\cM_{i(c_n)}(\mM_n)$ is attached.  The vertex $u_n$ of $\cV(\mM_n)$ incident to the corner where $\cM_{i(c_n)}(\mM_n)$ is attached to $\cV(\mM_n)$ is distributed according to $\nu_{\cV(\mM_n)}$ (conditionally on $\mM_n$). The graph distance between $v_n$ and $u_n$ is bounded by the diameter of $\cM_{i(c_n)}(\mM_n)$. Hence, by~\eqref{eq:smallattachment}, 
	\begin{align}
		n^{-1/4} d_{\mM_n}(v_n, u_n) \convp 0.
	\end{align}
	Using the characterisation of the Gromov--Hausdorff--Prokhorov distance from~\cite[Prop. 6]{MR2571957}, it follows that
	\begin{align}
		d_{\mathrm{GHP}} \left( (\mM_n, c n^{-1/4} d_{\mM_n}, \mu_{\mM_n}), (\cV(\mM_n), cn^{-1/4} d_{\cV(\mM_n)}, \nu_{\cV(\mM_n)}) \right) \convp 0.
	\end{align}
	Setting $c=(8/9)^{-1/4}$ and using~\eqref{eq:mapconverge} it follows that
		\begin{align*}
		\left(\cV(\mM_n), (8n/9)^{-1/4} d_{\cV(\mM_n)}, \nu_{\cV(\mM_n)}\right)  \convd (\mathbf{M}, d_{\mathbf{M}}, \mu_{\mathbf{M}}).
	\end{align*}
\end{proof}
The analogous statement of Lemma~\ref{le:2corepart1} for the smaller Gromov--Hausdorff distance is  implicit in  \cite[Proof of Thm. 3.7]{MR3318042}. 

The next result is proved by using Lemma~\ref{le:2corepart1} and adapting the exchangeable-decoration argument of~\cite[Lem. 5.3 and Cor. 6.2]{MR3729639}. We already carry out a very similar adaptation in full detail in the proof of Lemma~\ref{le:sourcesimplecorepart2} below. Since the proof of Lemma~\ref{le:2corepart2} would be almost identical, we refer the reader to it for details.

\begin{lemma}
	\label{le:2corepart2}
	The Prokhorov distance of the two measures $\nu_{\cV(\mM_n)}$ and $\mu_{\cV(\mM_n)}$ on the rescaled non-separable core $\left(\cV(\mM_n), (8n/9)^{-1/4} d_{\cV(\mM_n)}\right)$ satisfies
	\begin{align*}
		d_{\mathrm{P}}( \nu_{\cV(\mM_n)}, \mu_{\cV(\mM_n)} ) \convp 0.
	\end{align*}
\end{lemma}

Combining Lemma~\ref{le:2corepart1} and Lemma~\ref{le:2corepart2} we obtain Gromov--Hausdorff--Prokhorov convergence of the non-separable core $\cV(\mM_n)$ equipped with the stationary distribution. 

\begin{corollary}
	\label{eq:corfppvcore}
	As $n \to \infty$,
\begin{align*}
	\left(\cV(\mM_n), (8n/9)^{-1/4} d_{\cV(\mM_n)}, \mu_{\cV(\mM_n)}\right)  \convd (\mathbf{M}, d_{\mathbf{M}}, \mu_{\mathbf{M}}).
\end{align*}
\end{corollary}

\section{The simple core of non-separable planar maps}
\label{sec:simplecore}

\subsection{Simple components}

\label{sec:simplecomponents}

A planar map is called \emph{simple} if it has no loops and no multi-edges. We recall the decomposition into simple components.  Let $M$ be a rooted non-separable map which is not the loop map, and let $e_0$ be its distinguished oriented edge. Suppose that $M$ is drawn in the plane so that its unique unbounded face lies to the right of $e_0$. A $2$-cycle $\sigma$ in $M$ is a pair of distinct parallel edges with common endpoints, say $v$ and $w$.  The union of these two edges is a simple closed curve that separates two parts of $M$. We refer to the \emph{inner part} as the one that does not contain the unbounded face, the other one is called the \emph{outer part}. We say the $2$-cycle $\sigma$ is maximal if its outer part does not contain any other edge between $v$ and $w$. If $\sigma$ is maximal, we may \emph{contract} it to a single edge $e$, meaning we delete its inner part and one of its cycle edges. It does not matter which of the cycle edges we delete, but if one of them is the oriented root edge $e_0$ then we delete the other one, so that the resulting map still has an oriented root edge. Let $e$ denote the edge resulting from contracting $\sigma$. 

	\begin{figure}[t]
	\centering
	\begin{minipage}{\textwidth}
		\centering
		\includegraphics[width=1.0\linewidth]{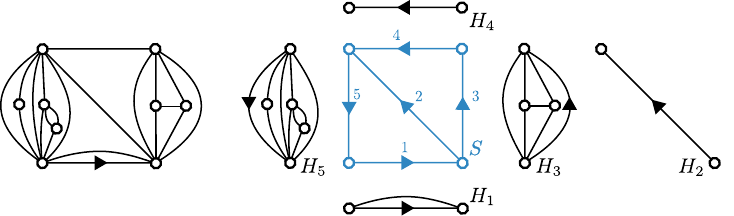}
		\caption{Decomposing a non-separable planar map into a simple component and the associated $\cH$-structures.}
		\label{fi:simple}
	\end{minipage}
\end{figure}

If we contract all maximal $2$-cycles of the non-separable planar map $M$ then we are left with its \emph{simple component} $S$ at the root edge. We may canonically orient all edges of $S$, and for each (canonically oriented) edge $e$ of $S$ we distinguish a planar map $H_e$. Let $v$ and $w$ denote the start and end of $e$. If $e$ did not result from contracting a maximal $2$-cycle of $M$, we let  $H_e$ be a planar map consisting of two vertices joined by an oriented root edge. If $e$ did result from contracting a maximal cycle $\sigma_e$ of $M$ we let $H_e$ be the planar map consisting of $\sigma_e$ and its inner part. We orient a boundary edge of $H_e$ from $v$ to $w$ in such a way that the inner part of $\sigma_e$ lies to its left. This way, $M$ may be recovered from $S$ and the family of \emph{$\cH$-components} $(H_e)_e$, with $e$ ranging over the edges of $S$. See Figure~\ref{fi:simple} for an illustration.

Let
$
\cV^*(z)=\sum_V z^{\ed(V)}
$
denote the ordinary generating series of rooted non-separable planar maps, with the exception that maps with only one vertex are not counted. Let
$
\cS(z)=\sum_S z^{\ed(S)}
$
denote the ordinary generating series of rooted simple non-separable planar maps.  We count the map consisting of two vertices joined by a single edge as both non-separable and simple. If we delete the non-root boundary edge of a non-trivial $\cH$-component we are left with a non-separable map with at least two vertices and no further constraints. Letting $\cH(z):=\sum_H z^{\ed(H)}$ denote the ordinary generating series for $\cH$-objects, we obtain
\begin{align}
	\label{eq:hcore}
	\cV^*(z)=\cS(\cH(z)),
	\qquad
	\cH(z) = z(1 + \cV^*(z)).
\end{align}
See~\cite[Sec. 5.1 and Tab. 3]{MR1871555}.

\subsection{Enumeration of simple non-separable planar maps}
\label{sec:enum}
Let 
\[
	u(z)  = z+2 z^2+7 z^3+30 z^4+143 z^5 + \ldots
\]
denote the solution of the equation
\[
	z = u(1-u)^2
\]
satisfying $u(0) = 0$. Existence and uniqueness of $u(z)$ for $z$ close to the origin are guaranteed by the implicit function theorem. Since $u(z) = z (1- u(z))^{-2}$,  Lagrange inversion yields for all $1 \le k \le n$
\begin{align}
	\label{eq:ucoeff}
	[z^n] u(z)^k = \frac{k}{n} [z^{n-k}](1-z)^{-2n} = \frac{k}{n}\binom{3n-k-1}{n-k}.
\end{align}
It follows by Stirling's formula that
\begin{align}
	\label{eq:uasymptotic}
	[z^n]u(z)  \sim \frac{1}{3\sqrt{3\pi}} n^{-3/2} \left(\frac{27}{4}\right)^n.
\end{align}
It was shown by~\cite{MR0146823} (with a slight adaptation since we do not count the loop map) that 
\begin{align}
	\label{eq:btp}
	\cV^*(z)=u(z)(1-u(z)-u(z)^2) = z + z^2 + 2 z^3 + 6 z^4+22 z^5 + \ldots.
\end{align}
Using~\eqref{eq:ucoeff} it follows that
\begin{align}
	\label{eq:nonsepenum}
	[z^n]\cV^*(z) = \begin{cases}
		1, & n=1,\\
		\frac{2(3n-3)!}{n!(2n-1)!}, 
		& n\ge 2.
	\end{cases}
\end{align}
By Stirling's approximation formula it follows that
\begin{align}
	\label{eq:vgrowth}
	[z^n]\cV^*(z) \sim \frac{2}{9\sqrt{3\pi}} n^{-5/2} \left(\frac{27}{4}\right)^n.
\end{align}
Consequently,
\begin{align}
	\label{eq:hgrowth}
	[z^n]\cH(z) &\sim \frac{8}{243\sqrt{3\pi}} n^{-5/2} \left(\frac{27}{4}\right)^n.
\end{align}
It follows from \eqref{eq:hcore}, \eqref{eq:btp} and $z = u(z)(1-u(z))^2$ that
\begin{align}
	\label{eq:hseries}
	\cH(z) 	&= z(1 + \cV^*(z)) \\
			&= u(z)(1-u(z))^2(1 + u(z)(1-u(z)-u(z)^2)) \nonumber \\
			&= u(z)(1-u(z))(1 - u(z)^2)^2. \nonumber
\end{align}
Since $\cH(0)=0$ and $\cH'(0)=1$, the formal power series $\cH(z)$ has a compositional inverse $\cH^{\langle -1 \rangle}(z)$ so that $\cH(\cH^{\langle -1 \rangle}(z)) = z$. Substituting $z$ by $\cH^{\langle -1 \rangle}(z)$ in the preceding display yields
\[
	u(\cH^{\langle -1 \rangle}(z)) = z a(u(\cH^{\langle -1 \rangle}(z))), \qquad  a(z) := (1-z)^{-1}(1-z^2)^{-2}.
\]
Moreover, by~\eqref{eq:hcore} and~\eqref{eq:btp}
\[
	\cS(z) = \cV^*( \cH^{\langle -1 \rangle}(z)) = b(u(\cH^{\langle -1 \rangle}(z))), \qquad b(z) := z(1-z-z^2).
\]
It follows by Lagrange--B\"urmann inversion that
\[
	[z^n]\cS(z) = \frac{1}{n} [z^{n-1}] b'(z)a(z)^n = \frac{1}{n} [z^{n-1}] (-3 z^2-2 z+1)a(z)^n.
\]
We have
\[
	\frac{xa'(x)}{a(x)} = \frac{x(1+5x)}{1-x^2}.
\]
Hence the smallest positive solution of $xa'(x)/ a(x)=1$ is given by $r = 1/3$. Define a nonnegative integer-valued random variable $\zeta$ with probability generating function
\[
	\Exb{z^\zeta} = a(r z ) / a(r).
\]
By choice of $r$ we have 
\[
	\Exb{\zeta} = 1, \qquad \Va{\zeta} = 15/8.
\]
Let $S_n^\zeta$ denote the sum of $n$ independent copies of $\zeta$. Note that $\Prb{\zeta=k}>0$ for all integers $k \ge 0$. Moreover, $\zeta$ has finite exponential moments. Hence by the Chebyshev-Edgeworth-Cram\'er expansion~\cite[Thm. 13 in Chap. VII]{Petrov1975} we have for each fixed $\ell \in \ndZ$
\begin{align*}
	[z^{n-1-\ell}] a(z)^n &= \frac{a(r)^n}{r^{n-1-\ell}} [z^{n-1-\ell}] \left(\frac{a(z r)}{a(r)}\right)^n  \\
	&= \frac{a(r)^n}{r^{n-1-\ell}} \mathbb{P}(S_n^\zeta=n-1-\ell) \\
	&= \frac{1}{\sqrt{2\pi n\Va{\zeta}}}\frac{a(r)^n}{r^{n-1-\ell}}\left(1 + \frac{361 + 144\ell - 480\ell^2}{1800n} + O(n^{-2}) \right).
\end{align*}
The reason why we end up with $n^{-5/2}$ is that the leading terms in the asymptotic expansion cancel out. This is why we state the second order asymptotic term in the expansion of $[z^{n-1-\ell}] a(z)^n$. It follows that
\begin{align}
	\label{eq:S-asymptotic}
[z^n]\cS(z) &= \frac{1}{n}  \left(-3 ([z^{n-3}]a(z)^n) -2 ([z^{n-2}]a(z)^n) + ([z^{n-1}]a(z)^n) \right) \\
	&\sim \frac{64}{225\sqrt{15\pi}} n^{-5/2} \left(\frac{729}{128}\right)^n. \nonumber
\end{align}

Next, we evaluate the generating series at their radii of convergence. Set
\begin{align*}
	v(t):=t(1-t)^2.
\end{align*}
By \eqref{eq:uasymptotic} and \eqref{eq:ucoeff}, the power series $u$ has radius of convergence $4/27$, is analytic in the disk $|z|<4/27$, and is continuous and strictly increasing on $[0,4/27]$. The identity $v(u(z))=z$, initially valid in a neighbourhood of the origin, holds throughout $|z|<4/27$ by analytic continuation and extends to $z=4/27$ by continuity. The function $v$ is strictly increasing on $[0,1/3]$, with $v(1/3) = 4/27$. Consequently, $u(4/27) = 1/3$. By~\eqref{eq:btp} and~\eqref{eq:hcore} it follows that
\begin{align*}
	\cV^*(4/27) = 5/27, \qquad \cH(4/27) = 128/729, \qquad \cS(128/729) = 5/27.
\end{align*}
Moreover, for $0 \le t \le 1/3$ we have $u(v(t)) = t$, and hence
\begin{align*}
	\cH(v(t))&=t(1-t)(1-t^2)^2, \\
	\cV^*(v(t))&=t(1-t-t^2).
\end{align*}
Differentiating and dividing by $v'(t)$ yields for $0<t<1/3$
\begin{align*}
	\cH'(v(t)) &= (1-t)(1+t)(1+2t), \\
	(\cV^*)'(v(t)) &= \frac{1+t}{1-t}.
\end{align*}
The coefficient estimates~\eqref{eq:vgrowth} and~\eqref{eq:hgrowth} ensure that the first derivatives converge at the corresponding radii of convergence. It follows that
\begin{align*}
	\cH'(4/27)=\frac{40}{27},
	\qquad
	(\cV^*)'(4/27)=2.
\end{align*}
Finally, differentiating $\cV^*(z)=\cS(\cH(z))$ for $0<z<4/27$ and letting $z \uparrow 4/27$ yields
\begin{align*}
	\cS'(128/729)=\frac{27}{20}.
\end{align*}

\subsection{Enriched tree encoding and core size}

As discussed in Section~\ref{sec:simplecomponents}, any non-separable planar map $V$ that is not the loop map may be uniquely described by its simple component at the root whose edges get substituted by the $\cH$-components of $V$. We let $\cS_0(V)$ denote the simple component at the root edge of $V$.  Each $\cH$-component $H$ that is not the link map consisting of two vertices joined by a single edge has exactly two distinct edges on the boundary of the outer face. The result of deleting the non-root edge on the boundary of $H$ is a non-separable map that may again be decomposed into a simple component with $\cH$-structures attached. Iterating this operation decomposes $V$ uniquely into simple components. We let $\cS(V)$ denote a canonically chosen simple component whose number of edges is maximal among all simple components of $V$, with ties broken according to the order induced by the root and the recursive decomposition. We call $\cS(V)$ the simple core of $V$. 

The described decomposition of $\cH$-structures is reflected in the following equation that results from~\eqref{eq:hcore} 
\begin{align}
	\label{eq:simpleenrichedtree}
	\cH(z)=z\bigl(1+\cS(\cH(z))\bigr)=z\cR(\cH(z)), \qquad 	\cR(z):=1+\cS(z).
\end{align}
We define $\cR$-objects as either simple non-separable planar maps or the map consisting of a single vertex and no edges. Equation~\eqref{eq:simpleenrichedtree} and the decomposition behind it identify the class of $\cH$-structures as the class of $\cR$-enriched trees~\cite{MR633783,MR1629341}. 

\begin{definition}
	An $\cR$-enriched plane tree is a pair $(T, \beta)$ so that $T$ is a (rooted) plane tree and $\beta$ assigns to each  vertex $v$ of $T$ with  outdegree denoted by $d^+_T(v)$ an $\cR$-structure $\beta(v)$ with $d^+_T(v)$ edges. We say $\beta(v)$ is the decoration of the vertex $v$, and $(T, \beta)$ is an $\cR$-enriched plane tree.
\end{definition}

The following structural result follows automatically by general combinatorial principles, see~\cite{MR633783,MR1629341}. We provide a justification for the reader's convenience.

\begin{proposition}
	There is a bijection between $\cH$-objects with $n \ge 1$ edges and  $\cR$-enriched plane trees with $n$ vertices.
\end{proposition}
\begin{proof}
	Let $H$ denote an $\cH$-object. If $H$ is the link map, we represent it by a tree $T$ consisting of a single vertex, to which we assign a simple map consisting of a single vertex and no edges as decoration. In this case it trivially holds that $\ed(H) = 1 = \ve(T)$.
	
	If $H$ is not the link map, it decomposes into a simple component $S$ with a number $d$ of edges ordered in canonical manner (starting with the root edge of $S$), and $\cH$-components $H_1, \ldots, H_d$ attached to the ordered edges of $S$. We have
	\begin{align}
		\label{eq:edheran}
		\ed(H) = 1 + \sum_{i=1}^d \ed(H_i),
	\end{align}
	since the decomposition of $H$ deletes the non-root edge on the $2$-cycle boundary of the outer face. We recursively represent $H$ by a plane tree $T$ whose root $o$ has outdegree $d$ and whose decoration is given by $\beta(o)=S$, so that the $i$th child for $1 \le i \le d$ gets identified with the root of the $\cR$-enriched plane tree $(T_i, \beta_i)$ corresponding to $H_i$. By structural induction we have $\ed(H_i) = \ve(T_i)$ for all $1 \le i \le d$, hence we have by~\eqref{eq:edheran} 
	\[
		\ed(H) = 1 + \sum_{i=1}^d \ve(T_i) = \ve(T).
	\]
	
	This shows that any $\cH$-object with $n$ edges gives rise to an $\cR$-enriched plane tree $(T, \beta)$ with $n$ vertices. 
	
	Conversely, let $(T,\beta)$ be an $\cR$-enriched plane tree with root $o$. If $\beta(o)$ is the empty $\cR$-object, then we associate with $(T,\beta)$ the link map. Otherwise, we recursively construct the $\cH$-objects associated with the enriched fringe subtrees rooted at	the children of $o$ and substitute them at the corresponding ordered edges of $\beta(o)$. We then adjoin the non-root boundary edge of the outer $2$-cycle. The resulting map is an $\cH$-object.
	
	The two constructions are inverse to each other and hence bijective.
\end{proof}

It is important to note that the only constraint on a decoration of a vertex is the outdegree of the vertex.

\begin{proposition}[{\cite[Lem. 6.1]{MR4132643}}]
	Let $\mH_n$ denote a uniform random $\cH$-object with $n \ge 1$ edges. Let $(\mT^{\cS}_n, \beta^\cS_n)$ denote its corresponding $\cR$-enriched tree. Then $\mT^{\cS}_n$ is distributed like a simply generated tree with weight sequence $\omega_k = [z^k] \cR(z)$, $k \ge 0$. Conditionally on~$\mT^{\cS}_n$, the decorations $(\beta^\cS_n(v))_{v\in\mT^{\cS}_n}$ are independent, and for each $v\in\mT^{\cS}_n$, the decoration $\beta^\cS_n(v)$ is uniformly distributed among all $\cR$-objects with $d_{\mT^{\cS}_n}^+(v)$ edges.
\end{proposition}

The simply generated tree model assumes a given finite plane tree $T$ having $n$ vertices with probability proportional to its weight $ \prod_{v \in T} \omega_{d^+_T(v)}$.
The series
\begin{align*}
	\varphi(z):=\sum_{k\ge0}\omega_kz^k=1+\cS(z).
\end{align*}
is known as the weight generating series of the model. Since $\cS(\rho_\cS)< \infty$ for the radius of convergence $\rho_\cS = 128/729$ of $\cS(z)$,  we may define a non-negative integer-valued random variable $\xi^\cS$ with probability generating function
\begin{align}
	\label{eq:taxi}
	\Ex{z^{\xi^\cS}} = \frac{\varphi(\rho_\cS z)}{\varphi(\rho_\cS)}
	= \frac{1+\cS(\rho_\cS z)}{1+\cS(\rho_\cS)}.
\end{align}
The simply generated tree $\mT^{\cS}_n$ is distributed like a ${\xi^\cS}$-Bienaym\'e--Galton--Watson tree $\mT^{\cS}$ conditioned on having $n$ vertices. See~\cite[Sec. 4]{MR2908619} for details. Using the values computed in Subsection~\ref{sec:enum}, we obtain
\begin{align}
	\label{eq:muxi}
	\Ex{{\xi^\cS}} = \frac{\rho_\cS\varphi'(\rho_\cS)}{\varphi(\rho_\cS)} =
\frac{\rho_\cS\cS'(\rho_\cS)}{1+\cS(\rho_\cS)} = \frac{1}{5}.
\end{align}
Moreover,~\eqref{eq:S-asymptotic} yields
\begin{align}
	\label{eq:simpleoffspringtail}
	\Pr{{\xi^\cS}=k} \sim \frac{6}{25\sqrt{15\pi}}k^{-5/2}, \qquad k \to \infty.
\end{align}

The outdegrees of the vertices in $\mT^{\cS}_n$ correspond to the number of edges in the simple components of $\mH_n$. Specifically, the maximal outdegree $\Delta(\mT^{\cS}_n)$ corresponds to the maximal number of edges in a simple component of $\mH_n$. Moreover, for $n \ge 3$, if we delete the non-root edge from the boundary of $\mH_n$ we are left with a map that is distributed like $\mV_{n-1}$.  Let $Z_{3/2}$ and $h$ be as in Proposition~\ref{pro:lltvcore}. By~\eqref{eq:taxi} and~\eqref{eq:muxi} the local limit theorem~\cite{MR4144886} for the maximal degree in conditioned subcritical Bienaym\'e--Galton--Watson trees applies, yielding:

\begin{proposition}
	\label{pro:edsvm1}
	Set $ g_{\cV} = \left(\frac{64}{9375}\right)^{1/3}$. We have
	\begin{align*}
		\Pr{\ed(\cS(\mV_n))=\ell} = \frac{1}{g_\cV n^{2/3}} \left(  h\left(\frac{4n/5-\ell}{g_\cV n^{2/3}}  \right) + o(1) \right)
	\end{align*}
	uniformly for all integers $\ell \ge 0$. Consequently, 
	\begin{align*}
		\frac{4n/5 - \ed(\cS(\mV_n))}{g_\cV n^{2/3}} \convd Z_{3/2}.
	\end{align*}
	Moreover, for any sufficiently small $\epsilon>0$ and any $0 < \delta < 4/5$ we have
	\[
		\Prb{\ed(\cS(\mV_n)) \le \epsilon n / \log(n)} = o(n^{-2/3})
	\]
	and
	\[
		\Prb{\ed(\cS(\mV_n)) \le \delta n} = O(n^{-2/3} ).
	\]
\end{proposition}
\begin{proof}
	The local limit theorem follows  from~\cite[Thm. 1.1]{MR4144886}, and the central limit theorem is a direct consequence of it. The bound for $\Prb{\ed(\cS(\mV_n)) \le \epsilon n / \log(n)}$ follows from~\cite[Eq. (2.10)]{MR4144886}. The bound for $\Prb{\ed(\cS(\mV_n)) \le \delta n}$ follows from~\cite[Eq. (2.11), (2.12)]{MR4144886}. 
\end{proof}

Let $v^*$ denote the lexicographically first vertex of $\mT^{\cS}_{n+1}$ with maximal outdegree. Deleting the edges between $v^*$ and its children yields a connected component $F_0(\mT^{\cS}_{n+1})$ that contains the root vertex $o$, and components $F_i(\mT^{\cS}_{n+1})$, $1 \le i \le \Delta(\mT^{\cS}_{n+1})$ corresponding to the fringe subtrees dangling from the children of $v^*$. 

Throughout the following we let the uniform non-separable planar map $\mV_n$ with $n \ge 2$ edges be constructed by taking the $\cH$-object $\mH_{n+1}$ corresponding to $(\mT^{\cS}_{n+1}, \beta^\cS_{n+1})$, and deleting the non-root edge on its boundary. Furthermore, in case that there are multiple simple components attaining the maximal number of edges we  break the tie in such a way that the simple core $\cS(\mV_n)$ of $\mV_n$ always corresponds to the decoration $\beta^\cS_{n+1}(v^*)$ of the vertex~$v^*$.

We may view $\mV_n$ as consisting of the simple core $\cS(\mV_n)$ together with maps $\cV_i(\mV_n)$, $1 \le i \le \ed(\cS(\mV_n))$ that get substituted at its edges: 

In case that $v^*$ is equal to the root vertex of $\mT^{\cS}_{n+1}$ it holds that the simple component at the root edge of $\mV_n$ is equal to its simple core. We let $\cV_i(\mV_n)$ denote the $\cH$-components inserted at the edges of this simple component. This way, the root edge of $\cV_1(\mV_n)$ corresponds to the root edge of $\mV_n$. 

In case that $v^*$ is not equal to the root vertex of $\mT^{\cS}_{n+1}$, the tree $F_0(\mT^{\cS}_{n+1})$ with its decorations corresponds to a planar map $M_1$ that has a marked and canonically oriented edge $e^*$ apart from its root edge.  The enriched fringe subtree of $\mT^{\cS}_{n+1}$ at the vertex $v^*$ corresponds (with its decorations) to a planar map $M_2$, and $\mV_n$ is constructed by deleting $e^*$ from $M_1$ and inserting   $M_2$ instead. The planar map $M_2$ is an $\cH$-object and hence decomposes into a simple component - precisely $\cS(\mV_n)$ - and $\cH$-objects $H_1, \ldots, H_{\ed(\cS(\mV_n))}$ that correspond to the trees $F_i(\mT^{\cS}_{n+1})$, $1 \le i \le \Delta(\mT^{\cS}_{n+1})$ and their decorations. Specifically, $F_1(\mT^{\cS}_{n+1})$ corresponds to the component of the root edge of $\cS(\mV_n)$. Thus, for $2 \le i \le \Delta(\mT^{\cS}_{n+1})$ we let $\cV_i(\mV_n)$ denote the $\cH$-component $H_i$. We let $\cV_1(\mV_n)$ denote the map obtained by inserting the $\cH$-component $H_1$ at the edge $e^*$ of $M_1$. This way, $\cV_1(\mV_n)$ contains the root edge of $\mV_n$.

Applying the limits ~\cite[Thm. 1.2, Cor. 1.10]{MR4144886} on the sizes of $F_i(\mT^{\cS}_{n+1})$, $0 \le i \le \Delta(\mT^{\cS}_{n+1})$ readily yields:

\begin{proposition}
	\label{pro:edsvm2}
	We have $\ed(\cV_1(\mV_n)) = O_p(1)$ and
	\begin{align*}
		\max_{1 \le i \le \ed(\cS(\mV_n))} \ed(\cV_i(\mV_n)) = O_p(n^{2/3}).
	\end{align*}
\end{proposition}

\subsection{Approximating non-separable planar maps by their simple core}
\label{sec:approximationthroughcore}
For ease of notation, we set
\[
	s_n = \ed(\cS(\mV_n)).
\]
Let $e_1, \ldots, e_{s_n}$ denote the edges of $\cS(\mV_n)$ in the canonical order used for the components $(\cV_i(\mV_n))_{1 \le i \le s_n}$.  For a vertex $v \in \cS(\mV_n)$ let $I_n(v) \subset \{1, \ldots, s_n\}$ denote the set of indices $i$ for which $v$ is incident to $e_i$. We let $\nu_{\cS(\mV_n)}$ denote the probability measure on $\cS(\mV_n)$ given by
\[
	\nu_{\cS(\mV_n)}(\{v\}) = \frac{1}{2n}\sum_{i \in I_n(v)}\ed(\cV_i(\mV_n)).
\]
The fact that this is a probability measure follows from
\[
	\sum_{i=1}^{\ed(\cS(\mV_n))} \ed(\cV_i(\mV_n)) = n,
\]
since every edge of $\cS(\mV_n)$ has two distinct endpoints. We emphasise that $\nu_{\cS(\mV_n)}$ depends on the entire map $\mV_n$, whereas the stationary distribution $\mu_{\cS(\mV_n)}$ only depends on the simple core $\cS(\mV_n)$.

\begin{lemma}
	\label{le:simplecorepart1}
	For any constant $c>0$,
	\begin{align*}
		d_{\mathrm{GHP}}\left((\mV_n,cn^{-1/4}d_{\mV_n},\mu_{\mV_n}),
		(\cS(\mV_n),cn^{-1/4}d_{\cS(\mV_n)},\nu_{\cS(\mV_n)})\right)
		\convp 0.
	\end{align*}
\end{lemma}
\begin{proof}
	Recall that, by the second identity in~\eqref{eq:hcore}, deleting the non-root boundary edge of $\mH_{k+1}$ yields a uniform non-separable non-loop map with $k$ edges. Adding the deleted boundary edge cannot increase graph distances. Hence, the upper diameter bound in~\cite[Thm. 3.7]{MR3318042} for the diameter of $\mV_{k}$ implies that, for every sufficiently small constant $\epsilon>0$, there are constants $a>0$ and $k_0 \ge 1$ such that
	\begin{align}
		\label{eq:hdiameterdeviation}
		\Pr{\Di(\mH_k)>k^{1/4+\epsilon}}
		\le \exp(-k^a)
	\end{align}
	for all $k \ge k_0$.
	
	Propositions~\ref{pro:edsvm1} and~\ref{pro:edsvm2} yield
	\begin{align}
		\label{eq:simplecoresize}
		s_n=4n/5+O_p(n^{2/3}), \qquad \ed(\cV_1(\mV_n))=O_p(1), \qquad
		\max_{1 \le i \le s_n}\ed(\cV_i(\mV_n))=O_p(n^{2/3}).
	\end{align}
	Choose constants $\epsilon$ and $\delta$ such that
	\[
		0<\epsilon<1/8, \qquad \frac{2}{3}\left(\frac{1}{4}+\epsilon\right)<\delta<\frac{1}{4}.
	\]
	Let $\cE_n$ denote the event
	\[
		\max_{1 \le i \le s_n}\ed(\cV_i(\mV_n))\le \log(n) n^{2/3}.
	\]
	For all sufficiently large $n$ we have
	\begin{align}
		\label{eq:sch}
		(\log(n) n^{2/3})^{1/4 + \epsilon} \le n^{\delta}.
	\end{align}
	A connected map with $k$ edges has diameter at most $k$. Hence, using~\eqref{eq:hdiameterdeviation}, the union bound, and the deterministic inequality $s_n \le n$, we obtain for sufficiently large $n$ (so that~\eqref{eq:sch} holds and so that $n^\delta > k_0$)
	\begin{align*}
		\Prb{\max_{2 \le i \le s_n}\Di(\cV_i(\mV_n))>n^\delta,\ \cE_n}
		&\le n \max_{k_0 \le k \le \log(n) n^{2/3}} \Prb{\Di(\mH_k) > n^\delta} \\
		&= n \max_{n^{\delta} \le k \le \log(n) n^{2/3}} \Prb{\Di(\mH_k) > n^\delta} \\
		&\le n \max_{n^{\delta} \le k \le \log(n) n^{2/3}} \Prb{\Di(\mH_k) > k^{1/4 + \epsilon}} \\
		&\le n\exp(-n^{a \delta}) \to 0.
	\end{align*}
	By~\eqref{eq:simplecoresize} we have $\Prb{\cE_n} \to 1$ as $n \to \infty$ and  $\Di(\cV_1(\mV_n)) \le \ed(\cV_1(\mV_n)) = O_p(1)$. Hence
	\begin{align}
		\label{eq:smallsimpleattachments}
		\max_{1 \le i \le s_n}\Di(\cV_i(\mV_n))=o_p(n^{1/4}).
	\end{align}
	
	Each component $\cV_i(\mV_n)$ intersects the simple core in the edge $e_i$. Any path in $\mV_n$ between two vertices of $\cS(\mV_n)$ may therefore be transformed into a path in $\cS(\mV_n)$ of no greater length by replacing every maximal subpath contained in a component $\cV_i(\mV_n)$ by the edge $e_i$, or by a path of length zero if the subpath enters and leaves the component through the same endpoint of $e_i$. Since $\cS(\mV_n)$ is a subgraph of $\mV_n$, it follows that
	\begin{align}
		\label{eq:simplecoreisometric}
		d_{\mV_n}(x,y)=d_{\cS(\mV_n)}(x,y)
	\end{align}
	for all vertices $x,y \in \cS(\mV_n)$. 
	
	Every vertex of $\mV_n$ belongs to one of the components $\cV_i(\mV_n)$ and hence lies at graph distance at most $\max_{1 \le i \le s_n}\Di(\cV_i(\mV_n))$ from $\cS(\mV_n)$. Thus, by~\eqref{eq:smallsimpleattachments} and~\eqref{eq:simplecoreisometric}
	\begin{align}
		\label{eq:hsimplecore}
		d_{\mathrm{H}}\left( (\mV_n,cn^{-1/4}d_{\mV_n}), (\cS(\mV_n),cn^{-1/4}d_{\cS(\mV_n)}) \right) \convp 0.
	\end{align}
	
	Let $c_n$ denote a uniformly selected corner of $\mV_n$, and let $v_n$ denote its incident vertex. Conditionally on $\mV_n$, the vertex $v_n$ is distributed according to $\mu_{\mV_n}$. There is a unique index $i(c_n) \in \{1, \ldots, s_n\}$ such that the half-edge corresponding to $c_n$ belongs to $\cV_{i(c_n)}(\mV_n)$. Conditionally on $c_n$ and $\mV_n$, let $u_n$ be a uniformly selected endpoint of the edge $e_{i(c_n)}$. Hence
	\[
		\Pr{i(c_n)=i \mid \mV_n} = \frac{\ed(\cV_i(\mV_n))}{n}.
	\]
	Thus, conditionally on $\mV_n$, the vertex $u_n$ is distributed according to $\nu_{\cS(\mV_n)}$. Moreover, $v_n$ and $u_n$ both belong to $\cV_{i(c_n)}(\mV_n)$. It follows from~\eqref{eq:smallsimpleattachments} that
	\begin{align}
		\label{eq:couplingsimplecore}
		n^{-1/4}d_{\mV_n}(v_n,u_n)
		\le
		n^{-1/4}\max_{1 \le i \le s_n}\Di(\cV_i(\mV_n))
		\convp 0.
	\end{align}
	Using the characterisation of the Gromov--Hausdorff--Prokhorov distance from~\cite[Prop. 6]{MR2571957}, together with~\eqref{eq:hsimplecore} and~\eqref{eq:couplingsimplecore} yields the asserted convergence.
\end{proof}

\section{The simple non-separable core of unrestricted planar maps}

\subsection{Core size}

We use the notation
$
	\cS(\mM_n):=\cS(\cV(\mM_n))
$ for the simple non-separable core of $\mM_n$. Here we set $\cS(\mM_n)$ to some placeholder value in case that $\cV(\mM_n)$ is the loop map. Concatenating the local limit theorems from Propositions~\ref{pro:lltvcore} and~\ref{pro:edsvm1} yields a local limit theorem for the number of edges in $\cS(\mM_n)$. A similar argument was used in~\cite[Cor. 9.6]{zbMATH07665039}.

\begin{proposition}
	\label{pro:lltsourcesimplecore}
	Set
	\begin{align*}
		g_{\cS} := \left( \left(\frac{4}{5}g_\cM\right)^{3/2} + \left(3^{-2/3}g_\cV\right)^{3/2}\right)^{2/3} = \left(\frac{33856}{759375}\right)^{1/3}.
	\end{align*}
	We have
	\begin{align*}
		\Pr{\ed(\cS(\mM_n))=\ell} = \frac{1}{g_\cS n^{2/3}} \left(  h\left(\frac{4n/15-\ell}{g_\cS n^{2/3}}  \right) + o(1) \right)
	\end{align*}
	uniformly for all integers $\ell \ge 0$. Consequently, 
	\begin{align*}
		\frac{4n/15 - \ed(\cS(\mM_n))}{g_\cS n^{2/3}} \convd Z_{3/2}.
	\end{align*}
\end{proposition}
\begin{proof}
	Conditionally on $\ed(\cV(\mM_n))=j$, the map $\cV(\mM_n)$ is a uniform non-separable planar map with $j$ edges. Hence
	\begin{align*}
		\Pr{\ed(\cS(\mM_n))=\ell} = \sum_{j=1}^{n} \Pr{\ed(\cV(\mM_n))=j} \Pr{\ed(\cS(\mV_j))=\ell}.
	\end{align*}
	Since the density $h$ is bounded, the uniform local limit
	theorem in Proposition~\ref{pro:edsvm1} yields a constant
	$C_0>0$ such that for all $j\ge1$
	\begin{align*}
		\sup_{\ell\in\ndZ} \Pr{\ed(\cS(\mV_j))=\ell} \le C_0j^{-2/3}.
	\end{align*}
	By Proposition~\ref{pro:lltvcore} we may select a sufficiently small constant $\epsilon>0$ and any constant $0 < \delta < 1/3$ so that
	\begin{align}
		\label{eq:wdy}
		&\sup_{\ell\in\ndZ} \sum_{1\le j\le\delta n} \Pr{\ed(\cV(\mM_n))=j} \Pr{\ed(\cS(\mV_j))=\ell} \\
		&\qquad\le \Pr{ \ed(\cV(\mM_n)) \le \frac{\epsilon n}{\log(n)}} + C_0 \left(\frac{\epsilon n}{\log(n)} \right)^{-2/3} \Pr{\ed(\cV(\mM_n))\le\delta n} \nonumber \\
		&\qquad= o(n^{-2/3}) + O\left( n^{-4/3}\log(n)^{2/3} \right) \nonumber\\
		&\qquad= o(n^{-2/3}). \nonumber
	\end{align}
	For $M>0$, set
	\begin{align*}
		\cI_n(M) := \{ j \in \ndZ \mid  |j-n/3| \le Mn^{2/3} \}.
	\end{align*}
	For fixed $M>0$ and all sufficiently large $n$, every $j\in\cI_n(M)$ satisfies $j>\delta n$. Hence using~\eqref{eq:wdy}
	\begin{multline*}
		n^{2/3}	\sup_{\ell\in\ndZ} \sum_{j\notin\cI_n(M)} \Pr{\ed(\cV(\mM_n))=j} \Pr{\ed(\cS(\mV_j))=\ell} \\
		\le o(1) + C_0\delta^{-2/3}	\Prb{ \left|\ed(\cV(\mM_n)) - n/3 \right| > Mn^{2/3}}.
	\end{multline*}
	The convergence in distribution in Proposition~\ref{pro:lltvcore} therefore implies
	\begin{align}
		\label{eq:buzz}
		\lim_{M\to\infty} \limsup_{n\to\infty} n^{2/3} \sup_{\ell\in\ndZ} \sum_{j\notin\cI_n(M)} \Pr{\ed(\cV(\mM_n))=j}\Pr{\ed(\cS(\mV_j))=\ell} = 0.
	\end{align}
	
	Fix $\eta>0$. By~\eqref{eq:buzz}, we may choose $M_1>0$ sufficiently large such that
	\begin{align}
		\label{eq:hamm}
		\limsup_{n\to\infty} n^{2/3} \sup_{\ell\in\ndZ} \sum_{j\notin\cI_n(M_1)} \Pr{\ed(\cV(\mM_n))=j} \Pr{\ed(\cS(\mV_j))=\ell} \le \eta.
	\end{align}
	With foresight, we enlarge $M_1$ further so that
	\begin{align}
		\label{eq:rex}
		\frac{\|h\|_\infty}{3^{-2/3}g_\cV} \int_{\ndR\setminus [-M_1/g_\cM,M_1/g_\cM]} h(u)\,\mathrm{d}u \le \eta.
	\end{align}
	For $M_2>0$, set
	\begin{align*}
		\cJ_n(M_2) := \left\{ \ell\in\ndZ \mid  \left| \ell-\frac{4n}{15}
		\right| \le M_2n^{2/3} \right\}.
	\end{align*}
	If $j\in\cI_n(M_1)$ and $\ell\notin\cJ_n(M_2)$, then
	\begin{align*}
		\left| \frac{4j}{5}-\ell \right| &\ge \left| \frac{4n}{15}-\ell\right| - \frac{4}{5} \left| j-\frac{n}{3} \right| \\
		&> \left( M_2-\frac{4M_1}{5} \right)n^{2/3}. 
	\end{align*}
	Moreover, $j/n\to1/3$ uniformly for $j\in\cI_n(M_1)$. By Proposition~\ref{pro:edsvm1} it follows that we may choose $M_2>4M_1/5$ sufficiently large so that
	\begin{align}
		\label{eq:sourceconditionaltail}
		\limsup_{n\to\infty} n^{2/3} \sup_{\substack{j\in\cI_n(M_1)\\ \ell\notin\cJ_n(M_2)}} \Pr{\ed(\cS(\mV_j))=\ell} \le	\eta.
	\end{align}
	With foresight, we  enlarge $M_2$ further so that
	\begin{align}
		\label{eq:slin}
		\frac{1}{g_\cS} \sup_{|x|\ge M_2/g_\cS} h(x) \le \eta.
	\end{align}
	Equations~\eqref{eq:hamm}, and \eqref{eq:sourceconditionaltail} imply
	\begin{align}
		\label{eq:potat}
		\limsup_{n\to\infty} n^{2/3} \sup_{\ell\notin\cJ_n(M_2)} \Pr{\ed(\cS(\mM_n))=\ell} \le 2\eta.
	\end{align}
	
	It remains to consider $\ell\in\cJ_n(M_2)$. Uniformly for $j\in\cI_n(M_1)$, we have
	\begin{align}
		\label{eq:sourceconditionalscaling}
		j^{2/3} = 3^{-2/3}n^{2/3}(1+o(1)).
	\end{align}
	Furthermore,
	\begin{align}
		\label{eq:sourceconditionalargument}
		\frac{4j/5-\ell}{g_\cV j^{2/3}} = \frac{ (4n/15-\ell)/n^{2/3} - \frac{4}{5}g_\cM \frac{n/3-j}{g_\cM n^{2/3}} }{g_\cV(j/n)^{2/3}}.
	\end{align}
	The density $h$ is bounded and uniformly continuous. Consequently, the uniform local limit theorems in Propositions~\ref{pro:lltvcore} and~\ref{pro:edsvm1}, together with~\eqref{eq:rex} ~\eqref{eq:sourceconditionalscaling} and
	\eqref{eq:sourceconditionalargument}, give 	uniformly for $\ell\in\cJ_n(M_2)$
	\begin{align*}
		&n^{2/3} \sum_{j\in\cI_n(M_1)} \Pr{\ed(\cV(\mM_n))=j} \Pr{\ed(\cS(\mV_j))=\ell} \nonumber \\
		&\qquad= \sum_{j\in\cI_n(M_1)} \frac{1}{g_\cM n^{2/3}} \frac{n^{2/3}}{g_\cV j^{2/3}} h\left(\frac{n/3-j}{g_\cM n^{2/3}}\right)	h\left(	\frac{4j/5-\ell}{g_\cV j^{2/3}}\right) + o(1) \\
		&\qquad = \frac{1}{3^{-2/3}g_\cV} \sum_{j\in\cI_n(M_1)} \frac{1}{g_\cM n^{2/3}} h\left( \frac{n/3-j}{g_\cM n^{2/3}} \right) h\left( \frac{ (4n/15-\ell)/n^{2/3} - \frac{4}{5}g_\cM \frac{n/3-j}{g_\cM n^{2/3}}	}{3^{-2/3}g_\cV} \right) + o(1)	\\
		&\qquad= \frac{1}{3^{-2/3}g_\cV} \int_{-M_1/g_\cM}^{M_1/g_\cM} h(u) h\left( \frac{(4n/15-\ell)/n^{2/3} - \frac{4}{5}g_\cM u } {3^{-2/3}g_\cV} \right) \,\mathrm{d}u + o(1) \\
		&\qquad= \frac{1}{3^{-2/3}g_\cV} \int_\ndR h(u) h\left( \frac{(4n/15-\ell)/n^{2/3} - \frac{4}{5}g_\cM u	}{3^{-2/3}g_\cV} \right) \,\mathrm{d}u + R +  o(1)
	\end{align*}
	for an error term $R$ satisfying $|R| \le \eta$.
	
	Let $Z(1)$ and $Z(2)$ be independent copies of $Z_{3/2}$. By the definition of $g_\cS$ and strict $3/2$-stability of $Z_{3/2}$,
	\begin{align*}
		\frac{4}{5}g_\cM Z(1) + 3^{-2/3}g_\cV Z(2) \eqdist g_\cS Z_{3/2}.
	\end{align*}
	Comparing densities yields, for every $y\in\ndR$,
	\begin{align*}
		\frac{1}{3^{-2/3}g_\cV}	\int_{\ndR} h(u) h\left( \frac{	y-\frac{4}{5}g_\cM u }{3^{-2/3}g_\cV}\right)\,\mathrm{d}u
		= \frac{1}{g_\cS} h\left( \frac{y}{g_\cS} \right).
	\end{align*}
	By \eqref{eq:hamm} it follows that
	\begin{align*}
		\limsup_{n\to\infty} \sup_{\ell\in\cJ_n(M_2)} \left|n^{2/3}	\Pr{\ed(\cS(\mM_n))=\ell} - \frac{1}{g_\cS}	h\left( \frac{4n/15-\ell}{g_\cS n^{2/3}} \right) \right| \le 2\eta.
	\end{align*}
	For $\ell\notin\cJ_n(M_2)$, Equations~\eqref{eq:potat} and
	\eqref{eq:slin} give
	\begin{align*}
		\limsup_{n\to\infty} \sup_{\ell\notin\cJ_n(M_2)} \left| n^{2/3} \Pr{\ed(\cS(\mM_n))=\ell} - \frac{1}{g_\cS}	h\left(	\frac{4n/15-\ell}{g_\cS n^{2/3}} \right) \right| \le 3\eta.
	\end{align*}
	Since $\eta>0$ was arbitrary, it follows that
	\begin{align*}
		\Pr{\ed(\cS(\mM_n))=\ell} = \frac{1}{g_\cS n^{2/3}} \left(  h\left(\frac{4n/15-\ell}{g_\cS n^{2/3}}  \right) + o(1) \right)
	\end{align*}
	uniformly for all integers $\ell \ge 0$. Using the values of $g_\cM$ and $g_\cV$ from Propositions~\ref{pro:lltvcore} and~\ref{pro:edsvm1}, respectively, we obtain
	\begin{align*}
		\left( \frac{4}{5}g_\cM \right)^{3/2} + \left( 3^{-2/3}g_\cV \right)^{3/2}
		= \frac{32}{45\sqrt{15}} + \frac{8}{75\sqrt{15}}
		= \frac{184}{225\sqrt{15}}.
	\end{align*}
	Hence
	\begin{align*}
		g_\cS^3 = \left( \frac{184}{225\sqrt{15}} \right)^2 = \frac{33856}{759375},
	\end{align*}
	which verifies the stated explicit value of $g_\cS$.
\end{proof}

\subsection{Scaling limit}
\label{sec:sourcesimplecorescaling}

For ease of notation, set
\[
	k_n=\ed(\cV(\mM_n)), \qquad	s_n=\ed(\cS(\mM_n)).
\]
Let $e_{n,1}, \ldots, e_{n,s_n}$ denote the edges of $\cS(\mM_n)$ in the canonical order of the decomposition of $\cV(\mM_n)$ from Subsection~\ref{sec:approximationthroughcore}, and let
\[
	\cV_{n,i}:=\cV_i(\cV(\mM_n)), \qquad W_{n,i}:=\ed(\cV_{n,i}), \qquad 1 \le i \le s_n.
\]
Thus,
\begin{align}
	\label{eq:atta}
	\sum_{i=1}^{s_n}W_{n,i}=k_n.
\end{align}
For a vertex $v \in \cS(\mM_n)$ let $I_n(v) \subset \{1, \ldots, s_n\}$ denote the set of indices $i$ for which $v$ is incident to $e_{n,i}$. As in Subsection~\ref{sec:approximationthroughcore}, we let $\nu_{\cS(\mM_n)}$ denote the probability measure on $\cS(\mM_n)$ given by
\[
	\nu_{\cS(\mM_n)}(\{v\}) = \frac{1}{2k_n} \sum_{i \in I_n(v)}W_{n,i}.
\]
We emphasise that this measure is obtained from the decomposition of the non-separable core $\cV(\mM_n)$ and hence does not take into account the components attached to $\cV(\mM_n)$ in the block-decomposition of $\mM_n$.

Conditionally on $k_n=k$, the simple core of $\cV(\mM_n)$ is distributed like the simple core of $\mV_k$. Proposition~\ref{pro:lltvcore} yields $k_n / n \convp 1/3$. By Corollary~\ref{eq:corfppvcore} and Lemma~\ref{le:simplecorepart1} we readily obtain:

\begin{lemma}
	\label{le:sourcesimplecorepart1}
	As $n \to \infty$,
	\begin{align*}
		\left( \cS(\mM_n), (8n/9)^{-1/4}d_{\cS(\mM_n)}, \nu_{\cS(\mM_n)} \right)
		\convd
		(\mathbf{M},d_{\mathbf{M}},\mu_{\mathbf{M}})
	\end{align*}
	in the Gromov--Hausdorff--Prokhorov sense.
\end{lemma}

The next lemma replaces the measure induced by the attached components with the stationary distribution on the simple core. The proof is analogous to~\cite[Lem. 5.3 and Cor. 6.2]{MR3729639}.

\begin{lemma}
	\label{le:sourcesimplecorepart2}
	The Prokhorov distance of the two measures $\nu_{\cS(\mM_n)}$ and $\mu_{\cS(\mM_n)}$ on the rescaled simple core $\left(\cS(\mM_n),(8n/9)^{-1/4}d_{\cS(\mM_n)}\right)$ satisfies
	\begin{align*}
		d_{\mathrm{P}}\left( \nu_{\cS(\mM_n)}, \mu_{\cS(\mM_n)} \right) \convp 0.
	\end{align*}
\end{lemma}
\begin{proof}
	 Proposition~\ref{pro:edsvm2}, together with Proposition~\ref{pro:lltvcore}, yields
	\begin{align}
		\label{eq:attasize}
		W_{n,1}=O_p(1),	\qquad	\max_{1 \le i \le s_n}W_{n,i}=O_p(n^{2/3}).
	\end{align}
	Moreover, Propositions~\ref{pro:lltvcore} and~\ref{pro:lltsourcesimplecore} yield
	\begin{align}
		\label{eq:sicosi}
		k_n=\frac{n}{3}+O_p(n^{2/3}), \qquad s_n=\frac{4n}{15}+O_p(n^{2/3}).
	\end{align}
	In particular, $s_n \ge 3$ with probability tending to $1$ as $n \to \infty$. Hence throughout the rest of the proof, we may safely ignore the case $s_n \le 2$ and set any term that would be ill-defined for $s_n \le 2$ to some adequate placeholder value.
	
	It follows from~\eqref{eq:atta},~\eqref{eq:attasize}, and~\eqref{eq:sicosi} that 
	\begin{align}
		\label{eq:sousim}
		\frac{\left(\sum_{i=2}^{s_n}W_{n,i}^2\right)^{1/2}}{\sum_{i=2}^{s_n}W_{n,i}}
		\le \frac{\left((\max_{1 \le i \le s_n}W_{n,i})\sum_{i=2}^{s_n}W_{n,i}\right)^{1/2}}{\sum_{i=2}^{s_n}W_{n,i}}  
		= \left( \frac{ \max_{1 \le i \le s_n}W_{n,i}}{k_n-W_{n,1}}\right)^{1/2} \convp 0.
	\end{align}
	Let $\nu_n^\circ$ and $\mu_n^\circ$ denote the probability measures on $\cS(\mM_n)$ given by
	\begin{align*}
		\nu_n^\circ(\{v\}) &= \frac{1}{2(k_n-W_{n,1})} \sum_{\substack{i \in I_n(v)\\2 \le i \le s_n}}W_{n,i},
		\\
		\mu_n^\circ(\{v\}) &= \frac{1}{2(s_n-1)} \left|\{i \in I_n(v)\mid 2 \le i \le s_n\}\right|.
	\end{align*}
	Fix constants $\epsilon,\delta>0$. Let $K_n$ denote the minimal number of balls of radius $\epsilon$ required to cover $\left(\cS(\mM_n), (8n/9)^{-1/4} d_{\cS(\mM_n)}\right)$. By Lemma~\ref{le:sourcesimplecorepart1}, the sequence of compact metric spaces
	\[
		\left(\cS(\mM_n), (8n/9)^{-1/4} d_{\cS(\mM_n)}\right)
	\]
	is tight with respect to the Gromov--Hausdorff metric. Hence there is a compact family $\cK$ of compact metric spaces such that, for all sufficiently large $n$, the rescaled simple core belongs to $\cK$ with probability at least $1-\delta$. By compactness of $\cK$ and \cite[Prop.~7.4.11(2)]{MR1835418}, the cardinalities of minimal $\epsilon$-nets are uniformly bounded over $\cK$.	Hence there exists an integer $K \ge 1$ such that for large enough $n$
	\begin{align}
		\label{eq:kbound}
		\Prb{K_n \le K} \ge 1 - \delta.
	\end{align}
	We may canonically select a partition
	\[
		A_{n,1},\ldots,A_{n,K_n}
	\]
	of the vertex set of $\left(\cS(\mM_n), (8n/9)^{-1/4} d_{\cS(\mM_n)}\right)$ into non-empty classes, each having diameter at most $2\epsilon$ in the rescaled metric.  For $1\le j\le K_n$ and $2\le i\le s_n$, let
	\[
		b_{n,i}^{(j)} := \frac{1}{2} \left|\{\text{endpoints of }e_{n,i}\text{ that belong to }A_{n,j}\}\right|.
	\]
	Thus $0\le b_{n,i}^{(j)}\le1$, and
	\begin{align*}
		\nu_n^\circ(A_{n,j}) &=\frac{ \sum_{i=2}^{s_n}W_{n,i}b_{n,i}^{(j)}}{\sum_{i=2}^{s_n}W_{n,i}},\\
		\mu_n^\circ(A_{n,j}) &= \frac{1}{s_n-1} \sum_{i=2}^{s_n}b_{n,i}^{(j)}.
	\end{align*}

	Conditionally on the simple core and on the distinguished component $\cV_{n,1}$,  the components $\cV_{n,2},\ldots,\cV_{n,s_n}$ are exchangeable. Let $\cF_n$ denote the sigma-field generated by the simple core, the distinguished component $\cV_{n,1}$, and the unordered family of the remaining components. Consequently, conditionally on $\cF_n$ the assignment of the unordered family of the remaining components (that is, all except $\cV_{n,1}$) to the edges $e_{n,2},\ldots,e_{n,s_n}$ is induced by a uniform random permutation.

	In particular, we have for every $2\le i\le s_n$,
	\[
	\Exb{W_{n,i} \mid \cF_n}
	=
	\frac{1}{s_n-1}
	\sum_{k=2}^{s_n}W_{n,k}.
	\]
	Hence for every $1 \le j \le K_n$
	\begin{align}
		\label{eq:fm}
		\Exb{\nu_n^\circ(A_{n,j})\mid\cF_n}
		&= \frac{1}{\sum_{k=2}^{s_n}W_{n,k}} \sum_{i=2}^{s_n} b_{n,i}^{(j)} \Exb{W_{n,i}\mid\cF_n} 
		\\
		&= \frac{1}{s_n-1} \sum_{i=2}^{s_n}b_{n,i}^{(j)} \nonumber\\
		&= \mu_n^\circ(A_{n,j}). \nonumber
	\end{align}
	Moreover, the variance formula for sampling without replacement yields
	\begin{align*}
		\Va{\nu_n^\circ(A_{n,j})\mid\cF_n} = \frac{ \sum_{i=2}^{s_n} \left( W_{n,i} - \frac{\sum_{k=2}^{s_n}W_{n,k}}{s_n-1} \right)^2
		}{\left(\sum_{i=2}^{s_n}W_{n,i}\right)^2}
		\frac{ \sum_{i=2}^{s_n} \left( b_{n,i}^{(j)} - \frac{\sum_{k=2}^{s_n}b_{n,k}^{(j)}}{s_n-1} \right)^2}{s_n-2}.
	\end{align*}
	Since $0\le b_{n,i}^{(j)}\le1$, we have with $\overline{b} := \frac{\sum_{k=2}^{s_n}b_{n,k}^{(j)}}{s_n-1}$ that
	\begin{align*}
		\sum_{i=2}^{s_n} \left(b_{n,i}^{(j)}-\overline{b}\right)^2
		&= \sum_{i=2}^{s_n} \left(b_{n,i}^{(j)}\right)^2 - (s_n-1)\overline{b}^{2} \\
		&\le \sum_{i=2}^{s_n}b_{n,i}^{(j)} - (s_n-1)\overline{b}^{2}\\
		&= (s_n-1)\overline{b}(1-\overline{b})\\
		&\le \frac{s_n-1}{4}.
	\end{align*}
	Hence, 
	\[
		\frac{\sum_{i=2}^{s_n} \left(b_{n,i}^{(j)}-\overline{b}\right)^2}{s_n-2} \le 1.
	\]
	Likewise,
	\[
		\sum_{i=2}^{s_n}\left(W_{n,i}-\frac{\sum_{k=2}^{s_n}W_{n,k}}{s_n-1}\right)^2 = \sum_{i=2}^{s_n}W_{n,i}^2 - (s_n - 1) \left( \frac{\sum_{k=2}^{s_n}W_{n,k}}{s_n-1}\right)^2 \le \sum_{i=2}^{s_n}W_{n,i}^2.
	\]
	Consequently,
	\begin{align}
		\label{eq:sm}
		\Va{\nu_n^\circ(A_{n,j})\mid\cF_n} &\le \frac{\sum_{i=2}^{s_n}W_{n,i}^2}{\left(\sum_{i=2}^{s_n}W_{n,i}\right)^2	}.
	\end{align}

	Note that $K_n$, the sets $A_{n,j}$, and the coefficients $b_{n,i}^{(j)}$ are $\cF_n$-measurable. For every constant $\eta>0$,~\eqref{eq:fm},~\eqref{eq:sm}, the conditional Chebyshev inequality and a union bound give 
	\begin{align*}
		\Prb{
			\max_{1\le j\le K_n} \left| \nu_n^\circ(A_{n,j}) - \mu_n^\circ(A_{n,j}) \right| > \eta  \middle| \cF_n } &\le \sum_{j=1}^{K_n} \Prb{ \left| \nu_n^\circ(A_{n,j}) - \mu_n^\circ(A_{n,j}) \right| > \eta \middle| \cF_n	} \\
		&\le \frac{K_n}{\eta^2}	\frac{\sum_{i=2}^{s_n}W_{n,i}^2}{\left(\sum_{i=2}^{s_n}W_{n,i}\right)^2}.
	\end{align*}
	By~\eqref{eq:sousim} and the fact that $\frac{\sum_{i=2}^{s_n}W_{n,i}^2}{\left(\sum_{i=2}^{s_n}W_{n,i}\right)^2} \le 1$ it follows that
	\begin{align}
		\label{eq:copro}
		\Prb{ K_n\le K,\, \max_{1\le j\le K_n} \left| \nu_n^\circ(A_{n,j}) - \mu_n^\circ(A_{n,j}) \right| > \eta } \le \frac{K}{\eta^2} \Exb{ \frac{\sum_{i=2}^{s_n}W_{n,i}^2}{\left(\sum_{i=2}^{s_n}W_{n,i}\right)^2}} \to 0.
	\end{align}
	
	Let $B$ be any subset of the vertex set of the simple core $\cS(\mM_n)$, and let $B^{3\epsilon}$ denote its $3\epsilon$-neighbourhood in the rescaled metric. If $A_{n,j}\cap B\neq\emptyset$, then $A_{n,j}\subset B^{3\epsilon}$, since $A_{n,j}$ has diameter at most $2\epsilon$. Hence, on the event $K_n\le K$,
	\begin{align*}
		\nu_n^\circ(B) &\le
		\sum_{\substack{1\le j\le K_n\\A_{n,j}\cap B\neq\emptyset}}
		\nu_n^\circ(A_{n,j})
		\\
		&\le \sum_{\substack{1\le j\le K_n\\A_{n,j}\cap B\neq\emptyset}} \mu_n^\circ(A_{n,j})
		+
		\sum_{j=1}^{K_n} \left|\nu_n^\circ(A_{n,j})	- \mu_n^\circ(A_{n,j})\right|\\
		&\le \mu_n^\circ(B^{3\epsilon}) + K \max_{1\le j\le K_n} \left|	\nu_n^\circ(A_{n,j})
		-
		\mu_n^\circ(A_{n,j})\right|.
	\end{align*}
	The same argument with $\nu_n^\circ$ and $\mu_n^\circ$ interchanged gives
	\[
		\mu_n^\circ(B) \le \nu_n^\circ(B^{3\epsilon}) + K \max_{1\le j\le K_n} \left| \nu_n^\circ(A_{n,j}) - \mu_n^\circ(A_{n,j})\right|.
	\]
	Consequently, by the definition of the Prokhorov distance,
	\begin{align*}
		\Prb{ K_n\le K,\,
			d_{\mathrm{P}}\left(\nu_n^\circ,\mu_n^\circ	\right)>3\epsilon}	\le \Prb{ K_n\le K,\,\max_{1\le j\le K_n}\left|			\nu_n^\circ(A_{n,j}) -	\mu_n^\circ(A_{n,j})\right| > \frac{3\epsilon}{K}}
		\to 0
	\end{align*}
	by~\eqref{eq:copro}. It follows from~\eqref{eq:kbound} that
	\begin{align*}
		\limsup_{n\to\infty}
		\Prb{d_{\mathrm{P}}\left( \nu_n^\circ, \mu_n^\circ \right)>3\epsilon } &\le \limsup_{n\to\infty}\Prb{K_n>K} \le\delta.
	\end{align*}
	Since $\epsilon,\delta>0$ were arbitrary, this proves
	\begin{align}
		\label{eq:proko} 
		d_{\mathrm{P}}\left(\nu_n^\circ, \mu_n^\circ \right) \convp 0.
	\end{align}

	Let $\eta_{e_{n,1}}$ denote the probability measure that assigns mass $1/2$ to each endpoint of $e_{n,1}$. We have
	\begin{align*}
		\nu_{\cS(\mM_n)}
		&=
		\frac{W_{n,1}}{k_n}\eta_{e_{n,1}}
		+
		\frac{k_n-W_{n,1}}{k_n}\nu_n^\circ,
		\\
		\mu_{\cS(\mM_n)}
		&=
		\frac{1}{s_n}\eta_{e_{n,1}}
		+
		\frac{s_n-1}{s_n}\mu_n^\circ.
	\end{align*}
	Consequently, by~\eqref{eq:attasize} and~\eqref{eq:sicosi}
	\begin{align*}
		d_{\mathrm{P}}(\nu_{\cS(\mM_n)},\nu_n^\circ)
		&\le
		\frac{W_{n,1}}{k_n} \convp 0,
		\\
		d_{\mathrm{P}}(\mu_{\cS(\mM_n)},\mu_n^\circ)
		&\le
		\frac{1}{s_n} \convp 0.
	\end{align*}
	By~\eqref{eq:proko} it follows that
	\begin{align*}
		d_{\mathrm{P}}\left(
		\nu_{\cS(\mM_n)},
		\mu_{\cS(\mM_n)}
		\right) \le o_p(1) + d_{\mathrm{P}}(\nu_n^\circ,\mu_n^\circ)
		\convp 0.
	\end{align*}
\end{proof}

Combining Lemmas~\ref{le:sourcesimplecorepart1} and~\ref{le:sourcesimplecorepart2} yields the scaling limit of the simple core equipped with its stationary distribution:

\begin{corollary}
	\label{co:corfppsimplecore}
	As $n \to \infty$,
	\begin{align*}
		\left( \cS(\mM_n), (8n/9)^{-1/4}d_{\cS(\mM_n)}, \mu_{\cS(\mM_n)} \right)
		\convd
		(\mathbf{M},d_{\mathbf{M}},\mu_{\mathbf{M}})
	\end{align*}
	in the Gromov--Hausdorff--Prokhorov sense.
\end{corollary}

\section{Transfer between different mixtures}

Our aim in this section is to transfer the scaling limit of $\cS(\mM_n)$ in  Corollary~\ref{co:corfppsimplecore} to a scaling limit of~$\cS(\mV_n)$.

\begin{lemma}
	\label{lem:stabletr}
	Let $h:\ndR\to[0,\infty)$ be a probability density that is continuous and whose Fourier transform has no zeros. Let $E$ be a Polish space. For every $n\ge 1$ and $k\in\ndZ$, let $Z_{n,k}$ be an $E$-valued random variable.
	
	Let $(A_n)_{n \ge 1}$ and $(B_n)_{n\ge 1}$ be integer-valued random
	variables, independent of $(Z_{n,k})_{n \ge 1, k \in \ndZ}$. Assume that there are constants $\mu_A,\mu_B,g_A,g_B>0$ such that for every $C>0$,
	\begin{align}
		\label{eq:lltaaa}
		\sup_{\substack{k\in\ndZ\\ |k - \mu_A n|\le C n^{2/3}}} \left| g_A n^{2/3}\Prb{A_n=k} - h\left( \frac{\mu_A n - k}{g_A n^{2/3}} \right) \right| \to 0,
	\end{align}
	and
	\begin{align}
		\label{eq:lltbbb}
		\sup_{\substack{k\in\ndZ\\ |k-\mu_B n|\le C n^{2/3}}} \left| g_B n^{2/3}\Prb{B_n=k} - h\left( \frac{\mu_B n-k}{g_B n^{2/3}} \right) \right| \to 0
	\end{align}
	as $n\to\infty$. Set $r = \frac{\mu_B}{\mu_A}$. For $\theta \in \ndR$ and all sufficiently large $n$, define
	\begin{align*}
		m_n(\theta) = \left\lfloor rn + \frac{\theta}{\mu_A}n^{2/3} \right\rfloor.
	\end{align*}
	Suppose that there is an $E$-valued random variable $Z$ such that, for every fixed $\theta \in \ndR$,
	\[
	Z_{n,A_{m_n(\theta)}} \convd Z.
	\]
	Then
	\[
	Z_{n,B_n} \convd Z.
	\]
\end{lemma}
\begin{proof}
	Fix $\theta\in\ndR$. Since $m_n(\theta) / n \to r > 0$, the integer $m_n(\theta)$ is positive for all sufficiently large $n$. Set
	\[
		a = g_A r^{2/3}, \qquad b = g_B,
	\]
	and define
	\[
		p_\theta(u) = \frac{1}{a} h\left( \frac{\theta-u}{a} \right), \qquad
		q(u) = \frac{1}{b} h\left(-\frac{u}{b}\right).
	\]
	Both $p_\theta$ and $q$ are probability densities. Moreover,
	\[
		p_\theta(u) = p_0(u - \theta),
	\]
	so the family $(p_\theta)_{\theta\in\ndR}$ is precisely the family of translates of $p_0$. For $k \in \ndZ$ set
	\[
		u_{n,k} = \frac{k-\mu_B n}{n^{2/3}}
	\]
	and the interval
	\[
		I_{n,k} = \left[ u_{n,k}-\frac12n^{-2/3}, u_{n,k}+\frac12n^{-2/3} \right[.
	\]
	The intervals $(I_{n,k})_{k\in\ndZ}$ form a partition of $\ndR$, and
	each interval has length $n^{-2/3}$. Write
	\[
		p_{n,\theta}(k) := \Prb{A_{m_n(\theta)}=k}, \qquad q_n(k) := \Prb{B_n=k}.
	\]
	Define piecewise constant functions $P_{n,\theta}$ and $Q_n$ by
	\[
		P_{n,\theta}(u) = n^{2/3} p_{n,\theta}(k), \qquad Q_n(u) = n^{2/3}q_n(k), \qquad u \in I_{n,k}.
	\]
	These are probability densities. Indeed,
	\[
		\int_{I_{n,k}}P_{n,\theta}(u)\,\mathrm{d}u = p_{n,\theta}(k), \qquad
		\int_{I_{n,k}}Q_n(u)\,\mathrm{d}u = q_n(k),
	\]
	and hence
	\[
		\int_{\ndR}P_{n,\theta}(u)\,\mathrm{d}u = \int_{\ndR}Q_n(u)\,\mathrm{d}u = 1.
	\]
	Fix $u \in \ndR$, and let $k = k(n,u)$ be the unique integer such that
	$u \in I_{n,k}$. By the definition of $I_{n,k}$,
	\[
		|u_{n,k}-u| \le \frac{1}{2} n^{-2/3},
	\]
	and therefore
	\[
		u_{n,k} \to u
	\]
	as $n \to \infty$. The definition of $m_n(\theta)$ gives
	\[
		m_n(\theta) = rn + \frac{\theta}{\mu_A} n^{2/3} + O(1).
	\]
	Multiplying by $\mu_A$ and using $\mu_A r=\mu_B$, we obtain
	\begin{align}
		\label{eq:trace}
		\mu_A m_n(\theta) = \mu_B n + \theta n^{2/3} + O(1).
	\end{align}
	Also, because $m_n(\theta)/n\to r$,
	\begin{align}
		\label{eq:trasca}
		m_n(\theta)^{2/3} = r^{2/3}n^{2/3}(1+o(1)).
	\end{align}
	Since
	\[
		k = \mu_B n + u_{n,k} n^{2/3},
	\]
	Equation \eqref{eq:trace} implies
	\[
		\mu_A m_n(\theta) - k = (\theta-u_{n,k})n^{2/3} + O(1).
	\]
	Together with \eqref{eq:trasca}, this gives
	\begin{equation}
		\label{eq:traama}
		\frac{\mu_A m_n(\theta)-k}
		{g_A m_n(\theta)^{2/3}} \to \frac{\theta-u}{g_A r^{2/3}} = \frac{\theta-u}{a},
	\end{equation}
	and
	\begin{equation}
		\label{eq:traapre}
		\frac{n^{2/3}}{g_A m_n(\theta)^{2/3}} \to \frac{1}{g_A r^{2/3}} = \frac{1}{a}.
	\end{equation}
	For fixed $u$ and $\theta$, the sequence in \eqref{eq:traama} remains bounded. Therefore the local limit estimate \eqref{eq:lltaaa} applies. Using also the continuity of  $h$, we obtain
	\begin{align*}
		P_{n,\theta}(u) &= n^{2/3}\Prb{A_{m_n(\theta)}=k} \\
		&= \frac{n^{2/3}}{g_A m_n(\theta)^{2/3}} \left(
		h\left( \frac{\mu_A m_n(\theta)-k}{g_A m_n(\theta)^{2/3}}\right) + o(1)
		\right)
		\\
		&\to
		\frac1a h\left(\frac{\theta-u}{a}\right)
		=
		p_\theta(u).
	\end{align*}
	Moreover, we have
	\[
		\frac{\mu_B n-k}{g_B n^{2/3}} = - \frac{u_{n,k}}{b} \to -\frac{u}{b}.
	\]
	Hence by \eqref{eq:lltbbb} 
	\begin{align*}
		Q_n(u)
		&=
		n^{2/3}\Prb{B_n=k}
		\\
		&= \frac{1}{g_B} \left(
		h\left( \frac{\mu_B n-k}{g_B n^{2/3}}\right) + o(1)
		\right)
		\\
		&\to \frac{1}{b} h\left(-\frac{u}{b}\right) = q(u).
	\end{align*}
	We have proved that, for every fixed $\theta$,
	\[
		P_{n,\theta}(u) \to p_\theta(u), \qquad Q_n(u) \to q(u)
	\]
	for every $u\in\ndR$. All the functions involved are probability densities. Hence we may apply Scheff\'e's
	lemma, yielding
	\begin{equation}
		\label{eq:profi}
		\|P_{n,\theta}-p_\theta\|_{L^1} \to 0,
		\qquad \|Q_n-q\|_{L^1} \to 0.
	\end{equation}
	For $g\in L^1(\ndR)$, define
	\[
		(\Pi_n g)(k) = \int_{I_{n,k}}g(u)\,\mathrm{d}u, \qquad k \in \ndZ.
	\]
	The map $\Pi_n: L^1(\ndR) \to \ell^1(\ndZ)$ is a contraction, since
	\begin{align*}
		\|\Pi_n g\|_{\ell^1} =
		\sum_{k\in\ndZ} \left| \int_{I_{n,k}}g(u)\,\mathrm{d}u \right|
		\le \sum_{k\in\ndZ} \int_{I_{n,k}}|g(u)|\,\mathrm{d}u = \|g\|_{L^1}.
	\end{align*}
	By construction,
	\[
		p_{n,\theta} = \Pi_nP_{n,\theta}, \qquad q_n = \Pi_nQ_n.
	\]
	It follows from \eqref{eq:profi} that
	\begin{align}
		\label{eq:daa}
		\|p_{n,\theta}-\Pi_np_\theta\|_{\ell^1} = \|\Pi_n(P_{n,\theta}-p_\theta)\|_{\ell^1}  \le \|P_{n,\theta}-p_\theta\|_{L^1}
		\to 0,
	\end{align}
	and
	\begin{align}
		\label{eq:dbb}
		\|q_n-\Pi_nq\|_{\ell^1} &= \|\Pi_n(Q_n-q)\|_{\ell^1} \le \|Q_n-q\|_{L^1} \to 0.
	\end{align}
	Let $\widehat{h}(x) = \int_{\ndR} h(t) e^{itx} \,\mathrm{d}t$ denote the Fourier transform of $h$. Since
	\[
		p_0(u) = \frac{1}{a} h\left(-\frac{u}{a}\right),
	\]
	the Fourier transform of $p_0$ is given by
	\[
		\widehat{p}_0(u) = \int_{\ndR} \frac{1}{a} h\left(-\frac{t}{a}\right) e^{itu} \,\mathrm{d}t = \widehat{h}(-au).
	\]
	We assumed that $\widehat{h}$ has no zeros, hence $\widehat{p}_0$ has no zeros. By Wiener's $L^1$ approximation theorem \cite{zbMATH03004303}, the linear span of the translates of $p_0$ is dense in $L^1(\ndR)$. In other words, since
	\[
		p_\theta(x) = p_0(x - \theta),
	\]
	it follows that the linear span of $(p_\theta)_{\theta \in \ndR}$ is dense in $L^1(\ndR)$. Hence for every $\epsilon>0$, there exist $\theta_1,\ldots,\theta_J\in\ndR$	and	$\alpha_1,\ldots,\alpha_J\in\ndR$ such that
	\begin{equation}
		\label{eq:wcaaa}
		\left\| q-\sum_{j=1}^J\alpha_jp_{\theta_j} \right\|_{L^1} <\epsilon.
	\end{equation}
	We have
	\begin{align*}
		\left\| q_n-\sum_{j=1}^J\alpha_jp_{n,\theta_j} \right\|_{\ell^1} &\le 
		\|q_n-\Pi_nq\|_{\ell^1} + \left\| \Pi_n\left( q-\sum_{j=1}^J\alpha_jp_{\theta_j} \right) \right\|_{\ell^1} \\
		&\quad\,+ \sum_{j=1}^J |\alpha_j| \|\Pi_np_{\theta_j}-p_{n,\theta_j}\|_{\ell^1}.
	\end{align*}
	The first summand tends to zero by~\eqref{eq:dbb}. The second summand is at most $\epsilon$ by \eqref{eq:wcaaa}. Every summand in the third term tends to zero by~\eqref{eq:daa}. Hence
	\begin{align}
		\label{eq:lattw}
		\limsup_{n\to\infty}\left\|q_n-\sum_{j=1}^J\alpha_jp_{n,\theta_j}\right\|_{\ell^1}\le \epsilon.
	\end{align}
	Let $F: E \to \ndR$ be bounded and continuous. Set
	\[
		b_{n,k} = \Ex{F(Z_{n,k})} - \Ex{F(Z)}.
	\]
	Then
	\begin{align}
		\label{eq:bnk}
		|b_{n,k}| \le 2\|F\|_\infty.
	\end{align}
	Since $A_{m_n(\theta)}$ is independent of $(Z_{n,k})_{n \ge 1, k \in \ndZ}$, we have
	\begin{align*}
		\sum_{k\in\ndZ}p_{n,\theta}(k)b_{n,k} &= \sum_{k\in\ndZ} \Prb{A_{m_n(\theta)}=k}	\bigl(\Ex{F(Z_{n,k})}-\Ex{F(Z)}\bigr)\\
		&=\Ex{F(Z_{n,A_{m_n(\theta)}})}	- \Ex{F(Z)}.
	\end{align*}
	We assumed that
	\[
		Z_{n,A_{m_n(\theta)}} \convd Z
	\]
	for every fixed $\theta$. Since $F$ is bounded and continuous, it follows that
	\begin{equation}
		\label{eq:stco}
		\sum_{k\in\ndZ}p_{n,\theta}(k)b_{n,k}
		\to 0
	\end{equation}
	as $n \to \infty$ for every fixed $\theta\in\ndR$. Analogously,  independence of $B_n$ from $(Z_{n,k})_{n \ge 1, k \in \ndZ}$ yields
	\[
		\Ex{F(Z_{n,B_n})}-\Ex{F(Z)} = \sum_{k\in\ndZ}q_n(k)b_{n,k}.
	\]
	Using~\eqref{eq:bnk} we obtain
	\begin{align*}
		\left|\sum_{k\in\ndZ}q_n(k)b_{n,k}\right|&\le \left|\sum_{k\in\ndZ}	\left(q_n(k) - \sum_{j=1}^J \alpha_jp_{n,\theta_j}(k)	\right)b_{n,k}\right|+\sum_{j=1}^J|\alpha_j|\left|\sum_{k\in\ndZ}p_{n,\theta_j}(k)b_{n,k}\right|\\
		&\le2\|F\|_\infty \left\|q_n-\sum_{j=1}^J\alpha_jp_{n,\theta_j}	\right\|_{\ell^1}+\sum_{j=1}^J|\alpha_j|\left| \sum_{k\in\ndZ}	p_{n,\theta_j}(k)b_{n,k}\right|.
	\end{align*}
	By \eqref{eq:stco} and \eqref{eq:lattw} it follows that
	\[
		\limsup_{n\to\infty} \left| \sum_{k\in\ndZ}q_n(k)b_{n,k}\right| \le 2\|F\|_\infty\epsilon.
	\]
	Since $\epsilon>0$ is arbitrary, this shows that
	\[
		\sum_{k\in\ndZ}q_n(k)b_{n,k} \to 0.
	\]
	Equivalently,
	\[
		\Ex{F(Z_{n,B_n})} \to \Ex{F(Z)}.
	\]
	This holds for every bounded continuous function $F:E\to\ndR$. Hence
	\[
		Z_{n,B_n}\convd Z.
	\]
	This completes the proof.
\end{proof}

\begin{corollary}
	\label{co:targetsico}
	As $n\to\infty$,
	\begin{align*}
		\left(\cS(\mV_n),	\left(\frac{3}{8n}\right)^{1/4}d_{\cS(\mV_n)}, \mu_{\cS(\mV_n)}\right)	\convd		(\mathbf{M},d_{\mathbf{M}},\mu_{\mathbf{M}}).
	\end{align*}
\end{corollary}
\begin{proof}
	For every integer $k \in \ndN \setminus \{2\}$, let $\mS_k$ denote a uniform rooted simple non-separable planar map with $k$ edges. We take the family $(\mS_k)_{k \in \ndN \setminus \{2\}}$ to be independent, and independent of the sequences of maps $(\mM_n)_{n \ge 1}$ and $(\mV_n)_{n \ge 1}$.
	
	Let $E$ be the Polish space of isometry classes of compact measured metric spaces equipped with the Gromov--Hausdorff--Prokhorov metric. For all $n\ge 1$ and $k \in \ndN \setminus \{2\}$, define the $E$-valued random variable
	\begin{align*}
		Z_{n,k}:= \left( \mS_k, \left(\frac{3}{8n}\right)^{1/4} d_{\mS_k}, \mu_{\mS_k} \right).
	\end{align*}
	We set $Z_{n,k}$ to an arbitrary placeholder value when $k=2$ or when $k \le 0$.
	
	We verify the assumptions of Lemma~\ref{lem:stabletr}. The density $h$ of the stable random variable $Z_{3/2}$ is continuous. Its Fourier transform is given by
	\begin{align*}
		\widehat{h}(t) = \Ex{\exp(itZ_{3/2})} = \exp((-it)^{3/2}),
		\qquad t\in\ndR,
	\end{align*}
	where the principal branch of  $z \mapsto z^{3/2}$ is used. Hence $\widehat{h}$ is integrable and has no zeros. 
	
	Proposition~\ref{pro:lltsourcesimplecore} and Proposition~\ref{pro:edsvm1} ensure that~\eqref{eq:lltaaa} and~\eqref{eq:lltbbb} hold with
	\begin{align*}
		A_n=\ed(\cS(\mM_n)),
		\quad
		B_n=\ed(\cS(\mV_n)),
		\quad
		\mu_A=\frac{4}{15},
		\quad
		\mu_B=\frac{4}{5},
		\quad
		g_A=g_{\cS},
		\quad
		g_B=g_{\cV}.
	\end{align*}
	Thus
	\begin{align*}
		r:=\frac{\mu_B}{\mu_A}=3,
	\end{align*}
	and, for each fixed $\theta\in\ndR$, the index appearing in Lemma~\ref{lem:stabletr} is
	\begin{align*}
		m_n(\theta) = \left\lfloor rn + \frac{\theta}{\mu_A}n^{2/3} \right\rfloor = \left\lfloor 3n+\frac{15\theta}{4}n^{2/3} \right\rfloor.
	\end{align*}
	Conditionally on their  number of edges, $\cS(\mM_n)$ and $\cS(\mV_n)$ are uniform non-separable simple planar maps. That is, for all $n, m$ for which the models are well-defined
	\begin{align*}
		\left( \cS(\mV_n) \mid \ed(\cS(\mV_n))=m \right) &\eqdist \mS_m, \\
		\left( \cS(\mM_n) \mid \ed(\cS(\mM_n))=m \right) &\eqdist \mS_m.
	\end{align*}
	Consequently, 
	\begin{align*}
		Z_{n, A_m} &\eqdist \left( \cS(\mM_m), \left(\frac{3}{8n}\right)^{1/4} d_{\cS(\mM_m)}, \mu_{\cS(\mM_m)} \right), \\
		Z_{n, B_n} &\eqdist \left( \cS(\mV_n), \left(\frac{3}{8n}\right)^{1/4} d_{\cS(\mV_n)}, \mu_{\cS(\mV_n)} \right).
	\end{align*}
	Fix $\theta \in \ndR$. We have $m_n(\theta) \sim 3n$ and hence
	\[
		(3/(8n))^{1/4} \sim (9/(8m_n(\theta)))^{1/4}.
	\]
	By Corollary~\ref{co:corfppsimplecore} we know that
	\begin{align*}
		\left( \cS(\mM_{m_n(\theta)}), (9 / (8m_n(\theta)))^{1/4}d_{\cS(\mM_{m_n(\theta)})}, \mu_{\cS(\mM_{m_n(\theta)})} \right) \convd (\mathbf{M},d_{\mathbf{M}},\mu_{\mathbf{M}})
	\end{align*}
	and hence
	\begin{align*}
		Z_{n,A_{m_n(\theta)}} \eqdist \left( \cS(\mM_{m_n(\theta)}), (3/(8n))^{1/4}d_{\cS(\mM_{m_n(\theta)})}, \mu_{\cS(\mM_{m_n(\theta)})} \right)
		\convd
		(\mathbf{M},d_{\mathbf{M}},\mu_{\mathbf{M}}).
	\end{align*}
	By Lemma~\ref{lem:stabletr} it follows that
	\begin{align*}
		Z_{n, B_n} \convd (\mathbf{M},d_{\mathbf{M}},\mu_{\mathbf{M}}).
	\end{align*}
	This completes the proof.
\end{proof}

\begin{proof}[Proof of Theorem~\ref{te:main1}]
	By Corollary~\ref{co:targetsico} we have
	\begin{align*}
		\left( \cS(\mV_n), \left(\frac{3}{8n}\right)^{1/4}d_{\cS(\mV_n)}, \mu_{\cS(\mV_n)} \right) \convd (\mathbf{M},d_{\mathbf{M}},\mu_{\mathbf{M}}).
	\end{align*}
	In particular, $(\cS(\mV_n),
	\left(\frac{3}{8n}\right)^{1/4}d_{\cS(\mV_n)})$ is tight with respect to the Gromov--Hausdorff metric. Hence the proof of Lemma~\ref{le:sourcesimplecorepart2} may be reproduced verbatim, yielding
	\begin{align*}
		d_{\mathrm{P}}\left( \nu_{\cS(\mV_n)}, \mu_{\cS(\mV_n)} \right) \convp 0.
	\end{align*}
	Hence
	\begin{align*}
		\left( \cS(\mV_n), \left(\frac{3}{8n}\right)^{1/4}d_{\cS(\mV_n)}, \nu_{\cS(\mV_n)} \right) \convd (\mathbf{M},d_{\mathbf{M}},\mu_{\mathbf{M}}).
	\end{align*}
	By Lemma~\ref{le:simplecorepart1} it follows that
	\begin{align*}
		\left(\mV_n, \left(\frac{3}{8n}\right)^{1/4} d_{\mV_n}, \mu_{\mV_n}\right) \convd (\mathbf{M}, d_{\mathbf{M}}, \mu_{\mathbf{M}}).
	\end{align*}
\end{proof}

\section{Diameter bounds for non-separable planar maps}

\subsection{Diameter moments of uniform planar maps}

\begin{lemma}
	\label{le:cudiam}
	For every $p>0$ there is a constant $C_p>0$ such that for all $n \ge 1$
	\begin{align*}
		\Ex{\Di(\mM_n)^p} \le C_p n^{p/4}.
	\end{align*}
\end{lemma} 
\begin{proof}
	Let $(\mM_n^\bullet,v_n^\star)$ be uniformly distributed among all (rooted) planar maps with $n$ edges and a marked vertex. This way, $\mM_n^\bullet$ assumes a given planar map $M$ with probability proportional to $\ve(M)$.
	
	The inverse Ambj{\o}rn--Budd bijection introduced in~\cite{MR3090757}, in the precise form of~\cite[Prop.~2.3]{MR3256874}, sends $(\mM_n^\bullet,v_n^\star)$ together with an independent uniform sign to a uniform rooted pointed quadrangulation $(\mQ_n,v_n^\star)$ with $n$ faces. In this coupling, any vertex of $\mM_n^\bullet$ is canonically identified with a vertex of $\mQ_n$, and the marked vertices get identified with each other.
	
	Since every such quadrangulation has $n+2$ vertices, forgetting the distinguished vertex of $(\mQ_n,v_n^\star)$ makes the resulting quadrangulation  $\mQ_n$ uniformly distributed among all quadrangulations with $n$ faces. Moreover,
	\begin{align*}
		d_{\mM_n^\bullet}(v,v_n^\star) = d_{\mQ_n}(v,v_n^\star)
	\end{align*}
	for every vertex $v$ of $\mM_n^\bullet$, see~\cite[Eq.~(2.1)]{MR3256874}. It follows that
	\begin{align*}
		\Di(\mM_n^\bullet) \le  2 \max_{v \in \mQ_n} d_{\mQ_n}(v,v_n^\star) \le 2\Di(\mQ_n) \le 4 \He(\mQ_n).
	\end{align*}
	Since all moments of $n^{-1/4}\He(\mQ_n)$ converge by~\cite[Cor.~3]{zbMATH02055268} it follows that for each $p>0$ there exists a constant $C_p'>0$ such that
	\begin{align}
		\label{eq:dipointed}
		\Ex{\Di(\mM_n^\bullet)^p} \le C_p' n^{p/4}
	\end{align}
	for all $n \ge 1$.
	
	The quasi-powers estimate for $\ve(\mM_n)$ in~\cite[Sec.~5]{MR4304258} allows us to apply~\cite[Lem.~7]{MR3071845}, yielding constants $c,C>0$ such that for all $n \ge 1$
	\begin{align}
		\label{eq:vbound}
		\Prb{\ve(\mM_n) \le n/4} \le C\exp(-cn).
	\end{align}
	A well-known consequence of Euler's formula and the dual map construction is that
	\begin{align*}
		\Ex{\ve(\mM_n)} =\frac{n+2}{2}.
	\end{align*}
	Thus, with $m_n$ denoting the number of planar maps with $n$ edges we have for any planar map $M$ with $n$ edges
	\begin{align*}
		\Prb{\mM_n^\bullet = M} = \frac{\ve(M)}{m_n \Exb{\ve(\mM_n)}} = \frac{2 \ve(M)}{m_n(n+2)}.
	\end{align*}
	It follows that for every non-negative functional $F$ of rooted planar maps we have
	\begin{align*}
		\Ex{F(\mM_n)} = \Exb{ \frac{n+2}{2\ve(\mM_n^\bullet)} F(\mM_n^\bullet) }.
	\end{align*}
	Applying this for
	\[
		F(M) = \Di(M)^p \one_{\{\ve(M)>n/4\}},
	\]
	and using~\eqref{eq:vbound} and the deterministic bound $\Di(M) \le \ed(M)$, we obtain
	\begin{align*}
		\Ex{\Di(\mM_n)^p} &= \Ex{\Di(\mM_n)^p\one_{\{\ve(\mM_n)>n/4\}}} + \Ex{\Di(\mM_n)^p\one_{\{\ve(\mM_n)\le n/4\}}} \\
		&\le 4\Ex{\Di(\mM_n^\bullet)^p} + n^p\Prb{\ve(\mM_n)\le n/4} \\
		&\le 4C_p' n^{p/4} + n^p C\exp(-cn).
	\end{align*}
	It follows that there exists a constant $C_p>0$ such that
	\begin{align*}
		\Ex{\Di(\mM_n)^p} \le C_p n^{p/4}
	\end{align*}
	for all $n \ge 1$.
\end{proof}

\subsection{Averaged diameter moments of non-separable maps}

\begin{proposition}
	\label{pro:andm}
	For every $p>0$ there is a constant $E_p>0$ such that for all $n\ge 1$
	\begin{align*}
		\sum_{m=1}^{n} \Ex{\Di(\mV_m)^p} \le E_p n^{1+p/4}.
	\end{align*}
\end{proposition}
\begin{proof}
	Every block of a connected graph is isometrically embedded, since a path that leaves the block through some cutvertex can only return to the block through the same cutvertex. Thus, $\Di(\cV(M)) \le \Di(M)$ for any planar map $M$. It follows that for any $N \ge 1$ by Lemma~\ref{le:cudiam} 
	\begin{align}
		\label{eq:codiam}
		\sum_{m=1}^{N} \Prb{\ed(\cV(\mM_N))=m}\Ex{\Di(\mV_m)^p} = \Ex{\Di(\cV(\mM_N))^p} \le \Ex{\Di(\mM_N)^p} \le C_pN^{p/4}.
	\end{align}
	Let $h$ and $g_{\cM}$ be as in Proposition~\ref{pro:lltvcore}. Since $h$ is continuous and integrates to one, there are constants $x_0 \in \ndR$ and $a,c_0>0$ such that
	\begin{align*}
		h(x) \ge c_0 \qquad \text{whenever } |x-x_0|\le 2a.
	\end{align*}
	Set
	\begin{align*}
		I_N := \left\{ m\in\ndZ \ \Big\vert\ \left| \frac{N/3-m}{g_{\cM}N^{2/3}}-x_0 \right| \le a \right\}.
	\end{align*}
	By the local limit theorem in Proposition~\ref{pro:lltvcore} it follows that there is a constant $c>0$ such that for all sufficiently large $N$,
	\begin{align*}
		\inf_{m\in I_N} \Prb{\ed(\cV(\mM_N))=m} \ge cN^{-2/3}.
	\end{align*}
	By~\eqref{eq:codiam} it follows that for all sufficiently large $N$
	\begin{align}
		\label{eq:wimwid}
		\sum_{m\in I_N}\Ex{\Di(\mV_m)^p} \le c^{-1} C_p N^{2/3+p/4}.
	\end{align}
	The centre and half-width of $I_N$ are given by
	\[
		c_N := \frac{N}{3}-g_{\cM}x_0N^{2/3}, \qquad w_N := a g_{\cM}N^{2/3}.
	\]
	Since
	\[
		c_{N+1}-c_N = \frac{1}{3} + O(N^{-1/3}),
	\]
	the sequence $(c_N)_N$ is strictly increasing for all sufficiently large $N$, with increments bounded away from zero and infinity. For all sufficiently large $r$ we have $c_{2r}<r$ and $c_{7r}>2r$ and $w_N \asymp r^{2/3}$  for $N\in[2r,7r]$. It follows that one may select integers $N_1, \ldots, N_J \in [2r,7r]$ so that $J=O(r^{1/3})$ and
	\begin{align}
		[r,2r]\cap\ndZ \subset \bigcup_{j=1}^{J}I_{N_j}.
	\end{align}
	By~\eqref{eq:wimwid} it follows that for all sufficiently large $r$
	\begin{align*}
		\sum_{r\le m\le2r} \Ex{\Di(\mV_m)^p} &\le \sum_{j=1}^{J} \sum_{m\in I_{N_j}}\Ex{\Di(\mV_m)^p} \le O(1) r^{1/3}r^{2/3+p/4} = O(r^{1+p/4})
	\end{align*}
	with uniform $O$-terms. Hence, taking $j_0$ sufficiently large we have with uniform $O$-terms
	\[
		\sum_{m=2^{j_0}}^{n} \Ex{\Di(\mV_m)^p} \le \sum_{j_0 \le j \le \log_2(n)} O(2^{j + jp/4}) \le O(n^{1 + p/4}).
	\]
	Taking into account the first $2^{j_0}$ summands it follows that there exists a constant $E_p>0$  such that for all $n \ge 1$
	\begin{align*}
		\sum_{m=1}^{n} \Ex{\Di(\mV_m)^p} \le E_p n^{1+p/4}.
	\end{align*}
\end{proof}

\section{Negligibility of small-degree decorations}

\label{sec:negli}

The purpose of this section is to verify a diameter condition required in the decorated stable tree regime of Theorem~\ref{te:main2}. We state this condition in the more general setting of marked trees, since it is of independent interest and may be applied in other contexts as well.

For a tree $T$ with root $o$ and a vertex $v\in T$, let $[o,v]_T$ denote the set of vertices on the unique path from $o$ to $v$, including both endpoints.  Let $\mT$ be a 	Bienaym\'e--Galton--Watson tree with offspring distribution $\xi$.  For every $N\ge1$ satisfying $\Pr{|\mT|=N}>0$, let $\mT_N$ have the conditional law of $\mT$ given $|\mT|=N$. Let $o$ denote its root. For each $k\ge0$, let $\mD_k$ be a non-negative random variable, with $\mD_0=0$.  Conditionally on $\mT_N$, assign independent marks $(\mD_v)_{v\in\mT_N}$ such that $\mD_v$ is distributed like $\mD_{d_{\mT_N}^+(v)}$ for each $v \in \mT_N$.  

Recall that for $\nu \ge 1$ we defined the size-biased random variable $\widehat\xi$ with 
\[
	\Prb{\widehat{\xi}=k} = k \Pr{\xi=k}
\]
for all integers $k \ge 0$.

\begin{proposition}
	\label{pro:promise}
	Suppose that $\nu=1$ and that condition~\eqref{eq:bwregvar} holds for some $1<\alpha<2$. Conditionally on $\widehat\xi=k$, let $\mD_{\widehat\xi}$ have the law of $\mD_k$. Define $\mD_\xi$ analogously.  Suppose that there are constants $C, \gamma_{\mathrm{dec}}>0$ and an integer $p\ge 2$ with
	\begin{align*}
		\sum_{k=1}^{m}\Ex{\mD_k^p} \le 	Cm^{1+p\gamma_{\mathrm{dec}}}
	\end{align*}
	for all large enough integers $m$. Then
	\begin{align*}
		\sup_{\substack{r>0\\\Prb{r/2<\widehat\xi\le r}>0}}
		\Exb{ \left(\frac{\mD_{\widehat\xi}}{r^{\gamma_{\mathrm{dec}}}}\right)^p \ \middle|\ r/2<\widehat\xi\le r } &<\infty,
		\\
		\sup_{\substack{r>0\\\Prb{r/2<\xi\le r}>0}} \Exb{ \left(\frac{\mD_\xi}{r^{\gamma_{\mathrm{dec}}}}\right)^p \ \middle|\ r/2<\xi\le r } &<\infty. 
	\end{align*}
\end{proposition}
\begin{proof}
	 By~\eqref{eq:bwregvar} we have
	\begin{align}
		\label{eq:xidensity}
		\Prb{\xi=j} \sim (2d_\omega)^{1+\alpha} L\left(j\right)j^{-1-\alpha}
		\qquad j \to \infty, \qquad j \in 2d_\omega \ndN.
	\end{align}
	By Karamata's theorem for sums over lattices it follows that as $r \to \infty$
	\begin{align*}
		\Pr{r/2<\widehat\xi\le r} \sim (2d_\omega)^\alpha \frac{2^{\alpha-1}-1}{\alpha-1} L(r)r^{1-\alpha}, \quad\quad
		\Pr{r/2<\xi\le r} \sim (2d_\omega)^\alpha \frac{2^\alpha-1}{\alpha} L(r)r^{-\alpha}.
	\end{align*}
	By the uniform convergence theorem for slowly varying functions it follows that there is a constant $C_1>0$ such that for all sufficiently large~$r$
	\begin{align*}
		\sup_{\substack{r/2<k\le r\\k\in2d_\omega\ndN}} \Prb{\widehat\xi=k} &\le C_1L(r)r^{-\alpha}, & \sup_{\substack{r/2<k\le r\\k\in2d_\omega\ndN}} \Prb{\xi=k} &\le C_1L(r)r^{-1-\alpha}.
	\end{align*}
	Hence, using $\sum_{k=1}^{m}\Ex{\mD_k^p} \le Cm^{1+p\gamma_{\mathrm{dec}}}$ we obtain
	\begin{align*}
		\Exb{ \left(\frac{\mD_{\widehat\xi}}{r^{\gamma_{\mathrm{dec}}}}\right)^p \one_{\{r/2<\widehat\xi\le r\}} }
		&= r^{-p\gamma_{\mathrm{dec}}} \sum_{\substack{r/2<k\le r\\k\in2d_\omega\ndN}} k\Prb{\xi=k}\Ex{\mD_k^p} \\
		&\le C_1 L(r)r^{-\alpha-p\gamma_{\mathrm{dec}}} \sum_{k\le r}\Ex{\mD_k^p} \\
		&\le C C_1 L(r)r^{1-\alpha}.
	\end{align*}
	An analogous calculation yields
	\[
	\Exb{\left(\frac{\mD_\xi}{r^{\gamma_{\mathrm{dec}}}}\right)^p \one_{\{r/2<\xi\le r\}} } \le C C_1 L(r)r^{-\alpha}.
	\]
	It follows that
	\begin{align*}
		\sup_{r>0}
		\Exb{\left(\frac{\mD_{\widehat\xi}}{r^{\gamma_{\mathrm{dec}}}}\right)^p \ \middle|\ r/2<\widehat\xi\le r } <\infty, \qquad 
		\sup_{r>0} \Exb{ \left(\frac{\mD_\xi}{r^{\gamma_{\mathrm{dec}}}}\right)^p \ \middle|\ r/2<\xi\le r } <\infty,
	\end{align*}
	with $r$ restricted to values with $\Prb{r/2<\xi\le r}>0$.
\end{proof}

The next lemma is a version of~\cite[Prop.~3.2]{zbMATH07790315} with less strict requirements. The main difference is that instead of uniform bounds on $\Exb{\mD_k^p}$ for a specific $p>0$ we assume bounds on the averages $m^{-1} \sum_{k=1}^{m}\Ex{\mD_k^p}$. Proposition~\ref{pro:promise} ensures that the proof may be carried out fully analogously. We omit the details.
\begin{lemma}
	\label{le:smalldeco} 
	Suppose that $\nu=1$ and that condition~\eqref{eq:bwregvar} holds for some $1<\alpha<2$.  Let $\gamma_{\mathrm{dec}}>\alpha-1$.  Suppose that for some integer $p$ with
	\begin{align*}
		p >	\max\left(2, \frac{4\alpha}{2\gamma_{\mathrm{dec}}+1-\alpha} \right)
	\end{align*}
	there is a constant $C>0$ such that
	\begin{align*}
		\sum_{k=1}^{m}\Ex{\mD_k^p} \le 	Cm^{1+p\gamma_{\mathrm{dec}}}
	\end{align*}
	for all $m\ge1$.  Then, for every $\epsilon>0$,
	\begin{align*}
		\lim_{\delta\downarrow0}
		\sup_{\substack{N\ge1\\\Pr{|\mT|=N}>0}}
		\Prb{\max_{v\in\mT_N} \sum_{u\in[o,v]_{\mT_N}} \mD_u \one_{\{d_{\mT_N}^+(u)\le\delta b_N\}} > \epsilon b_N^{\gamma_{\mathrm{dec}}} } = 0.
	\end{align*}
\end{lemma}

\section{Scaling limits of block-weighted planar maps}
\label{sec:bwscaling}

Throughout this section, all limits are taken as $n \to \infty$, restricted to $n \equiv 0 \mod d_\omega$, and we write
\[
	N := 2n + 1.
\]
A $\cV$-enriched tree $(T, \beta_T)$ is a pair of a finite plane tree $T$ and a function $\beta_T$ that assigns to each vertex $v \in T$ a non-separable planar map $\beta_T(v)$ with $d^+_T(v)$ corners. Any such tree corresponds bijectively to a planar map. It is constructed by starting with the non-separable planar map $\beta_T(o)$ corresponding to the root $o$ of $T$. The corners of $\beta_T(o)$ may be listed in a canonical order, and we attach the map corresponding to the $\cV$-enriched fringe subtree at the $i$th child of $o$ to the $i$th corner of $\beta_T(o)$. 

Given $T$, we may form the canonical $\cV$-decoration by assigning to each vertex  $v \in T$ an independent uniformly at random selected non-separable planar map with $d^+_T(v)$ many corners. This makes the tree $T$ with its canonical $\cV$-decoration a random $\cV$-enriched tree.

Let $\mT_N$ denote the simply generated tree corresponding to the weight-sequence $[z^i]\Phi(z)$, $i \ge 0$. Let $\beta_N$ denote its canonical decoration. The random $\cV$-enriched plane tree $(\mT_N, \beta_N)$ corresponds to the planar map $\mM_n^\omega$, see~\cite[Lem. 6.1]{MR4132643} and~\cite{zbMATH07004718}.

When $\nu>0$ the tree $\mT_N$ is distributed like the $\xi$-Bienaym\'e--Galton--Watson tree $\mT$ conditioned on having $N$ vertices. We let $\beta$ denote the canonical decoration on $\mT$. This way,  $(\mT, \beta)$ corresponds to a planar map that we denote by $\mM^\omega$.

Every non-root vertex of $\mT_N$ corresponds canonically to one corner of $\mM_n^\omega$. Let
\[
	\pi_N:\mT_N\longrightarrow\mM_n^\omega
\]
map a non-root tree vertex to the map vertex incident to the corresponding corner, and map the root $o$ of $\mT_N$ to the vertex incident to the root corner. Define the probability measure
\begin{align}
	\mu_N^{\mathrm{corn}} := \frac{1}{N-1} \sum_{v\in\mT_N\setminus\{o\}}\delta_v.
\end{align}
Because the $N-1=2n$ non-root tree vertices correspond bijectively to the corners of $\mM_n^\omega$, the stationary distribution $\mu_{\mM_n^\omega}$ is the push-forward measure of $\mu_N^{\mathrm{corn}}$ along $\pi_N$. The total variation distance between $\mu_N^{\mathrm{corn}}$ and the uniform probability measure on all $N$ vertices of $\mT_N$ is equal to $1/N$ and hence tends to zero as $n \to \infty$. 

In the case $\nu\ge1$, recall that we defined a random non-separable planar map $\mV_{\widehat\xi}$ so that conditionally on $\widehat{\xi}=2m$ we have that $\mV_{\widehat{\xi}}$ is uniformly distributed among all non-separable planar maps with $m$ edges.

\subsection{The finite variance and stable case}

The following lemma is an extension of~\cite[Thm. 6.62]{MR4132643}.

\begin{lemma}
	\label{le:blot}
	Suppose that \eqref{eq:bwfvv} holds, or that $\nu=1$ and Condition~\eqref{eq:bwregvar} holds for some $5/4<\alpha\le2$.
	Then $\kappa = \Exb{d(\mV_{\widehat\xi})} \in ]0, \infty[$ and for $r_N:=\frac{N}{b_N}$
	\begin{align*}
		d_{\mathrm{GHP}} \left( \left( \mM_n^\omega, r_N^{-1}d_{\mM_n^\omega}, \mu_{\mM_n^\omega} \right), \left( \mT_N, \kappa r_N^{-1}d_{\mT_N}, \mu_N^{\mathrm{corn}} \right) \right) \convp 0.
	\end{align*}
\end{lemma}
\begin{proof}
	For ease of notation, we refer to Condition~\eqref{eq:bwfvv} as the \emph{finite variance case}  and to the case $\nu=1$ with Condition~\eqref{eq:bwregvar} for some $5/4<\alpha\le2$ and additionally $\Va{\xi}=\infty$ as the \emph{stable case}. Recall that in the finite variance case we set $\alpha=2$.
	
	We use the following consequence of~\cite[Thm.~3.7]{MR3318042}: for every sufficiently small $\epsilon>0$, there are constants $c_1,c_2,c_3>0$ such that
	\begin{align}
		\label{eq:diamn}
		\Prb{ \Di(\mV_m)>m^{1/4+\epsilon} } \le c_1 \exp(-c_2 m^{c_3})
	\end{align}
	for all $m \ge 1$. In particular, since $\Di(\mV_m) \le m$, we have for any $p \ge 1$ and any $\epsilon>0$
	\[
		\sup_{m \ge 1} \Ex{\Di(\mV_m)^p} m^{- p/4 - \epsilon} < \infty.
	\]
	Consequently, for any $1< p<4 (\alpha-1)$ we have
	\begin{align}
		\label{eq:pelk}
		\Exb{d(\mV_{\widehat\xi})^p} < \infty. 	
	\end{align}
	Indeed, in the finite variance case this follows from $\Ex{\widehat{\xi}} < \infty$ and in the  stable case this follows from  $\Pr{\widehat{\xi}=2k}$ varying regularly with index~$-\alpha$. In particular, $0<\kappa<\infty$ in both cases. Set
	\[
		\Delta_N^{\mathrm{out}} := \max_{v\in\mT_N}d_{\mT_N}^+(v), \qquad \Delta_N^{\mathrm{dec}} := \max_{v\in\mT_N}\Di(\beta_N(v)).
	\]
	Then there exists $0<q<1$ with
	\begin{align}
		\label{eq:oprn}
		\Delta_N^{\mathrm{dec}} = o_p(r_N^q).
	\end{align}
	Indeed, in the finite variance case, we have $\Delta_N^{\mathrm{out}} = o_p(r_N)$ by~\cite[Eq. (9.4)]{MR2908619}. The tree $\mT_N$ has $N$ vertices, hence by a union bound and~\eqref{eq:diamn} it follows that~\eqref{eq:oprn} holds for any $1/2<q<1$. In the stable case, the functional convergence established in~\cite{MR1964956} for the \L{}ukasiewicz path of~$\mT_N$ rescaled by $b_N^{-1}$ implies $\Delta_N^{\mathrm{out}} = O_p(b_N)$. Hence for any $\epsilon_1, \epsilon_2 >0$ we may choose $M>0$ sufficiently large such that for all large enough $N$ and all $1 > q >1 / (4(\alpha-1))$ (possible since $5/4<\alpha\le2$) we have
	\begin{align*}
		\Prb{\Delta_N^{\mathrm{dec}} \ge \epsilon_1 r_N^q} &\le \epsilon_2 + \Prb{\Delta_N^{\mathrm{dec}} \ge \epsilon_1 r_N^q, \Delta_N^{\mathrm{out}} \le M b_N} \\
		&\le \epsilon_2 + N \sup_{1 \le m \le M b_N / 2} \Prb{\Di(\mV_{m}) \ge \epsilon_1 r_N^q} \\
		&= \epsilon_2 + N \sup_{\epsilon_1 r_N^q \le m \le M b_N / 2} \Prb{\Di(\mV_{m}) \ge \epsilon_1 r_N^q}.
	\end{align*}
	Since $b_N$ is regularly varying with index $1/\alpha$ and $r_N=N/b_N$, it follows by $5/4 < \alpha \le 2$ and~\eqref{eq:diamn} that 
	\[
		 N \sup_{\epsilon_1 r_N^q \le k \le M b_N} \Prb{\Di(\mV_k) \ge \epsilon_1 r_N^q} \to 0.
	\]
	Since $\epsilon_1, \epsilon_2>0$ were arbitrary, this verifies~\eqref{eq:oprn}.

	For a vertex $v\in\mT_N$ at height $h$, write
	\[
		o = v_0, v_1, \ldots, v_h =v
	\]
	for the vertices on the path $[o,v]_{\mT_N}$ from the root $o$ to $v$. Any geodesic in $\mM_n^\omega$ from $\pi_N(o)$ to $\pi_N(v)$ passes through the blocks $\beta_N(v_0), \ldots, \beta_N(v_{h-1})$, possibly via paths of length $0$. Let $J_0(v), \ldots, J_{h-1}(v) \ge 0$ denote the lengths of these paths. This way,
	\begin{align}
		\label{eq:dov}
		d_{\mM_n^\omega}(\pi_N(o), \pi_N(v)) = \sum_{j=0}^{h-1} J_j(v).
	\end{align}
	Set 
	\begin{align*}
		R_N := \max_{v\in\mT_N\setminus\{o\}} \max_{0\le i < \he_{\mT_N}(v)} \left| \sum_{j=i}^{\he_{\mT_N}(v)-1} \left(J_j(v)-\kappa\right)\right|.
	\end{align*}
	We are going to prove
	\begin{align}
		\label{eq:proveme}
		R_N=o_p(r_N).
	\end{align}
	Dwass' formula and the local limit theorem~\cite[Theorem~4.2.1]{MR0322926} yield
	\begin{align}
		\label{eq:tnbound}
		\Prb{|\mT|=N} = \frac{1}{N}	\Prb{\xi_1+\ldots+\xi_N=N-1} \asymp(N b_N)^{-1}.
	\end{align}
	By the functional convergence established in~\cite{MR1964956} for the height process of $\mT_N$ we have
	\begin{align}
		\label{eq:htnbound}
		\He(\mT_N)=O_p(r_N).
	\end{align}
	For every non-negative measurable functional $F$ of a random vector of length $h$, the spinal decomposition (see for example~\cite[Appendix~A.3]{zbMATH07790315}) yields
	\begin{align}
		\label{eq:spinal}
		\Exb{ \sum_{\substack{v\in\mT\\ |v|=h}} F\bigl((J_j(v))_{0\le j<h}\bigr) } = \Exb{ F\bigl((D_j)_{0\le j<h}\bigr) },
	\end{align}
	where $D_0,D_1,\ldots$ denote independent copies of $d(\mV_{\widehat\xi})$. Hence, with $0<q<1$ as in~\eqref{eq:oprn} and $1< p< 2$ as in~\eqref{eq:pelk} (possible since $\alpha>5/4$) we have for any $\epsilon>0$ and $A>0$ by~\eqref{eq:tnbound}
	\begin{align}
		\label{eq:mlbound}
		&\Prb{ R_N > \epsilon r_N, \He(\mT_N) \le A r_N, \Delta_N^{\mathrm{dec}}\le r_N^q } \\
		&\qquad \le \frac{1}{\Prb{|\mT|=N}} \sum_{1 \le h\le Ar_N} \Prb{ \max_{0\le i<h} \left| \sum_{j=i}^{h-1}(D_j-\kappa) \right|	> \epsilon r_N, \max_{0\le j<h}D_j\le r_N^q} \nonumber \\
		&\qquad \le O(N b_N) \sum_{1 \le h\le Ar_N} \Prb{ \max_{1\le i \le h} \left| \sum_{j=1}^{i}(D_j-\kappa) \one_{D_j \le r_N^q} \right|	> \epsilon r_N}.  \nonumber
	\end{align}
	Note that
	\[
		0 = \Exb{ d(\mV_{\widehat\xi}) - \kappa } = \Exb{ (d(\mV_{\widehat\xi})-\kappa)\one_{\{d(\mV_{\widehat\xi})\le r_N^q\}} } + \Exb{(d(\mV_{\widehat\xi})-\kappa)\one_{\{d(\mV_{\widehat\xi}) > r_N^q\}}}.
	\]
	On the event $d(\mV_{\widehat\xi}) > r_N^q$ we have $d(\mV_{\widehat\xi})^{p-1} > r_N^{q(p-1)}$  and hence $d(\mV_{\widehat\xi}) < d(\mV_{\widehat\xi})^p / r_N^{q(p-1)}$. Thus, for large enough $N$
	\begin{align*}
		\left|\Exb{ (d(\mV_{\widehat\xi})-\kappa)\one_{\{d(\mV_{\widehat\xi})\le r_N^q\}} } \right|
		&= \Exb{(d(\mV_{\widehat\xi})-\kappa)\one_{\{d(\mV_{\widehat\xi}) > r_N^q\}}}\\
		&\le \Exb{ d(\mV_{\widehat\xi})\one_{\{d(\mV_{\widehat\xi}) > r_N^q\}} } \\
		&\le r_N^{-q(p-1)}\Ex{d(\mV_{\widehat\xi})^p}.
	\end{align*}
	Likewise, 
	\begin{align*}
		\Va{ (d(\mV_{\widehat\xi})-\kappa)\one_{\{d(\mV_{\widehat\xi})\le r_N^q\}} } &\le \Exb{(d(\mV_{\widehat\xi})-\kappa)^2\one_{\{d(\mV_{\widehat\xi})\le r_N^q\}}} \\
		&\le \Exb{	d(\mV_{\widehat\xi})^2\one_{\{d(\mV_{\widehat\xi})\le r_N^q\}} } +\kappa^2 \\
		&\le r_N^{q(2-p)}\Ex{d(\mV_{\widehat\xi})^p} +\kappa^2.
	\end{align*}
	Finally, $(d(\mV_{\widehat\xi})-\kappa)\one_{\{d(\mV_{\widehat\xi})\le r_N^q\}}$ lies in the interval $[- \kappa, r_N^q]$ and so does its first moment. Hence 
	\begin{align*}
		\left| (d(\mV_{\widehat\xi})-\kappa)\one_{\{d(\mV_{\widehat\xi})\le r_N^q\}} - \Exb{ (d(\mV_{\widehat\xi})-\kappa)\one_{\{d(\mV_{\widehat\xi})\le r_N^q\}} } \right| \le 2 \max(\kappa, r_N^q)
	\end{align*}
	almost surely.
	
	By Bernstein's inequality  it follows that for all large enough $N$ we have uniformly  for all $1 \le h\le Ar_N$ and $1 \le i \le h$ 
	\begin{align*}
		&\Prb{ \left| \sum_{j=1}^{i}(D_j-\kappa) \one_{D_j \le r_N^q} \right|	> \epsilon r_N} \\
		&\le \Prb{ \left| \sum_{j=1}^{i}\left( (D_j-\kappa) \one_{D_j \le r_N^q} - \Exb{ (d(\mV_{\widehat\xi})-\kappa)\one_{\{d(\mV_{\widehat\xi})\le r_N^q\}}} \right)\right|	> \epsilon r_N - i r_N^{-q(p-1)}\Ex{d(\mV_{\widehat\xi})^p}} \\
		&\le 2 \exp\left( - \frac{(1/2+o(1)) \epsilon^2 r_N^2}{r_N^{1+q(2-p)}(\Ex{d(\mV_{\widehat\xi})^p}+o(1)) + (2/3 + o(1))r_N^{q+1} \epsilon   } \right).
	\end{align*}
	By~\eqref{eq:mlbound} and a union bound it follows that
	\begin{align*}
		\Prb{ R_N > \epsilon r_N, \He(\mT_N) \le A r_N, \Delta_N^{\mathrm{dec}}\le r_N^q } \to 0.
	\end{align*}
	Since $\epsilon>0$ was arbitrary, it follows by~\eqref{eq:oprn} and~\eqref{eq:htnbound} that $R_N = o_p(r_N)$. This completes the verification of~\eqref{eq:proveme}.
		
	Let $u,v\in\mT_N$, and let $w$ be their most recent common ancestor. Any path in $\mM_n^\omega$ from $\pi_N(u)$ to $\pi_N(v)$ passes through the block $\beta_N(w)$. It follows that
	\[
		d_{\mM_n^\omega}(\pi_N(u), \pi_N(v)) = d_{\mM_n^\omega}(\pi_N(w), \pi_N(u)) + d_{\mM_n^\omega}(\pi_N(w), \pi_N(v)) + R
	\]
	for an error $R$ satisfying $|R| \le 2 \Delta_N^{\mathrm{dec}}$. By~\eqref{eq:proveme} and~\eqref{eq:oprn} it follows that
	\[
		\sup_{u,v \in \mT_N} |d_{\mM_n^\omega}(\pi_N(u), \pi_N(v)) - \kappa d_{\mT_N}(u,v)| =  o_p(r_N).
	\]
	Since the stationary distribution $\mu_{\mM_n^\omega}$ is the push-forward measure of $\mu_N^{\mathrm{corn}}$ along $\pi_N$, it follows that
	\[
		d_{\mathrm{GHP}} \left( \left( \mM_n^\omega, r_N^{-1}d_{\mM_n^\omega}, \mu_{\mM_n^\omega} \right), \left( \mT_N, \kappa r_N^{-1}d_{\mT_N}, \mu_N^{\mathrm{corn}} \right) \right) \convp 0.
	\]
\end{proof}

\begin{corollary}
	Suppose that \eqref{eq:bwfvv} holds, or that $\nu=1$ and Condition~\eqref{eq:bwregvar} holds for some $5/4<\alpha\le2$. Then
	\[
	\left( \mM_n^\omega, \frac{2^{1/\alpha-1}}{\kappa}\frac{b_{n}}{n} d_{\mM_n^\omega}, \mu_{\mM_n^\omega} \right) \convd (\mathbf{T}_\alpha, d_{\mathbf{T}_\alpha}, \mu_{\mathbf{T}_\alpha}).
	\]
\end{corollary}
\begin{proof}
	 The invariance principles for conditioned Bienaym\'e--Galton--Watson trees~\cite{MR1964956} yield
	\begin{align*}
		\left(\mT_N,\frac{b_N}{N}d_{\mT_N},\mu_N^{\mathrm{corn}}\right) \convd (\mathbf{T}_\alpha,d_{\mathbf{T}_\alpha},\mu_{\mathbf{T}_\alpha}).
	\end{align*}
	By Lemma~\ref{le:blot} it follows that $\kappa <\infty$ and
	\begin{align*}
		\left(\mM_n^\omega,\frac{b_N}{\kappa N}d_{\mM_n^\omega},\mu_{\mM_n^\omega}\right) \convd (\mathbf{T}_\alpha,d_{\mathbf{T}_\alpha},\mu_{\mathbf{T}_\alpha}).
	\end{align*}
	The sequence $(b_n)_n$ is regularly varying with index $1/\alpha$. Consequently, $b_N / b_n \to2^{1/\alpha}$ since $N \sim 2n$. This implies the asserted scaling limit.
\end{proof}

Thus, in the present regime the planar map $\mM_n^\omega$ has asymptotically a tree-like shape. Let us also mention that Brownian tree limits of face-weighted models of planar maps have been observed in pioneering work~\cite{MR3342658} via condensation phenomena of simply generated trees.

\subsection{The $\alpha=5/4$ phase transition}

\begin{lemma}
	\label{le:crit54}
	Suppose that $\nu=1$, that Condition~\eqref{eq:bwregvar} holds with $\alpha=5/4$, and that
	\[
	c_\omega := \lim_{k \to \infty} L(k) \in ]0,\infty[.
	\]
	Set
	\[
	\kappa_{5/4} := 2d_\omega^{5/4}c_\omega \left(\frac{8}{3}\right)^{1/4}\mathfrak d_{\mathbf{M}}.
	\]
	Then
	\begin{align*}
		\Exb{d(\mV_{\widehat\xi})\one_{\{\widehat\xi\le x\}}} \sim \kappa_{5/4}\log x, \qquad x\to\infty.
	\end{align*}
\end{lemma}
\begin{proof}
	By Theorem~\ref{te:main1} and the re-rooting invariance of non-separable planar maps it follows that
	\[
		\left(\frac{3}{8k}\right)^{1/4}d(\mV_k) \convd d_{\mathbf{M}}(U_1,U_2), \qquad k\to\infty.
	\]
	In particular, for any $A>0$
	\[
		\Exb{ \min\left(A, \left(\frac{3}{8k}\right)^{1/4}d(\mV_k)\right) } \to \Exb{\min\left(A, d_{\mathbf{M}}(U_1,U_2)\right)}, \qquad k \to \infty.
	\]
	By Condition~\eqref{eq:bwregvar} and $L(k) \to c_\omega$ we have
	\[
		\Prb{\xi=2k} \sim c_\omega d_\omega^{9/4}k^{-9/4}, \qquad k\to\infty, \qquad k\equiv0 \mod d_\omega.
	\]
	Consequently,
	\[
		\sum_{\substack{1\le k\le m\\k\equiv0 \mod d_\omega}}k^{5/4}\Prb{\xi=2k} \sim c_\omega d_\omega^{5/4}\log m.
	\]
	By the Toeplitz lemma it follows that as $m \to \infty$
	\begin{multline}
		\label{eq:eltopa}
		\frac{1}{\log m}\sum_{\substack{1\le k\le m\\k\equiv0 \mod d_\omega}}k^{5/4}\Prb{\xi=2k}\Exb{ \min\left(A, \left(\frac{3}{8k}\right)^{1/4}d(\mV_k)\right) } \\ \to c_\omega d_\omega^{5/4}\Exb{\min\left(A, d_{\mathbf{M}}(U_1,U_2)\right)}.
	\end{multline}
	Fix $p>1$. Since $d(\mV_k)\le\Di(\mV_k)$, Proposition~\ref{pro:andm} yields, uniformly for $j\ge0$,
	\begin{align*}
		\sum_{\substack{2^j\le k<2^{j+1}\\k\equiv0 \mod d_\omega}} \frac{1}{k}\Exb{\left(\left(\frac{3}{8k}\right)^{1/4}d(\mV_k)\right)^p} &\le  (2^{-j})^{1 + p/4} \left(\frac{3}{8}\right)^{p/4} \sum_{\substack{2^j\le k<2^{j+1}\\k\equiv0 \mod d_\omega}} \Exb{\Di(\mV_k)^p}\\
		&\le (3/8)^{p/4} 2^{1+ p/4} E_p .
	\end{align*}
	Hence, as $m \to \infty$
	\[	
		\sum_{\substack{1\le k\le m\\k\equiv0 \mod d_\omega}} k^{5/4}\Prb{\xi=2k}\Exb{\left(\left(\frac{3}{8k}\right)^{1/4}d(\mV_k)\right)^p} = O(\log m).
	\]
	For any $x \ge A$ it holds that $x - A \le x = x A^{p-1} A^{1-p}  \le x^p A^{1-p}$. It follows that
	\[
		\limsup_{m\to\infty}\frac{1}{\log m}\sum_{\substack{1\le k\le m\\k\equiv0 \mod d_\omega}}k^{5/4}\Prb{\xi=2k}\Exb{\max\left(0, \left(\frac{3}{8k}\right)^{1/4}d(\mV_k)-A\right)} \le O(A^{1-p})
	\]
	with an $O$-term that does not depend on $A$. Note that for all $x \ge 0$ we have $x = \min(x,A) + \max(0, x-A)$. Hence, summing the last display with~\eqref{eq:eltopa} yields 
	\begin{multline*}
		\limsup_{m\to\infty} \frac{1}{\log m}\sum_{\substack{1\le k\le m\\k\equiv0 \mod d_\omega}}k^{5/4}\Prb{\xi=2k}\Exb{\left(\frac{3}{8k}\right)^{1/4}d(\mV_k)} \\ \le c_\omega d_\omega^{5/4}\Exb{\min\left(A, d_{\mathbf{M}}(U_1,U_2)\right)} + CA^{1-p}.
	\end{multline*}
	for a constant $C>0$ that does not depend on $A$. By~\eqref{eq:eltopa} we also have
	\begin{multline*}
		\liminf_{m\to\infty} \frac{1}{\log m}\sum_{\substack{1\le k\le m\\k\equiv0 \mod d_\omega}}k^{5/4}\Prb{\xi=2k}\Exb{\left(\frac{3}{8k}\right)^{1/4}d(\mV_k)} \\ \ge  c_\omega d_\omega^{5/4}\Exb{\min\left(A, d_{\mathbf{M}}(U_1,U_2)\right)}.
	\end{multline*}
	Since $A>0$ was arbitrary, and since the diameter of the Brownian sphere has a finite average, it follows that as $m \to \infty$
	\[
		\frac{1}{\log m}\sum_{\substack{1\le k\le m\\k\equiv0 \mod d_\omega}}k^{5/4}\Prb{\xi=2k}\Exb{\left(\frac{3}{8k}\right)^{1/4}d(\mV_k)} \to c_\omega d_\omega^{5/4}\mathfrak d_{\mathbf{M}}.
	\]
	Since, conditionally on $\widehat\xi=2k$, the random variable $d(\mV_{\widehat\xi})$ is distributed like $d(\mV_k)$, it follows that
	\begin{align*}
		\Exb{d(\mV_{\widehat\xi})\one_{\{\widehat\xi\le2m\}}} &= \sum_{\substack{1\le k\le m\\k\equiv0 \mod d_\omega}}2k\Prb{\xi=2k}\Exb{d(\mV_k)} \\
		&= 2\left(\frac{8}{3}\right)^{1/4}\sum_{\substack{1\le k\le m\\k\equiv0 \mod d_\omega}}k^{5/4}\Prb{\xi=2k}\Exb{\left(\frac{3}{8k}\right)^{1/4}d(\mV_k)} \\
		&\sim 2d_\omega^{5/4}c_\omega\left(\frac{8}{3}\right)^{1/4}\mathfrak d_{\mathbf{M}}  \log(m)
	\end{align*}
	as $m \to \infty$. Since the logarithm is slowly varying this completes the proof.
\end{proof}

\begin{lemma}
	\label{le:crti54b}
	Suppose that $\nu=1$, that Condition~\eqref{eq:bwregvar} holds with $\alpha=5/4$, and that
	\[
		c_\omega = \lim_{k \to \infty} L(k) \in ]0,\infty[.
	\]
	Then
	\begin{align*}
		d_{\mathrm{GHP}}\Bigg( \left( \mM_n^\omega, \frac{b_N}{N\log b_N}d_{\mM_n^\omega}, \mu_{\mM_n^\omega} \right), \left( \mT_N, \frac{\kappa_{5/4}b_N}{N}d_{\mT_N}, \mu_N^{\mathrm{corn}} \right) \Bigg) \convp 0.
	\end{align*}
\end{lemma}
\begin{proof}
	Recall that $N = 2n+1$. We restrict $n$ to multiples of $d_\omega$. 	Proposition~\ref{pro:andm} implies that for every $p>1$,
	\begin{align}
		\label{eq:piano1}
		\sum_{k\le x}k^{-5/4}\Exb{\Di(\mV_k)^p} &\le \sum_{0\le j\le\lfloor\log_2x\rfloor}2^{-5j/4}\sum_{2^j\le k<2^{j+1}}\Exb{\Di(\mV_k)^p} \\
		&\le E_p2^{1+p/4}\sum_{0\le j\le\lfloor\log_2x\rfloor}2^{j(p-1)/4} \nonumber\\
		&=O\left(x^{(p-1)/4}\right).\nonumber
	\end{align}
	Analogous calculations yield  for every $p>5$,
	\begin{align}
		\label{eq:piano2}
		\sum_{k\le x}k^{-9/4}\Exb{\Di(\mV_k)^p} = O\left(x^{(p-5)/4}\right).
	\end{align}
	and
	\begin{align}
		\label{eq:piano3}
		\sum_{k>x}k^{-9/4}\Exb{\Di(\mV_k)} \le O(x^{-1}).
	\end{align} 
	Karamata's theorem yields
	\[
		\Prb{\xi\ge x} \sim \frac{4c_\omega}{5}(2d_\omega)^{5/4}x^{-5/4}, \qquad x\to\infty.
	\]
	Hence
	\begin{align}
		\label{eq:symps}
		b_N \sim 2d_\omega\left(\frac{16c_\omega\Gamma(3/4)}{5}\right)^{4/5}N^{4/5}.
	\end{align}
	Set
	\[
		K_N:=\frac{N^{4/5}}{(\log N)^{1/2}}, \qquad t_N:=\frac{N^{1/5}}{(\log N)^{1/16}}.
	\]
	Let $(\xi_i)_{i \ge 1}$ denote independent copies of $\xi$ and set $S_j = \sum_{i=1}^j \xi_i$ for all $j \ge 0$. Using Dwass' formula,  the local limit theorem~\cite[Theorem~4.2.1]{MR0322926}, Markov's inequality and~\eqref{eq:piano2} and~\eqref{eq:symps} it follows that
	\begin{align}
		\label{eq:diaam1}
		&\Prb{\max_{\substack{v\in\mT_N\\d_{\mT_N}^+(v)\le K_N}}\Di(\beta_N(v))>t_N} \\
		&\qquad \le \Exb{\sum_{v\in\mT_N}\one_{\{d_{\mT_N}^+(v)\le K_N\}}\one_{\{\Di(\beta_N(v))>t_N\}}} \nonumber \\
		&\qquad \le \Prb{S_N = N-1}^{-1} N \sum_{\substack{0 \le k \le K_N, \\ k \equiv 0 \mod 2}} \Pr{\xi=k} \Prb{\Di(\mV_{k/2}) \ge t_N} \Prb{S_{N-1} = N-1 - k} \nonumber \\
		&\qquad \le O(N) t_N^{-12}\sum_{k\le K_N/2}k^{-9/4}\Exb{\Di(\mV_k)^{12}} \nonumber\\
		&\qquad \le O(N) K_N^{7/4} t_N^{-12} \nonumber\\
		&\qquad = O\left((\log N)^{-1/8}\right).\nonumber
	\end{align}
	Likewise, using~\eqref{eq:piano3} we obtain
	\[
		\Exb{\sum_{\substack{v\in\mT_N\\d_{\mT_N}^+(v)>K_N}}\Di(\beta_N(v))} \le O(N) \sum_{k>K_N/2}k^{-9/4}\Exb{\Di(\mV_k)} = O\left(N^{1/5} (\log N)^{1/2}\right).
	\]
	Consequently,
	\begin{align}
		\label{eq:diaam2}
		\frac{1}{N^{1/5} \log N}\sum_{\substack{v\in\mT_N\\d_{\mT_N}^+(v)>K_N}}\Di(\beta_N(v)) \convp 0.
	\end{align}
	By~\eqref{eq:diaam1} it also follows that
	\begin{align}
		\label{eq:diaam3}
		\max_{v\in\mT_N}\Di(\beta_N(v)) = o_p(N^{1/5} \log N).
	\end{align}
	
	Set
	\[
		\kappa_N := \Exb{ \min\left(d(\mV_{\widehat\xi})\one_{\{\widehat\xi\le K_N\}}, t_N\right)}.
	\]
	Lemma~\ref{le:crit54} and $\log K_N \sim \frac{4}{5} \log N$ yield
	\[
		\Exb{d(\mV_{\widehat\xi})\one_{\{\widehat\xi\le K_N\}}} = \frac{4}{5} \kappa_{5/4}  \log N + o(\log N).
	\]
	Since $\Prb{\widehat\xi=2k}=2k\Prb{\xi=2k}=O(k^{-5/4})$, Equation~\eqref{eq:piano1} with $p=12$ gives
	\begin{align*}
		0 &\le \Exb{d(\mV_{\widehat\xi})\one_{\{\widehat\xi\le K_N\}}}-\kappa_N \\
		&\le \Exb{\max\left(d(\mV_{\widehat\xi})-t_N, 0\right)\one_{\{\widehat\xi\le K_N\}} } \\
		&\le t_N^{-11}\Exb{d(\mV_{\widehat\xi})^{12}\one_{\{\widehat\xi\le K_N\}}} \\
		&\le O(t_N^{-11})\sum_{k\le K_N/2}k^{-5/4}\Exb{\Di(\mV_k)^{12}} \\
		&=O(t_N^{-11}K_N^{11/4}) \\
		&=O\left((\log N)^{-11/16}\right).
	\end{align*}
	Here we have used $\max(x-y,0) \le x^{12} / y^{11} $ for all $x,y>0$. Thus
	\begin{align}
		\label{eq:ccm}
		\kappa_N = \frac{4}{5}\kappa_{5/4} \log N + o(\log N).
	\end{align}
	Note also that by~\eqref{eq:piano1}
	\begin{align}
		\label{eq:ccv}
		\Exb{ \min\left(d(\mV_{\widehat\xi})\one_{\{\widehat\xi\le K_N\}}, t_N\right)^2} 	&\le \Exb{d(\mV_{\widehat\xi})^2\one_{\{\widehat\xi\le K_N\}}} \\
							&\le O(1)\sum_{k\le K_N/2}k^{-5/4}\Exb{\Di(\mV_k)^2} \nonumber \\
							&=O(K_N^{1/4}) \nonumber \\
							&=O\left(\frac{N^{1/5}}{(\log N)^{1/8}}\right). \nonumber 
	\end{align}
	For each $N$, let $D_0,D_1,\ldots$ denote independent copies of
	\[
		\min\left(d(\mV_{\widehat\xi})\one_{\{\widehat\xi\le K_N\}},t_N\right).
	\]
	Fix $A,\epsilon>0$. By~\eqref{eq:ccv} we have uniformly for $1\le h\le AN^{1/5}$ and $0\le i<h$, 
	\[
		\sum_{j=i}^{h-1}\Va{D_j}\le h\Exb{D_0^2}=O\left(\frac{N^{2/5}}{(\log N)^{1/8}}\right).
	\]
	Moreover, almost surely $|D_j-\kappa_N|\le t_N$. Bernstein's inequality therefore gives, uniformly for $1\le h\le AN^{1/5}$ and $0\le i<h$,
	\begin{align*}
		&\Prb{\left|\sum_{j=i}^{h-1}(D_j-\kappa_N)\right|>\epsilon N^{1/5}\log N} \\
		&\qquad\le 2\exp\left(-\frac{\epsilon^2N^{2/5}\log^2N}{2\left(C_AN^{2/5}(\log N)^{-1/8}+\epsilon t_NN^{1/5}\log N/3\right)}\right) \\
		&\qquad\le 2\exp\left(-c_{A,\epsilon}(\log N)^{17/16}\right)
	\end{align*}
	for a constant $c_{A,\epsilon}>0$ that does not depend on $N$. A union bound over $0\le i<h$ consequently yields
	\begin{align}
		\label{eq:tbo}
		\Prb{\max_{0\le i<h}\left|\sum_{j=i}^{h-1}(D_j-\kappa_N)\right|>\epsilon N^{1/5}\log N}\le 2h\exp\left(-c_{A,\epsilon}(\log N)^{17/16}\right).
	\end{align}
	For $v\in\mT_N$ with ancestral line
	\[
		o=v_0,v_1,\ldots,v_h=v,
	\]
	use the passage lengths $J_j(v)$ from the proof of Lemma~\ref{le:blot}. Thus $J_j(v)$ is the distance traversed inside the decoration $\beta_N(v_j)$ when passing from $v_j$ to $v_{j+1}$, and
	\[
		d_{\mM_n^\omega}(\pi_N(o),\pi_N(v))=\sum_{j=0}^{h-1}J_j(v).
	\]
	We use the same definition for the unconditioned enriched tree $(\mT,\beta)$. Set
	\[
		R_N := \max_{v\in\mT_N \setminus\{o\}}\max_{0\le i<\he_{\mT_N}(v)} \left|\sum_{j=i}^{\he_{\mT_N}(v)-1} \left( \min\left(J_j(v)\one_{\{d_{\mT_N}^+(v_j)\le K_N\}}, t_N\right)-\kappa_N\right)\right|.
	\]
	Using~\eqref{eq:spinal} analogously as in~\eqref{eq:mlbound} we obtain
	\begin{align*}
		&\Prb{R_N>\epsilon N^{1/5}\log N,\ \He(\mT_N)\le AN^{1/5}} \\
		&\qquad\le\Prb{|\mT|=N}^{-1}\sum_{1\le h\le AN^{1/5}}2h\exp\left(-c_{A,\epsilon}(\log N)^{17/16}\right) \\
		&\qquad\le O(Nb_N)\exp\left(-c_{A,\epsilon}(\log N)^{17/16}\right) ( AN^{1/5} )^2 \\
		&\qquad\to 0,
	\end{align*}
	since Dwass' formula and the lattice local limit theorem~\cite[Theorem~4.2.1]{MR0322926}, applied along the admissible values of $N$, give
	\[
		\Prb{|\mT|=N}=\frac{1}{N}\Prb{S_N=N-1}=\Theta((Nb_N)^{-1}).
	\]
	By the functional limits of encoding processes of $\mT_N$~\cite{MR1964956} we have
	\begin{align}
		\label{eq:hmtno}
		\He(\mT_N)=O_p(N/b_N)=O_p(N^{1/5}).
	\end{align}
	It follows that
	\begin{align}
		\label{eq:cts}
		R_N = o_p(N^{1/5}\log N).
	\end{align}
	
	By~\eqref{eq:diaam1} it holds with probability tending to one that
	\[
		\max_{\substack{v\in\mT_N\\d_{\mT_N}^+(v)\le K_N}}\Di(\beta_N(v))\le t_N.
	\]
	On this event, whenever $d_{\mT_N}^+(v_j)\le K_N$ we have
	\[
		J_j(v)\le\Di(\beta_N(v_j))\le t_N,
	\]
	and hence
	\[
		\min\left(J_j(v)\one_{\{d_{\mT_N}^+(v_j)\le K_N\}}, t_N\right) = J_j(v).
	\]
	Consequently, for every $v\in\mT_N$ and $0\le i<\he_{\mT_N}(v)$,
	\begin{align*}
		0 & \le\sum_{j=i}^{\he_{\mT_N}(v)-1}J_j(v) - \sum_{j=i}^{\he_{\mT_N}(v)-1} \min\left(J_j(v)\one_{\{d_{\mT_N}^+(v_j)\le K_N\}}, t_N\right) \\
		&\le \sum_{\substack{x\in\mT_N\\d_{\mT_N}^+(x)>K_N}} \Di(\beta_N(x)).
	\end{align*}
	Combining this with~\eqref{eq:diaam2} and~\eqref{eq:cts} yields
	\[
		\max_{v\in\mT_N\setminus\{o\}}\max_{0\le i<\he_{\mT_N}(v)}\left|\sum_{j=i}^{\he_{\mT_N}(v)-1}J_j(v)-\kappa_N\bigl(\he_{\mT_N}(v)-i\bigr)\right|=o_p(N^{1/5}\log N).
	\]
	By~\eqref{eq:ccm} and~\eqref{eq:hmtno} we have
	\[
		\left|\kappa_N - \frac{4}{5}\kappa_{5/4}\log N\right|\He(\mT_N)=o_p(N^{1/5}\log N).
	\]
	Therefore,
	\begin{align}
		\label{eq:diaam4}
		\max_{v\in\mT_N \setminus\{o\}}\max_{0\le i<\he_{\mT_N}(v)}\left|\sum_{j=i}^{\he_{\mT_N}(v)-1}\left(J_j(v)-\frac{4}{5}\kappa_{5/4}\log N\right)\right| = o_p(N^{1/5}\log N).
	\end{align}
	
	Let $u,v\in\mT_N$ and let $w$ be their most recent common ancestor. Since any path in $\mM_n^\omega$ from $\pi_N(u)$ to $\pi_N(v)$ passes through the block $\beta_N(w)$ we have
	\[
		d_{\mM_n^\omega}(\pi_N(u), \pi_N(v)) = d_{\mM_n^\omega}(\pi_N(w), \pi_N(u)) + d_{\mM_n^\omega}(\pi_N(w), \pi_N(v)) + R
	\]
	for an error $|R| \le 2 \max_{v\in\mT_N}\Di(\beta_N(v))$. By~\eqref{eq:diaam3} and~\eqref{eq:diaam4} it follows that
	\[
		\sup_{u,v\in\mT_N}\left|d_{\mM_n^\omega}(\pi_N(u),\pi_N(v))- \frac{4}{5} \kappa_{5/4}\log N d_{\mT_N}(u,v)\right| = o_p(N^{1/5}\log N).
	\]
	The stationary distribution $\mu_{\mM_n^\omega}$ is the push-forward measure of $\mu_N^{\mathrm{corn}}$ along $\pi_N$. Using~\eqref{eq:symps} it follows that
	\[
		d_{\mathrm{GHP}}\Bigg(\left(\mM_n^\omega,\frac{b_N}{N\log b_N}d_{\mM_n^\omega},\mu_{\mM_n^\omega}\right),\left(\mT_N,\frac{\kappa_{5/4}b_N}{N}d_{\mT_N},\mu_N^{\mathrm{corn}}\right)\Bigg)\convp0.
	\]
\end{proof}

\begin{corollary}
	Suppose that $\nu=1$, that Condition~\eqref{eq:bwregvar} holds with $\alpha=5/4$, and that
	\[
	c_\omega := \lim_{k \to \infty} L(k) \in ]0,\infty[.
	\]
	Then
	\[
	\left( \mM_n^\omega, \frac{1}{\mathfrak d_{\mathbf{M}}}\left(\frac{6}{d_\omega}\right)^{1/4}\left(\frac{5\Gamma(3/4)^4}{c_\omega}\right)^{1/5}\frac{1}{n^{1/5}\log n}d_{\mM_n^\omega}, \mu_{\mM_n^\omega} \right) \convd \left( \mathbf T_{5/4}, d_{\mathbf T_{5/4}}, \mu_{\mathbf T_{5/4}} \right).
	\]
\end{corollary}
\begin{proof}
	By the main result of~\cite{MR1964956} it follows that
	\[
		\left( \mT_N, \frac{b_N}{N}d_{\mT_N}, \mu_N^{\mathrm{corn}} \right) \convd \left( \mathbf T_{5/4}, d_{\mathbf T_{5/4}}, \mu_{\mathbf T_{5/4}} \right).
	\]
	Combining this convergence with Lemma~\ref{le:crti54b} gives
	\[
		\left( \mM_n^\omega, \frac{b_N}{\kappa_{5/4}N\log b_N}d_{\mM_n^\omega}, \mu_{\mM_n^\omega} \right) \convd \left( \mathbf T_{5/4}, d_{\mathbf T_{5/4}}, \mu_{\mathbf T_{5/4}} \right).
	\]
	Using~\eqref{eq:symps} and $N \sim 2n$ it follows that
	\[
		\frac{b_N}{\kappa_{5/4}N\log b_N} \sim \frac{1}{\mathfrak d_{\mathbf{M}}}\left(\frac{6}{d_\omega}\right)^{1/4}\left(\frac{5\Gamma(3/4)^4}{c_\omega}\right)^{1/5}\frac{1}{n^{1/5}\log n}.
	\]
	This completes the proof.
\end{proof}

\subsection{The Brownian sphere decorated stable tree case}

\begin{corollary}
	Suppose that $\nu=1$ and Condition~\eqref{eq:bwregvar}  holds with $1<\alpha<5/4$. Then
	\[
	\left( \mM_n^\omega, \left( \frac{3}{2^{2 + 1/\alpha} b_{n}} \right)^{1/4} d_{\mM_n^\omega}, \mu_{\mM_n^\omega} \right) \convd \left(\mathbf{T}^{\mathrm{dec}}_\alpha(\mathbf{M}), d_{\mathbf{T}^{\mathrm{dec}}_\alpha(\mathbf{M})}, \mu_{\mathbf{T}^{\mathrm{dec}}_\alpha(\mathbf{M})}\right).
	\]
\end{corollary}
\begin{proof}
	Lemma~\ref{le:smalldeco} applies to the diameters of non-separable planar maps as marks: Consider the case where $\mD_0=0$, and, for $m\ge1$,  $\mD_{2m} \eqdist \Di(\mV_m)$ and $\mD_{2m-1} := 0$.  Choose any integer $p > \frac{4\alpha}{3/2-\alpha}>2$ and set $\gamma_{\mathrm{dec}} :=1/4 > \alpha-1$. By Proposition~\ref{pro:andm},
	\begin{align*}
		\sum_{k=0}^{m}\Ex{\mD_k^p} \le \sum_{j=1}^{\lfloor m/2\rfloor}\Ex{\Di(\mV_j)^p} \le Cm^{1+p/4}.
	\end{align*}
	Hence the requirements of Lemma~\ref{le:smalldeco} are satisfied, yielding for each $\epsilon>0$
	\begin{align}
		\label{eq:nonobvious}
		\lim_{\delta \downarrow 0} \sup_{\substack{N \ge 1 \\ \Pr{|\mT|=N}>0}}
		\Prb{\max_{v\in\mT_N} \sum_{u\in[o,v]_{\mT_N}} \Di(\beta_N(u)) \one_{\{d_{\mT_N}^+(u)\le\delta b_N\}} > \epsilon b_N^{1/4} } = 0.
	\end{align}
	Hence we may apply the invariance principle~\cite[Thm.~5.1]{zbMATH07790315} for decorated stable trees to $\mM_n^\omega$. The only non-obvious Condition B1 is precisely~\eqref{eq:nonobvious}. We obtain
	\begin{align*}
		\left( \mM_n^\omega, b_N^{-1/4}d_{\mM_n^\omega}, \mu_{\mM_n^\omega} \right) \convd \left(\mathbf{T}^{\mathrm{dec}}_\alpha(\mathbf{M}),\left(\frac{4}{3}\right)^{1/4}d_{\mathbf{T}^{\mathrm{dec}}_\alpha(\mathbf{M})},\mu_{\mathbf{T}^{\mathrm{dec}}_\alpha(\mathbf{M})} \right).
	\end{align*}
	Since $N\sim2n$ and $b_n$ varies regularly with index $1/\alpha$, we have $b_N \sim 2^{1/\alpha}b_n$. This yields the asserted normalisation.
\end{proof}

\subsection{The $0\le\nu<1$ case}
\label{sec:subcond}

Let
\[
	\Delta_N:=\max_{v\in\mT_N}d_{\mT_N}^+(v),
\]
and let $v_N^*$ denote the lexicographically first vertex of $\mT_N$ whose outdegree equals $\Delta_N$. Set
\[
	\Delta_N^{(2)}:=\max_{v\in\mT_N\setminus\{v_N^*\}}d_{\mT_N}^+(v).
\]
Delete the edges between $v_N^*$ and its children. Let $F_{N,0}$ denote the resulting component containing the root of $\mT_N$, and let $F_{N,1},\ldots,F_{N,\Delta_N}$ denote the fringe subtrees rooted at the children of $v_N^*$, listed according to their plane order. Set
\[
	W_{N,i}:=|F_{N,i}|, \qquad 0\le i\le\Delta_N.
\]

\begin{lemma}
	\label{le:subcond}
	Suppose that either $0<\nu<1$ and Condition~\eqref{eq:bwregvar} holds with $\alpha\ge1$, or that $\nu=0$ and Condition~\eqref{eq:bwcc14} holds. We have
	\begin{align}
		\label{eq:subdeg}
		\frac{\Delta_N}{(1-\nu)N}\convp1, \qquad \frac{\Delta_N^{(2)}}{N}\convp0.
	\end{align}
	Moreover,
	\begin{align}
		\label{eq:subfr}
		\frac{\max_{0\le i\le\Delta_N}W_{N,i}}{N}\convp0.
	\end{align}
	In the $\nu=0$ and~\eqref{eq:bwcc14} case we additionally have
	\begin{align}
		\label{eq:subcc}
		N-1-\Delta_N=o_p(N^{1/4}).
	\end{align}
\end{lemma}
\begin{proof}
	Suppose first that $0<\nu<1$ and Condition~\eqref{eq:bwregvar} holds with $\alpha\ge1$.
	
	If $\alpha>1$, Equation~\eqref{eq:subdeg} follows from~\cite[Thm.~1]{MR3335012}. If $\alpha=1$, it follows from~\cite[Thm.~1.1(i)]{zbMATH08114289}.
	
	Let $(\mT(i))_{i\ge1}$ denote independent copies of the unconditioned $\xi$-Bienaym\'e--Galton--Watson tree $\mT$. For every fixed $0<\delta<1-\nu$, we have
	\begin{align}
		\label{eq:subforest}
		(F_{N,i})_{1\le i\le\lfloor\delta N\rfloor}\atv(\mT(i))_{1\le i\le\lfloor\delta N\rfloor}.
	\end{align}
	For $\alpha>1$, this is~\cite[Cor.~2.7]{MR3335012}, and for $\alpha=1$ it is~\cite[Prop.~4.2]{zbMATH08114289}.
	
	Since $\Ex{|\mT|}=(1-\nu)^{-1}$, the law of large numbers and~\eqref{eq:subforest} yield
	\begin{align}
		\label{eq:sublln}
		\frac{1}{N}\sum_{i=1}^{\lfloor\delta N\rfloor}W_{N,i}\convp\frac{\delta}{1-\nu}.
	\end{align}
	For every $\epsilon>0$,
	\[
	\Prb{\max_{1\le i\le\lfloor\delta N\rfloor}|\mT(i)|>\epsilon N}\le\delta N\Prb{|\mT|>\epsilon N}\to0.
	\]
	Hence, by~\eqref{eq:subforest},
	\begin{align}
		\label{eq:submaxfirst}
		\frac{1}{N}\max_{1\le i\le\lfloor\delta N\rfloor}W_{N,i}\convp0.
	\end{align}
	Since
	\[
	W_{N,0}+\sum_{i=1}^{\Delta_N}W_{N,i}=N,
	\]
	we have
	\[
	\frac{W_{N,0}}{N}\le1-\frac{1}{N}\sum_{i=1}^{\lfloor\delta N\rfloor}W_{N,i}
	\]
	with probability tending to one. Letting first $N\to\infty$ in~\eqref{eq:sublln} and then $\delta\uparrow1-\nu$ proves
	\[
	\frac{W_{N,0}}{N}\convp0.
	\]
	
	The distribution of $\mT_N$ is invariant under reversing the order of the children of every vertex. By~\eqref{eq:subdeg}, the maximal outdegree is attained by a unique vertex with probability tending to one. Consequently,~\eqref{eq:submaxfirst} also holds for the last $\lfloor\delta N\rfloor$ children of $v_N^*$. Choose $\delta$ such that $\frac{1-\nu}{2}<\delta<1-\nu$.	Since $\Delta_N/N\convp1-\nu$, the first and last $\lfloor\delta N\rfloor$ children cover all children of $v_N^*$ with probability tending to one. Hence
	\[
	\frac{1}{N}\max_{1\le i\le\Delta_N}W_{N,i}\convp0.
	\]
	This proves~\eqref{eq:subfr} when $0<\nu<1$.
	
	Suppose now that $\nu=0$ and~\eqref{eq:bwcc14} holds. By the balls-in-boxes representation of the outdegrees of simply generated trees~\cite[Sec. 15]{MR2908619} and~\eqref{eq:bwcc14}, for every fixed $0<\epsilon<1$,
	\[
		\Prb{N-1-\Delta_N\ge\lfloor\epsilon n^{1/4}\rfloor}=\frac{[z^{2n}]\left(\Phi_{\le2n-\lfloor\epsilon n^{1/4}\rfloor}(z)\right)^{2n+1}}{[z^{2n}]\Phi(z)^{2n+1}} \to 0.
	\]
	This verifies~\eqref{eq:subcc} and implies $\Delta_N/ N \convp 1$. Hence we have also verified the first convergence in~\eqref{eq:subdeg}. Since the sum of all outdegrees of $\mT_N$ equals $N-1$,
	\[
		\Delta_N^{(2)}\le\sum_{v\in\mT_N\setminus\{v_N^*\}}d_{\mT_N}^+(v)=N-1-\Delta_N=o_p(N^{1/4}).
	\]
	This proves the second convergence in~\eqref{eq:subdeg}. Finally,
	\[
		\sum_{i=0}^{\Delta_N}(W_{N,i}-1)=N-1-\Delta_N.
	\]
	All summands on the left-hand side are non-negative. Hence
	\[
		\max_{0\le i\le\Delta_N}W_{N,i}\le N-\Delta_N=o_p(N^{1/4}),
	\]
	which proves~\eqref{eq:subfr} and completes the proof.
\end{proof}

We next derive a moment estimate for the unconditioned block-weighted planar map $\mM^\omega$ encoded by $(\mT,\beta)$.

\begin{lemma}
	\label{le:submom}
	Suppose that $0<\nu<1$ and that Condition~\eqref{eq:bwregvar} holds with $\alpha \ge 1$. We have
	\[
	\Ex{\Di(\mM^\omega)^4}<\infty.
	\]
\end{lemma}
\begin{proof}
	Define $(\mD_k)_{k \ge 0}$ so that $\mD_0=0$ and for all $m \ge 1$
	\[
		\mD_{2m} \eqdist \Di(\mV_m), \qquad \mD_{2m-1}=0.
	\]
	By Proposition~\ref{pro:andm}, there is a constant $E_4>0$ such that
	\begin{align}
		\label{eq:subav}
		\sum_{k=0}^{r}\Ex{\mD_k^4}\le E_4r^2
	\end{align}
	for all $r\ge1$. Set $q:=2d_\omega$. By Condition~\eqref{eq:bwregvar}, the uniform convergence theorem for slowly varying functions, and~\eqref{eq:subav}, there are constants $C_1, C_2>0$ with
	\begin{align*}
		\sum_{2^j\le k<2^{j+1}}\Prb{\xi=qk}\Ex{\mD_{qk}^4}
		&\le  C_1 L(2^j)2^{-j(1+\alpha)}\sum_{k<2^{j+1}}\Ex{\Di(\mV_{d_\omega k})^4}\\
		&\le C_2 L(2^j)2^{-j(\alpha-1)}
	\end{align*}
	for all sufficiently large $j$.
	
	If $\alpha>1$, the series $ \sum_{j\ge0}L(2^j)2^{-j(\alpha-1)}$	converges by Potter's bounds. If $\alpha=1$, we have by $\Ex{\xi}<\infty$ that
	\[
		\sum_{k\ge1}\frac{L(k)}{k}<\infty.
	\]
	By the uniform convergence theorem for slowly varying functions,
	\[
		\sum_{2^j\le k<2^{j+1}}\frac{L(k)}{k}\asymp L(2^j)
	\]
	for all sufficiently large $j$. Hence $\sum_{j\ge0}L(2^j)<\infty$.
	
	 Define $\mD_\xi$ so that conditional on $\xi=k$ it is distributed like $\mD_k$.  We have therefore proved
	\begin{align}
		\label{eq:subrootmom}
		\Ex{\mD_\xi^4} < \infty.
	\end{align}
	
	For $h\ge0$, let $\mR_h$ denote the maximal distance from the root vertex in the map encoded by the first $h$ generations of $(\mT,\beta)$. We use the convention $\mR_0=0$. Conditionally on $\xi$, let $(\mR_{h-1}(i))_{1 \le i \le \xi}$ denote independent copies of $\mR_{h-1}$. This way,
	\[
	\mR_h\le\mD_\xi+\max_{1\le i\le\xi}\mR_{h-1}(i),
	\]
	where the maximum over the empty set is taken to be zero.
	
	For every $\eta>0$, there is a constant $C_\eta>0$ such that
	\[
	(x+y)^4\le C_\eta x^4+(1+\eta)y^4, \qquad x,y\ge0.
	\]
	Since $0<\nu<1$ there exists $\eta>0$ such that $(1+\eta)\nu<1$. It follows from~\eqref{eq:subrootmom} and $\Ex{\xi}=\nu$ that
	\begin{align*}
		\Ex{\mR_h^4}
		&\le C_\eta\Ex{\mD_\xi^4}+(1+\eta)\Exb{\max_{1\le i\le\xi}\mR_{h-1}(i)^4}\\
		&\le C_\eta\Ex{\mD_\xi^4}+(1+\eta)\nu\Ex{\mR_{h-1}^4}.
	\end{align*}
	Consequently,
	\[
	\sup_{h\ge0}\Ex{\mR_h^4}<\infty.
	\]
	The tree $\mT$ is almost surely finite, and hence $\mR_h \convas \He(\mM^\omega)$ as $h \to \infty$. By monotone convergence,
	\[
	\Ex{\He(\mM^\omega)^4}<\infty.
	\]
	Since $\Di(\mM^\omega)\le2\He(\mM^\omega)$, this completes the proof.
\end{proof}

Set
\[
	\mV_N^*:=\beta_N(v_N^*).
\]
For $1\le i\le\Delta_N$, let $M_{N,i}$ denote the planar map encoded by the enriched fringe subtree $(F_{N,i},\beta_N|_{F_{N,i}})$. Let $M_{N,0}$ denote the root-side component obtained from $(F_{N,0},\beta_N|_{F_{N,0}})$ by replacing the decoration at the marked vertex $v_N^*$ by a single vertex. Thus, the maps $M_{N,0},M_{N,1},\ldots,M_{N,\Delta_N}$ are attached to $\mV_N^*$ at specified corners to form $\mM_n^\omega$. Let $\varpi_N: \mM_n^\omega \to \mV_N^*$ fix $\mV_N^*$ pointwise and map every vertex of each $M_{N,i}$ to its attachment vertex in $\mV_N^*$. Set
\[
	R_N:=\max_{x\in\mM_n^\omega}d_{\mM_n^\omega}(x,\varpi_N(x)).
\]

\begin{lemma}
	\label{le:submet}
	Suppose that either $0<\nu<1$ and Condition~\eqref{eq:bwregvar} holds with $\alpha\ge1$, or that $\nu=0$ and Condition~\eqref{eq:bwcc14} holds. Then
	\[
		R_N=o_p(N^{1/4}).
	\]
\end{lemma}
\begin{proof}
	Suppose first that $0<\nu<1$ and~\eqref{eq:bwregvar}  with $\alpha\ge1$.	Let $(\mM^\omega(i))_{i\ge1}$ denote independent copies of $\mM^\omega$. By~\eqref{eq:subforest} we have for every fixed $0<\delta<1-\nu$,
	\begin{align}
		\label{eq:submapforest}
		(M_{N,i})_{1\le i\le\lfloor\delta N\rfloor}\atv(\mM^\omega(i))_{1\le i\le\lfloor\delta N\rfloor}.
	\end{align}
	By Lemma~\ref{le:submom}, for every $\epsilon>0$,
	\begin{align*}
		\Prb{\max_{1\le i\le\lfloor\delta N\rfloor}\Di(\mM^\omega(i))>\epsilon N^{1/4}}
		&\le\delta N\Prb{\Di(\mM^\omega)>\epsilon N^{1/4}}\\
		&\le\delta\epsilon^{-4}\Exb{\Di(\mM^\omega)^4\one_{\{\Di(\mM^\omega)>\epsilon N^{1/4}\}}}\\
		&\to 0.
	\end{align*}
	It follows from~\eqref{eq:submapforest} that
	\[
	\max_{1\le i\le\lfloor\delta N\rfloor}\Di(M_{N,i})=o_p(N^{1/4}).
	\]
	Since the distribution of $\mT_N$ is invariant under reversing the order of siblings, the same conclusion holds for the last $\lfloor\delta N\rfloor$ children of $v_N^*$. Using~\eqref{eq:subdeg} and the fact that $0<\delta<1-\nu$ was arbitrary, we obtain
	\begin{align}
		\label{eq:subchilddiam}
		\max_{1\le i\le\Delta_N}\Di(M_{N,i})=o_p(N^{1/4}).
	\end{align}
	
	By~\cite[Sec. 20]{MR2908619}, together with the first convergence in~\eqref{eq:subdeg}, the marked tree $F_{N,0}$ converges in distribution to an almost surely finite tree.   Consequently,
	\begin{align}
		\label{eq:subrootdiam}
		\Di(M_{N,0})=O_p(1).
	\end{align}
	Every component $M_{N,i}$ intersects $\mV_N^*$ only at its attachment vertex. Hence
	\[
		R_N\le\max_{0\le i\le\Delta_N}\Di(M_{N,i}).
	\]
	Equations~\eqref{eq:subchilddiam} and~\eqref{eq:subrootdiam} yield
	\[
		R_N = o_p(N^{1/4}).
	\]
	
	It remains to consider the case when $\nu=0$ and~\eqref{eq:bwcc14}. For any vertex $v \in \mT_N$ the decoration $\beta_N(v)$ has $d_{\mT_N}^+(v)/2$ edges. Hence the total number of edges contained in all decorations except $\mV_N^*$ equals
	\[
	\sum_{v\in\mT_N\setminus\{v_N^*\}}\ed(\beta_N(v))=\frac{1}{2}\sum_{v\in\mT_N\setminus\{v_N^*\}}d_{\mT_N}^+(v)=\frac{N-1-\Delta_N}{2}.
	\]
	Every path from a vertex of $\mM_n^\omega$ to its image under $\varpi_N$ may be chosen inside the union of these decorations. Therefore,
	\[
	R_N\le\frac{N-1-\Delta_N}{2}=o_p(N^{1/4})
	\]
	by~\eqref{eq:subcc}. This completes the proof.
\end{proof}

\begin{corollary}
	Suppose that either $0<\nu<1$ and Condition~\eqref{eq:bwregvar} holds for some $\alpha\ge1$, or $\nu=0$ and Condition~\eqref{eq:bwcc14} holds. Then
	\[
		\left(\mM_n^\omega,\left(\frac{3}{8(1-\nu)n}\right)^{1/4}d_{\mM_n^\omega},\mu_{\mM_n^\omega}\right) \convd (\mathbf{M},d_{\mathbf{M}},\mu_{\mathbf{M}}).
	\]
\end{corollary}

\begin{proof}
	Conditionally on $\Delta_N=2m$, the map $\mV_N^*$ is a uniform non-separable planar map with $m$ edges. Hence $\ed(\mV_N^*)=\frac{\Delta_N}{2}$. By~\eqref{eq:subdeg} it follows that
	\[
		\frac{\ed(\mV_N^*)}{(1-\nu)N/2}\convp1.
	\]
	Hence by Theorem~\ref{te:main1}
	\begin{align*}
		\left(\mV_N^*,\left(\frac{3}{4(1-\nu)N}\right)^{1/4}d_{\mV_N^*},\mu_{\mV_N^*}\right)\convd(\mathbf{M},d_{\mathbf{M}},\mu_{\mathbf{M}}).
	\end{align*}
	By Lemma~\ref{le:submet},
	\begin{align*}
		d_{\mathrm{H}}\left(\left(\mM_n^\omega,\left(\frac{3}{4(1-\nu)N}\right)^{1/4}d_{\mM_n^\omega}\right),\left(\mV_N^*,\left(\frac{3}{4(1-\nu)N}\right)^{1/4}d_{\mV_N^*}\right)\right)\convp0.
	\end{align*}
	By Lemma~\ref{le:subcond} 
	\[
		\frac{\left(\sum_{i=1}^{\Delta_N}W_{N,i}^2\right)^{1/2}}{\sum_{i=1}^{\Delta_N}W_{N,i}}\le\left(\frac{\max_{1\le i\le\Delta_N}W_{N,i}}{N-W_{N,0}}\right)^{1/2}\convp0.
	\]
	Hence by fully analogous arguments as for Lemma~\ref{le:2corepart2}, it follows that 
	\begin{align*}
		d_{\mathrm{GHP}}\left(\left(\mM_n^\omega,\left(\frac{3}{4(1-\nu)N}\right)^{1/4}d_{\mM_n^\omega},\mu_{\mM_n^\omega}\right),\left(\mV_N^*,\left(\frac{3}{4(1-\nu)N}\right)^{1/4}d_{\mV_N^*},\mu_{\mV_N^*}\right)\right)\convp0
	\end{align*}
	and hence, using $N \sim 2n$, 
	\[
		\left(\mM_n^\omega,\left(\frac{3}{8(1-\nu)n}\right)^{1/4}d_{\mM_n^\omega},\mu_{\mM_n^\omega}\right)\convd(\mathbf{M},d_{\mathbf{M}},\mu_{\mathbf{M}}).
	\]
\end{proof}
\subsection{The critical Cauchy case}
\label{sec:cauchy}

Throughout this subsection we focus on   $\nu=1$ with Condition~\eqref{eq:bwregvar} holding for $\alpha=1$. We retain the notation $\Delta_N$, $\Delta_N^{(2)}$, $v_N^*$, $(F_{N,i}, M_{N,i}, W_{N,i})_{0\le i\le\Delta_N}$,  $\mV_N^*$, $\varpi_N$, and $R_N$ introduced above. Define
\[
	\widetilde L(x):=(2d_\omega)^2L(x/(2d_\omega)), \qquad x>0.
\]
Then Condition~\eqref{eq:bwregvar} yields
\[
	\Prb{\xi=k}\sim\widetilde L(k)k^{-2}, \qquad k\to\infty, \qquad k\in2d_\omega\ndN.
\]

Let $u_{N,0},u_{N,1},\ldots,u_{N,N-1}$ denote the vertices of $\mT_N$ in lexicographic order, and let $(\cW_j(\mT_N))_{0\le j\le N}$ be its \L{}ukasiewicz path defined by
\[
	\cW_0(\mT_N):=0, \qquad \cW_{j+1}(\mT_N):=\cW_j(\mT_N)+d_{\mT_N}^+(u_{N,j})-1, \qquad 0\le j<N.
\]
Let $I_N$ be the index satisfying $u_{N,I_N}=v_N^*$. Let $(X_i)_{i\ge1}$ be independent copies of $X=\xi-1$, and set
\[
	S_0:=0, \qquad S_j:=\sum_{i=1}^jX_i, \qquad \underline S_j:=\min_{0\le k\le j}S_k, \qquad j\ge1.
\]
For a finite walk $w=(w_j)_{0\le j\le m}$ with $w_0=0$, write $y_j:=w_j-w_{j-1}$ for $1\le j\le m$, and let
\[
	k(w):=\min\left\{0\le j\le m\mid w_j=\min_{0\le i\le m}w_i\right\}.
\]
The \emph{Vervaat transform} $\operatorname{Ver}(w)$ is the walk starting at zero whose increments, in order, are
\[
	y_{k(w)+1},\ldots,y_m,y_1,\ldots,y_{k(w)},
\]
where an initial or terminal string is omitted when empty. Thus, the increments of $w$ are read in cyclic order, starting immediately after the first time at which $w$ attains its overall minimum. Set
\begin{align}
	\label{eq:vervaa}
	\cZ^{(N)}:=\operatorname{Ver}\bigl(S_0,S_1,\ldots,S_{N-1},-1\bigr).
\end{align}
The walk to which the transform is applied has increment sequence
\[
	X_1,\ldots,X_{N-1},-1-S_{N-1}.
\]
We refer to the last element in this list as the appended increment. We collect several results and consequences of~\cite{MR3914558}.

\begin{proposition}
	\label{pro:catree}
	We have
	\begin{align}
		\label{eq:cadeg}
		\Delta_N=b_N+O_p(a_N), \qquad \Delta_N^{(2)}=O_p(a_N).
	\end{align}
	In particular,
	\[
	\frac{\Delta_N}{b_N}\convp1, \qquad \frac{\Delta_N^{(2)}}{b_N}\convp0,
	\]
	and the maximal outdegree is attained by a unique vertex with probability tending to one. Moreover,
	\[
	\frac{I_N}{N}\convp0.
	\]
	Furthermore,
	\begin{align}
		\label{eq:catv}
		\bigl(\cW_j(\mT_N):0\le j\le N\bigr)\atv\bigl(\cZ_j^{(N)}:0\le j\le N\bigr).
	\end{align}
	With probability tending to one, the appended increment $-1-S_{N-1}$ is strictly larger than $X_1,\ldots,X_{N-1}$. We also have
	\[
	S_{N-1}=-b_N+O_p(a_N), \qquad  \underline S_{N-1}=-b_N+O_p(a_N).
	\]
	Setting $\cW_j(\mT_N):=0$ for $j>N$, we have 
	\begin{align}
		\label{eq:capost}
		\sup_{0\le t \le 1}\left|\frac{\cW_{I_N+1+\lfloor Nt\rfloor}(\mT_N)-\cW_{I_N+1}(\mT_N)}{b_N}+t\right|\convp0.
	\end{align}
	Finally,
	\[
		\frac{\max_{0\le i\le\Delta_N}W_{N,i}}{N}\convp0.
	\]
\end{proposition}
\begin{proof}
	In the notation of~\cite{MR3914558}, the centring sequence is negative and its absolute value is the positive sequence $b_N$ used here. The main results~\cite[Thms.~1 and 3]{MR3914558} give~\eqref{eq:cadeg}. Since $a_N=o(b_N)$, it readily follows that $\Delta_N / b_N\convp1$ and $\Delta_N^{(2)}/b_N \convp0$. In particular, $\Delta_N^{(2)} = o_p(\Delta_N)$, yielding that with probability tending to one the maximal outdegree is attained by a unique vertex.
	
	Combining~\cite[Prop.~8]{MR3914558} and~\cite[Thm.~21]{MR3914558} gives~\eqref{eq:catv}. The assertion concerning the appended increment follows from~\cite[Cor.~22(i)]{MR3914558}. The functional limit~\cite[Eq.~(8)]{MR3914558} implies
	\[
		\sup_{0\le t\le1}\left|\frac{S_{\lfloor Nt\rfloor}+tb_N}{a_N}\right|=O_p(1),
	\]
	and therefore
	\[
	\max_{0\le j\le N-1}\left|S_j+\frac{j}{N}b_N\right|=O_p(a_N).
	\]
	Since $b_N=o(N)$ and $a_N\to\infty$, we have $b_N/N=o(a_N)$. It follows that
	\[
		S_{N-1}=-b_N+O_p(a_N), \qquad \underline S_{N-1}=-b_N+O_p(a_N).
	\]
	The convergence $I_N/N\convp0$ is~\cite[Cor.~25(ii)]{MR3914558}.
	
	Let
	\[
		K_N:=\min\left\{0\le j<N\mid\cZ_{j+1}^{(N)}-\cZ_j^{(N)}=\max_{0\le i<N}\bigl(\cZ_{i+1}^{(N)}-\cZ_i^{(N)}\bigr)\right\}.
	\]
	Extend $\cZ^{(N)}$ by setting $\cZ_j^{(N)}:=0$ for $j>N$. The argument in the proof of~\cite[Thm.~23]{MR3914558} yields
	\[
		\left(\frac{\cZ_{K_N+1+\lfloor Nt\rfloor}^{(N)}-\cZ_{K_N+1}^{(N)}}{b_N}:0\le t\le1\right)\convd(-t:0\le t\le1)
	\]
	in $\mathbb D([0,1],\ndR)$ equipped with the Skorokhod $J_1$ topology. By~\eqref{eq:catv} it follows that likewise
	\[
		\left(\frac{\cW_{I_N+1+\lfloor Nt\rfloor}(\mT_N)-\cW_{I_N+1}(\mT_N)}{b_N}:0\le t\le1\right)\convd(-t:0\le t\le1)
	\]
	in $\mathbb D([0,1],\ndR)$ equipped with the Skorokhod $J_1$ topology. Since the limiting function is deterministic and continuous, this implies convergence in probability with respect to  the uniform norm. In particular,~\eqref{eq:capost} follows.
	
	Set
	\[
	J_N:=\inf\{j>I_N\mid\cW_j(\mT_N)=\cW_{I_N}(\mT_N)-1\}.
	\]
	The vertices with lexicographic indices in $[I_N,J_N[$ are precisely the vertices of the fringe subtree rooted at $v_N^*$. Since
	\[
		\cW_{I_N+1}(\mT_N)=\cW_{I_N}(\mT_N)+\Delta_N-1,
	\]
	we have by~\eqref{eq:cadeg}
	\[
		\frac{\cW_{I_N}(\mT_N)-1-\cW_{I_N+1}(\mT_N)}{b_N}=-\frac{\Delta_N}{b_N}\convp-1.
	\]
	Hence by~\eqref{eq:capost} it follows that
	\[
		\frac{J_N}{N} \convp 1.
	\]
	The component $F_{N,0}$ contains only $v_N^*$, the vertices preceding $v_N^*$ in lexicographic order, and the vertices following the fringe subtree rooted at $v_N^*$. Consequently,
	\[
	W_{N,0}\le I_N+N-J_N+1,
	\]
	and hence $W_{N,0}/N\convp0$.
	
	For $1\le i\le\Delta_N$, let $s_{N,i}$ denote the lexicographic index of the root of $F_{N,i}$. The vertices of $F_{N,i}$ form a consecutive interval of length $W_{N,i}$ in lexicographic order, and the corresponding excursion of the \L{}ukasiewicz path satisfies
	\[
		\cW_{s_{N,i}+W_{N,i}}(\mT_N)=\cW_{s_{N,i}}(\mT_N)-1, \qquad \cW_{s_{N,i}+j}(\mT_N) \ge \cW_{s_{N,i}}(\mT_N) \quad\text{for}\quad 1 \le j < W_{N,i}.
	\]
	Hence,  for all $1\le i\le\Delta_N$
	\begin{align*}
		\left|\frac{W_{N,i}}{N}-\frac{1}{b_N}\right| &= \Bigg|\left(\frac{\cW_{s_{N,i}+W_{N,i}}(\mT_N)-\cW_{I_N+1}(\mT_N)}{b_N}+\frac{s_{N,i}+W_{N,i}-I_N-1}{N}\right) \\
		&\quad -\left(\frac{\cW_{s_{N,i}}(\mT_N)-\cW_{I_N+1}(\mT_N)}{b_N}+\frac{s_{N,i}-I_N-1}{N}\right)\Bigg| \\
		&\le2\sup_{0\le t \le 1}\left|\frac{\cW_{I_N+1+\lfloor Nt\rfloor}(\mT_N)-\cW_{I_N+1}(\mT_N)}{b_N}+t\right|.
	\end{align*}
	By~\eqref{eq:capost} it follows that
	\[
		\frac{\max_{1\le i\le\Delta_N}W_{N,i}}{N} \convp 0.
	\]
	This completes the proof.
\end{proof}

For the remainder of this subsection, specialise the marks $(\mD_k)_{k\ge0}$ from Section~\ref{sec:negli} by setting $\mD_0:=0$, $\mD_{2m}\eqdist\Di(\mV_m)$ for $m\ge1$, and $\mD_{2m-1}:=0$. Conditionally on $\xi$, let $\mD_\xi$ have the law of $\mD_k$ on the event $\{\xi=k\}$.

\begin{lemma}
	\label{le:camom}
	For every $p>4$ there exists a constant $G_p>0$ such that for all sufficiently large $r$, 
	\[
		\Exb{(1+\mD_\xi)^p\one_{\{\xi\le r\}}}\le G_p\widetilde L(r)r^{p/4-1}.
	\]
	Moreover, there is a constant $G>0$ such that for all sufficiently large $r$,
	\[
		\Exb{\xi(1+\mD_\xi)\one_{\{\xi\le r\}}}\le G\widetilde L(r)r^{1/4}.
	\]
\end{lemma}
\begin{proof}
	By the uniform convergence theorem for slowly varying functions, there exist $C_1>0$ and $r_0>0$ such that for all $r \ge r_0$
	\[
		\sup_{r/2<k\le r} \Prb{\xi=k} \le C_1 \widetilde{L}(r)r^{-2}.
	\]
	By Proposition~\ref{pro:andm} it follows that for every $p>0$,
	\[
		\Exb{\mD_\xi^p\one_{\{r/2<\xi\le r\}}}\le C_1 \widetilde L(r)r^{-2}\sum_{k\le r}\Ex{\mD_k^p}\le C_1 E_p\widetilde L(r)r^{p/4-1}.
	\]
	If $p>4$, decomposing $[r_0,r]$ into dyadic intervals gives
	\[
		\Exb{\mD_\xi^p\one_{\{\xi\le r\}}}\le \Exb{\mD_\xi^p\one_{\{\xi\le r_0\}}} + C_1 E_p\sum_{\substack{j\ge0\\2^{-j}r\ge r_0}}\widetilde L(2^{-j}r)(2^{-j}r)^{p/4-1}.
	\]
	Fix $0<\delta<p/4-1$. By Potter's bounds, there is a constant $C_2>0$ such that for all large enough $r$
	\[
		\widetilde L(2^{-j}r)(2^{-j}r)^{p/4-1} \le C_2 \widetilde{L}(r)r^{p/4-1}2^{-j(p/4-1-\delta)}.
	\]
	We may increase $r_0$ so that this holds for all $r \ge r_0$. Hence
	\begin{align*}
		\Exb{\mD_\xi^p\one_{\{\xi\le r\}}}  \le \Exb{\mD_\xi^p\one_{\{\xi\le r_0\}}} + C_1 C_2 E_p\widetilde L(r)r^{p/4-1} \frac{1}{1-2^{-(p/4-1-\delta)}}.
	\end{align*}
	For every $x\ge0$, we have $(1+x)^p\le2^{p-1}(1+x^p)$, hence we obtain
	\begin{align*}
		\Exb{(1+\mD_\xi)^p\one_{\{\xi\le r\}}} \le O(\widetilde L(r)r^{p/4-1})
	\end{align*}
	as $r \to \infty$. By fully analogous arguments, we obtain
	\begin{align*}
		\Exb{\xi(1+\mD_\xi)\one_{\{\xi\le r\}}}\le O(\widetilde L(r)r^{1/4}).
	\end{align*}
	This completes the proof.
\end{proof}

For any $\cV$-enriched tree $(T, \beta_T)$ define
\[
	K(T,\beta_T):=\max_{v\in T}\sum_{u\in[o,v]_{T}}\left(1+\Di(\beta_T(u))\right),
\]
with $o$ denoting the root of $T$. This way, the map $M$ encoded by $(T, \beta_T)$ satisfies $\He(M) \le K(T, \beta_T)$.

\begin{lemma}
	\label{le:caradius}
	We have
	\[
		x^4\Prb{K(\mT,\beta)>x}\to0, \qquad x \to \infty.
	\]
\end{lemma}
\begin{proof}
	Set
	\[
		\ell(x):=\Exb{\xi\one_{\{\xi>x\}}}, \qquad x\ge0.
	\]
	Karamata's theorem for sums over lattices gives
	\[
		\Prb{\xi>x}=O\left(\frac{\widetilde L(x)}{x}\right), \qquad \widetilde L(x)=o(\ell(x)), \qquad x \to \infty.
	\]
	The function $\ell$ is slowly varying. By the smooth variation theorem, see~\cite[Sec. 1.8.1, Thm. 1.8.2]{zbMATH00043570} there is a positive smooth function~$\eta$ such that
	\[
		\eta(x)\sim \sqrt{\frac{\widetilde L(x^4)}{\ell(x^4)}}, \qquad \frac{x\eta'(x)}{\eta(x)}\to0,   \qquad x \to \infty.
	\]
	Consequently,
	\[
		\eta(x) \to 0,  \qquad \frac{\widetilde L(x^4)}{\ell(x^4)}=o(\eta(x)).
	\]

	Let $m_0, C_0 \ge 1$ denote constants to be determined later. Set
	\[
		g(x):=C_0\eta(x)x^{-4}, \qquad x>0.
	\]
	Since $x\eta'(x)/\eta(x) \to 0$ as $x \to \infty$, for all sufficiently large $x$ and all $1\le y\le x/2$,
	\[
		\log\left(\frac{\eta(x-y)}{\eta(x)}\right)\le\int_{x-y}^x\frac{|\eta'(s)|}{\eta(s)}\,\mathrm{d}s\le\int_{x-y}^x\frac{\mathrm{d}s}{s}=\log\left(\frac{x}{x-y}\right).
	\]
	It follows that
	\[
	\frac{g(x-y)}{g(x)}=\frac{\eta(x-y)}{\eta(x)}\left(\frac{x}{x-y}\right)^4\le\left(\frac{x}{x-y}\right)^5.
	\]
	It follows that there  is a constant $C>0$ such that
	\begin{align}
		\label{eq:cagshift}
		g(x-y)\le g(x)\left(1+C\frac{y}{x}\right)
	\end{align}
	for all sufficiently large $x$ and all $1\le y\le x/2$.
	
	Write
	\[
		K:=K(\mT,\beta), \qquad Y:=1+\mD_\xi.
	\]
	Let $( K_i)_{i\ge1}$ be independent copies of $ K$, independent of $(\xi,Y)$. The branching property of the Bienaym\'e--Galton--Watson tree $\mT$ gives
	\[
		K\eqdist Y+\max_{1\le i\le\xi} K_i,
	\]
	where the maximum over the empty set is zero. Using the inequality $1-(1-z)^k\le kz$ we obtain for $k,r \ge 0$
	\begin{align}
		\label{eq:doit}
		\Pr{K>m \mid \xi = k, Y = r} = 1 - (1 - \Pr{K>m-r})^k \le k \Pr{K > m-r}.
	\end{align}
	
	We will prove by induction that
	\[
		\Pr{K>m} \le g(m)
	\]
	for every $m\in\ndN$.  Suppose that $m>m_0$ and that $\Pr{K>j} \le g(j)$ for every $1 \le j < m$.  On the event $\{\xi\le m^4,\ Y\le m/2\}$ we have $1\le m-Y<m$, so the induction hypothesis and~\eqref{eq:doit} and~\eqref{eq:cagshift} yield
	\begin{align*}
		&\Pr{K>m}\\&\le\Prb{\xi>m^4}+\Prb{Y>m/2,\ \xi\le m^4}+\Exb{\xi g(m-Y)\one_{\{\xi\le m^4,\ Y\le m/2\}}} \\
		&\le\Prb{\xi>m^4}+\Prb{Y>m/2,\ \xi\le m^4}+g(m)\left(\Exb{\xi\one_{\{\xi\le m^4\}}}+\frac{C}{m}\Exb{\xi Y\one_{\{\xi\le m^4\}}}\right).
	\end{align*}
	Since $\Ex{\xi}=1$ and $\ell(x)=\Exb{\xi\one_{\{\xi>x\}}}$, we have
	\[
		\Exb{\xi\one_{\{\xi\le m^4\}}}=1-\ell(m^4).
	\]
	Moreover, Lemma~\ref{le:camom} yields
	\[
		\frac{1}{m}\Exb{\xi Y\one_{\{\xi\le m^4\}}}\le G\widetilde L(m^4).
	\]
	Fix $p>4$. By Karamata's theorem there is a constant $C'>0$ with
	\[
		\Prb{\xi>m^4}\le C'm^{-4}\widetilde L(m^4)
	\]
	uniformly in $m$. By Markov's inequality and Lemma~\ref{le:camom}, there is a constant $C_p>0$ such that uniformly in $m$
	\begin{align*}
		\Prb{Y>m/2, \xi\le m^4}&\le\left(\frac{2}{m}\right)^p\Exb{Y^p\one_{\{\xi\le m^4\}}}\\
		&\le C_p m^{-4} \widetilde L(m^4).
	\end{align*}
	Hence we arrive at
	\[
		\Pr{K>m}\le g(m)\left(1-\ell(m^4)+CG\widetilde L(m^4)\right)+ (C' + C_p) m^{-4}\widetilde L(m^4).
	\]
	Since $\widetilde L(x)=o(\ell(x))$, for all sufficiently large $m$,
	\[
		1-\ell(m^4)+CG\widetilde L(m^4)\le1-\frac{1}{2}\ell(m^4).
	\]
	Furthermore, since $C_0 \ge 1$
	\[
		\frac{m^{-4}\widetilde L(m^4)}{g(m)\ell(m^4)}=\frac{1}{C_0}\frac{\widetilde L(m^4)}{\eta(m)\ell(m^4)} \le \frac{\widetilde L(m^4)}{\eta(m)\ell(m^4)} \to0
	\]
	by the choice of $\eta$. Hence, for all sufficiently large $m$ (without any dependence on $C_0$),
	\[
		(C' + C_p) m^{-4}\widetilde L(m^4)\le\frac{1}{2}g(m)\ell(m^4).
	\]
	We may choose $m_0\ge 1$ so that all preceding estimates for large enough $m$ hold  whenever $m \ge m_0$. This may be done without any dependence on $C_0$. We may next select $C_0 \ge 1$ depending on $m_0$ so that $\Pr{K>k} \le g(k)$ for all $1 \le k \le m_0$. This way, we arrive at 
	\[
		\Pr{K>m}\le g(m),
	\]
	completing the induction. We have therefore proved
	\[
		m^4\Prb{ K>m} \le C_0\eta(m)\to0, \qquad m \to \infty.
	\]
	This readily implies the analogous statement with a real-valued variable $x$ in place  of~$m$.
\end{proof}

\begin{lemma}
	\label{le:cadist}
	We have
	\[
		R_N = o_p(b_N^{1/4}).
	\]
\end{lemma}
\begin{proof}
	For $i\ge0$, set
	\[
		\tau_i:=\inf\{j\ge0\mid S_j=-i\}.
	\]
	Since the increments of $S$ belong to $\{-1,0,1,\ldots\}$ and have mean zero, the stopping times $\tau_i$ are almost surely finite. For every $i\ge1$, define the finite walk
	\[
		S^{(i)}_k:=S_{\tau_{i-1}+k}-S_{\tau_{i-1}}, \qquad 0\le k\le\tau_i-\tau_{i-1}.
	\]
	Then
	\[
		S^{(i)}_0=0,\qquad S^{(i)}_k\ge0\quad\text{for}\quad0\le k<\tau_i-\tau_{i-1},\qquad S^{(i)}_{\tau_i-\tau_{i-1}}=-1.
	\]
	Thus, $S^{(i)}$ is the \L{}ukasiewicz path of a finite plane tree, which we denote by $\mT[i]$. By the strong Markov property at the times $(\tau_i)_{i\ge0}$, the trees $(\mT[i])_{i\ge1}$ are independent copies of the unconditioned $\xi$-Bienaym\'e--Galton--Watson tree $\mT$.
	
	For each $i \ge 1$ let $\beta[i]$ denote the canonical  $\cV$-decoration on $\mT[i]$. This way,
	\[
		(\mT[i],\beta[i]), \qquad i\ge1,
	\]
	are independent copies of $(\mT,\beta)$. Define
	\[
		K_i = K(\mT[i],\beta[i]), \qquad i \ge 1.
	\]
	Define the event
	\[
		\cB_N := \left\{-1-S_{N-1}>\max_{1\le i\le N-1}X_i\right\}.
	\]
	On $\cB_N$, all increments of the Vervaat transform $\cZ^{(N)}$ in~\eqref{eq:vervaa} belong to $\{-1,0,1,\ldots\}$ and have sum $-1$. Hence $\cZ^{(N)}$ is the \L{}ukasiewicz path of a plane tree, which we denote by $\mT_N^\circ$. Let $\beta_N^\circ$ denote the canonical $\cV$-decoration on $\mT_N^\circ$. On the complementary event $\cB_N^c$, define $(\mT_N^\circ,\beta_N^\circ)$ arbitrarily.
	
	By~\eqref{eq:catv} and the fact that $\Prb{\cB_N}\to1$,  we have
	\[
		(\mT_N,\beta_N)\atv(\mT_N^\circ,\beta_N^\circ).
	\]
	We may therefore assume that the two $\cV$-enriched trees are coupled so that
	\[
		\Prb{(\mT_N,\beta_N)=(\mT_N^\circ,\beta_N^\circ)}\to1.
	\]
	Let $\cE_N$ denote the event that this equality holds, and that $\cB_N$ holds, and that $S_{N-1}<0$. Proposition~\ref{pro:catree} gives
	\[
		\Prb{\cE_N} \to 1.
	\]

	On $\cE_N$, the appended increment $-1-S_{N-1}$ is the unique largest increment of $\cZ^{(N)}$ and corresponds to $v_N^*$. Since a \L{}ukasiewicz increment is equal to the corresponding outdegree minus one,
	\[
		\Delta_N=-S_{N-1}.
	\]
	By the definition of the Vervaat transform, the first $\Delta_N$ excursions following this increment are $S^{(1)},\ldots,S^{(\Delta_N)}$. Consequently, under the enriched coupling,
	\[
		(F_{N,i},\beta_N|_{F_{N,i}})=(\mT[i],\beta[i]), \qquad 1\le i\le\Delta_N.
	\]
	By the definition of $K_i$, it follows that
	\[
		\max_{x\in M_{N,i}}d_{\mM_n^\omega}(x,\varpi_N(x))\le K_i, \qquad 1\le i\le\Delta_N.
	\]
	Fix $\epsilon>0$. By a union bound,
	\begin{multline*}
		\Prb{\max_{1\le i\le\Delta_N}\max_{x\in M_{N,i}}d_{\mM_n^\omega}(x,\varpi_N(x))>\epsilon b_N^{1/4}}\\
		\le\Prb{\cE_N^c}+\Prb{\Delta_N>2b_N}+2b_N\Prb{K(\mT,\beta)>\epsilon b_N^{1/4}}.
	\end{multline*}
	By Lemma~\ref{le:caradius},
	\[
		b_N\Prb{K(\mT,\beta)>\epsilon b_N^{1/4}} \to 0.
	\]
	Hence, since $\Delta_N/b_N\convp1$ by Proposition~\ref{pro:catree},
	\begin{align}
		\label{eq:cachildradius}
		\max_{1\le i\le\Delta_N}\max_{x\in M_{N,i}}d_{\mM_n^\omega}(x,\varpi_N(x))=o_p(b_N^{1/4}).
	\end{align}
		
	We now control the root-side component by re-rooting. Conditionally on $\mM_n^\omega$, let $c_N$ be a uniformly selected corner, and let $\widehat{\mM}_n^\omega$ be the same map re-rooted at $c_N$. Since the weight of a map does not depend on its root corner, we have
	\[
		\widehat{\mM}_n^\omega\eqdist\mM_n^\omega.
	\]
	Let $\widehat R_{N,0}$ denote the maximal distance from a vertex of the root-side component of $\widehat{\mM}_n^\omega$ to its attachment point on its largest block. Then
	\[
		\widehat R_{N,0}\eqdist\max_{x\in M_{N,0}}d_{\mM_n^\omega}(x,\varpi_N(x)).
	\]
		
	A $\cV$-enriched tree with $m$ vertices encodes a map with $m-1$ corners. Consequently, $M_{N,0}$ has $W_{N,0}-1$ corners, while $\mV_N^*$ has $\Delta_N$ corners. On the event that the largest block is unique, if $c_N$ is a corner of neither $M_{N,0}$ nor $\mV_N^*$, then it belongs to a unique child-side component $M_{N,i}$, and this component becomes the root-side component of $\widehat{\mM}_n^\omega$. Therefore, by Proposition~\ref{pro:catree} and~\eqref{eq:cachildradius}
	\begin{align*}
		\Prb{\max_{x\in M_{N,0}}d_{\mM_n^\omega}(x,\varpi_N(x))>\epsilon b_N^{1/4}}
		&=\Prb{\widehat R_{N,0}>\epsilon b_N^{1/4}}\\
		&\le\Prb{\max_{1\le i\le\Delta_N}\max_{x\in M_{N,i}}d_{\mM_n^\omega}(x,\varpi_N(x))>\epsilon b_N^{1/4}} + o(1) \\
		&= o(1).
	\end{align*}
	Combining this with~\eqref{eq:cachildradius} gives
	\[
		R_N=o_p(b_N^{1/4}).
	\]
\end{proof}

\begin{corollary}
	We have
	\[
	\left(\mM_n^\omega,\left(\frac{3}{8b_n}\right)^{1/4}d_{\mM_n^\omega},\mu_{\mM_n^\omega}\right)\convd(\mathbf{M},d_{\mathbf{M}},\mu_{\mathbf{M}}).
	\]
\end{corollary}
\begin{proof}
	Since $\ed(\mV_N^*)=\Delta_N/2$ and $\Delta_N/b_N\convp1$, Theorem~\ref{te:main1} yields
	\[
		\left(\mV_N^*,\left(\frac{3}{4b_N}\right)^{1/4}d_{\mV_N^*},\mu_{\mV_N^*}\right)\convd(\mathbf{M},d_{\mathbf{M}},\mu_{\mathbf{M}}).
	\]
	By Lemma~\ref{le:cadist} we have
	\begin{align*}
		d_{\mathrm{H}}\left(\left(\mM_n^\omega,\left(\frac{3}{4b_N}\right)^{1/4}d_{\mM_n^\omega}\right),\left(\mV_N^*,\left(\frac{3}{4b_N}\right)^{1/4}d_{\mV_N^*}\right)\right)\convp0.
	\end{align*}
	Proposition~\ref{pro:catree} gives
	\[
		\frac{\left(\sum_{i=1}^{\Delta_N}W_{N,i}^2\right)^{1/2}}{\sum_{i=1}^{\Delta_N}W_{N,i}}\le\left(\frac{\max_{1\le i\le\Delta_N}W_{N,i}}{N-W_{N,0}}\right)^{1/2}\convp0.
	\]
	Hence, by fully analogous arguments as for Lemma~\ref{le:2corepart2}, it follows that 
	\begin{align*}
		d_{\mathrm{GHP}}\left(\left(\mM_n^\omega,\left(\frac{3}{4b_N}\right)^{1/4}d_{\mM_n^\omega},\mu_{\mM_n^\omega}\right),\left(\mV_N^*,\left(\frac{3}{4b_N}\right)^{1/4}d_{\mV_N^*},\mu_{\mV_N^*}\right)\right)\convp0.
	\end{align*}
	We have $N \sim 2n$ and hence $b_N \sim 2 b_n$. Hence the proof is complete.
\end{proof}

\subsection{The full phase diagram}

The proof of Theorem~\ref{te:main2} is immediate from the corollaries at the end of the preceding subsections which cover the individual regimes. Corollary~\ref{co:ma} is a consequence of Theorem~\ref{te:main2}:

\begin{proof}[Proof of Corollary~\ref{co:ma}]
	The scaling limits themselves follow directly from Theorem~\ref{te:main2}.  The expressions of the scaling constants are immediate since
	\[
		\Phi_u(z) = 1 + u(z^2 + \cV^*(z^2))
	\]
	yields for $\nu_u \ge 1$
	\[
		\tau_u = r_u \sqrt{1-r_u}
	\]
	and
	\[
		\Prb{\widehat{\xi}_u=2m} = r_u \frac{(3m-3)!}{(m-1)!(2m-1)!} (r_u^2(1-r_u))^{m-1}, \qquad m \ge 1
	\]
	and for $\nu_u > 1$
	\[
		\Va{\xi_u} = \frac{\tau_u^2\Phi_u''(\tau_u)} {\Phi_u(\tau_u)} = \frac{r_u}{3r_u-2}.
	\]
\end{proof}

\bibliographystyle{abbrv}
\bibliography{blank}

\end{document}